\documentclass[12pt]{article}
\usepackage{hyperref}
\usepackage{diagbox}
\usepackage{amssymb}
\usepackage{amsmath,bm}
\usepackage{multirow}
\usepackage{makecell}
\usepackage{graphicx}
\usepackage{subfigure}
\usepackage{float}
\usepackage[titletoc, title]{appendix}
\usepackage{stmaryrd}
\usepackage{lipsum}
\usepackage{amsfonts}
\usepackage{graphicx}
\usepackage{epstopdf}
\usepackage{algorithmic}
\usepackage{amsopn}
\usepackage[nameinlink]{cleveref}
\usepackage{amssymb,amsmath,amscd}
\usepackage[margin=1in]{geometry}
\usepackage{color}

\newtheorem{assumption}{Assumption}

\newtheorem{remark}{Remark}
\newtheorem{theorem}{Theorem}
\newtheorem{lemma}{Lemma}
\newtheorem{proof}{Proof}

\begin{document}

	\title{Locking-free HDG methods for Reissner–Mindlin plates equations on polygonal meshes}
\date{}
\author{
	Gang Chen\thanks{School of Mathematics, Sichuan University, Chengdu, China,  
		email:cglwdm@scu.edu.cn}, 
	Lu Zhang\thanks{School of Mathematics, Sichuan University, Chengdu, China, 
		email:2021222010005@stu.scu.edu.cn}, 
	Shangyou Zhang\thanks{Department of Mathematics Science, University of Delaware, Newark, USA, email:szhang@udel.edu}
	}
\maketitle
% REQUIRED
\begin{abstract} 
	We present and analyze a new hybridizable discontinuous Galerkin method (HDG) for the Reissner–Mindlin plate bending system. Our method is based on the formulation utilizing Helmholtz Decomposition. Then the system is decomposed into three problems: two trivial Poisson problems and a perturbed saddle-point problem. We apply HDG scheme for these three problems fully. This scheme yields the optimal convergence rate ($(k+1)$th order in the $\mathrm{L}^2$ norm) which is uniform with respect to plate thickness (locking-free) on general meshes.  We further analyze the matrix properties and precondition the new finite element system. Numerical experiments are presented to confirm our theoretical analysis.

keywords: Reissner–Mindlin plate, hybridizable discontinuous Galerkin method, error analysis, locking-free 
\end{abstract}

	\section{Introduction}
In this paper we present a hybridizable discontinuous Glaerkin method (HDG) for Reissner–Mindlin (RM) plate bending system. The equations of RM can be written as the following equations:
\begin{subequations}
	\begin{align}
		&&	-\nabla\cdot\mathcal C\bm{\epsilon}(\bm\theta)
		-\lambda t^{-2}(\nabla \omega-\bm\theta)&=\bm f,&\text{in }\Omega, \\
		&&	-\lambda t^{-2}\nabla\cdot(\nabla \omega-\bm \theta)&=g,&\text{in }\Omega,
	\end{align}
	with the hard clamped boundary condition
	\begin{equation}
		\bm\theta=\bm 0,\quad \omega=0, \qquad \text{on }\partial\Omega.
	\end{equation}
\end{subequations}
Here $t$ denotes the plate thickness, the unknowns $w$ and $\bm\theta$ denote the transverse displacement of the midplane and the rotation of the fibers normal to it. The functions $\bm f\in[\mathrm{L}^2(\Omega)]^2$ and $g\in\mathrm{L}^2(\partial\Omega)$ denote a body force and the transverse loading acting on a polygonal region $\Omega\in\mathbb{R}^2$  with boundary $\partial \Omega$, respectively. The symmetric gradient $\bm{\epsilon}(\bm\phi):= \frac{1}{2}(\nabla \bm\phi +(\nabla\bm\phi)^T)$ which takes values in the space $\mathbb{S}=\mathbb{R}_{\rm sym}^{2\times2}$. The operator $\mathcal C: \mathbb{S}\rightarrow\mathbb{S}$ is the positive definite plane stress constitutive tensor as follows
\begin{align}\nonumber
	\mathcal C\bm\tau =\frac{E}{12(1-\nu^2)}
	[(1-\nu)\bm\tau+\nu{\rm tr}(\bm\tau)\bm I],
\end{align}
%We subsequently get 
%\begin{align}\nonumber
	%	\mathcal C^{-1}\bm\tau =\frac{12(1+\nu)}{E}-\frac{12\nu}{E}\text{tr}(\bm \tau)\bm I,
	%\end{align}
with $\lambda=\frac{\kappa E}{2(1+\nu)}$ and $\bm\tau \in \mathbb{S}$, the Young’s modulus $E>0$, the Poisson's ratio $\nu\in(0,\frac{1}{2}]$ and the shear correction factor $\kappa$ is usually chosen as $\frac 5 6$. More details about RM can be found in \cite{Falk2008}. A simple calculation shows that
\begin{align}\label{key}
	\frac{12(1-\nu)}{E}\|\bm \tau\|_0^2\leq(\mathcal{C}^{-1}\bm\tau,\bm\tau)\leq\frac{12(1+\nu)}{E}\|\bm \tau\|_0^2,
\end{align}
holds true for any $\bm\tau\in\mathbb{S}$.\par
Standard low order conforming finite element methods are known to exhibit shear locking and yield poor results for small thickness $t\rightarrow 0$ (see \cite{Arnold1981}). For dealing with shear locking phenomenon, it is useful to introduce shear stress $\gamma:= \lambda t^{-2}(\nabla \omega-\theta)$ with corresponding finite dimensional space $\bm\Gamma_h$. Then many successful conforming standard finite element methods introduce interpolation operator $\bm\Pi: \bm H_0^1(\Omega)\rightarrow \bm\Gamma_h$ to make $\nabla \omega-\bm\Pi\bm\theta\rightarrow 0$  as $t\rightarrow 0$, for triangular elements, see \cite{Duran,Brezzi1,Brezzi2,Brezzi3,Bathe,Falk,Zhongnian,yu2017analysis,gallistl2021taylor,HANSBO2011638}, for rectangular elements, see \cite{Bathe2, Brezzi1,Stenberg,Hansbo_2014} and for quadrilateral elements, see \cite{arnold2002approximation,CARSTENSEN20111161}. Also, many researchers develop a number of other methods. A nonconforming finite element for $\omega$ is given by equivalent formulation of RM using the discrete Helmholtz decomposition and the MINI element for stokes problem, see \cite{duran1991inf,arnold1989uniformly}. Low-order nonconforming elements, both for $\bm\theta$ and $\omega$, obtain the desire stability by adding some inter-element term, see \cite{chinosi2006nonconforming,lovadina2005low} for details. Jun Hu and Zhong-Ci Shi generalize the rectangular nonconforming Wilson element method dropping the requirement of $H_3(\Omega)$ regularity on the solution, see \cite{HU2008464}. Stenberg \cite{stenberg1995new} and Hughes \cite{hughes1988mixed} applied least-squares stabilization schemes. J. Kiendl et al. \cite{KIENDL2015489} develop two collocation schemes based on the equilibrium equations of the Reissner–Mindlin plate in a primal formulation. G. Kikis et al. \cite{KIKIS2021113490} implement the anisotropic phase-field model of brittle fracture to the isogeometric Reissner–Mindlin shell formulation.\par
A lot of literature show that the numerical trace or numerical flux is crucial to obtain stable numerical solution. So the discontinuous Galerkin method (DG) appears naturally in \cite{doi:10.1137/0710071}, called the interior penalty at first and has been applied in solving various problems, see\cite{doi:10.1137/S0036142901384162} for elliptic problems and \cite{ARNOLD20073660,Arnold2005AFO} for RM problems. DG is not required to satisfy the standard inter-element continuity condition like conforming finite element method, but usually adds the penalty term of piecewise numerical flux. So DG applies flexibly on general polyonal meshes. However, based on above merit, DG needs more computability and memory with adding the extra degrees of freedom especially for high-degree polynomial approximations. Under the framework of DG, many other methods for RM have developed. Daniele and Pietro \cite{DIPIETRO2022136} use an enrichment of the two-dimensional discrete de Rham space with edge unknowns for discretization recently. F\"{u}hrer and Heuer \cite{10.48550/1906.04869, doi:10.1137/22M1498838} apply discontinuous Petrov–Galerkin method for RM with numerical flux and the numerical trace as unknowns, which is the main difference from HDG with only numerical trace being unknown.\par
The method presented in this paper namely HDG is based on local discontinuous Galerkin method (LDG), see \cite{doi:10.1137/S0036142900371003}. Therein LDG firstly introduce an auxiliary variable to convert a high-order system into a system of first-order equations, then defining the numerical trace for all unknowns. The important point is that the auxiliary variable can be solvable locally because the prescribed numerical trace of primal unknown does not depend on the auxiliary variable. In this sense, LDG decreases the computational expense compared with DG (assuming local system could be paralyzingly solved) and retains the flexibility. Soon combining with other continuous hybridized methods, HDG was introduced by\cite{doi:10.1137/070706616}, directly called LDG-H therein. Corollary 3.2 of \cite{doi:10.1137/070706616} pointed out HDG is not identical with LDG because the numerical trace, as a new introduced unknown, depends on the numerical flux. The new numerical trace leads to a small and better system which involves only numerical trace as unknowns. This is much more efficient for high-degree polynomials, see \cite{10.1007/s10915-011-9501-7}. Over the past decade, the application of HDG extends more and more widely, in particular for RM in \cite{Chen1}. \par
In this paper, we continue to extend the work of HDG to RM plates problems on polygonal meshes by utilizing a Helmholtz decomposition. Then the system is decomposed into three problems: two trivial Poisson problems and a perturbed saddle-point problem. The main analysis is in the perturbed saddle-point problem. The method retains the flexibility of DG and we can analyze the error  locally. We show the new HDG method is locking-free and produces optimal-order solutions. \par
The paper is organized as follows. Section 2 restates some necessary notations and preliminary results. Sections 3 and 4 include our method and the main result. In these two sections, we fully carry out the analysis and get an optimal convergence rate for the norm defined by discrete formulation of our method, and in $\mathrm{L}^2$ norm, for arbitrary polynomial degree (greater than zero). Section 5 analyzes the properties of matrix from our method. The main improvement from \cite{Chen1} is embodied here. We end in Section 6 by presenting several numerical experiments, verifying our theoretical results.\par
\section{Notation and Preliminary Results}
\subsection{Notation}
To introduce the fully discrete HDG formulation for the RM plates model, we firstly fix some notation. Let $\Omega,T$ be the bounded simply connected Lipschitz domain in $\mathbb{R}^s$ where $s=1,2$. We shall use the usual Sobolev spaces such as $H^m(T)$ and $H^m_0(T)$ to denote the 
$m$-th order Sobolev spaces on $T$, and use $\|\cdot\|_{m,T}, |\cdot|_{m,T}$ to denote the norm and semi-norm on these spaces. We use $(\cdot,\cdot)$ to denote inner product on $H^m(T)$. When $T=\Omega$, for abbreviation, we denote $\|\cdot\|_{m}:=\|\cdot\|_{m,T}$ and $|\cdot|_{m}:=|\cdot|_{m,T}$.
In particular, when $s=1$ (in 1D), we use $\langle\cdot,\cdot\rangle$ to replace $(\cdot,\cdot)$. By convention, we use
boldface type for vector (or tensor) analogues of the spaces along with the vector-valued functions. For an integer $k\geq 0$, $P_k(T)$ denotes the set of all polynomials defined on $T$ with degree at most $K$. And for clarity, we reintroduce some operators $\nabla,\nabla^{\perp}$ with scalar variable $\omega$
$$\nabla\omega=(\partial_x\omega,\partial_y\omega),\qquad \nabla^{\perp}\omega=(\partial_y\omega,-\partial_x\omega),$$
and $\nabla\cdot, \nabla\times$ with vector variable $\bm \phi=(\phi_1,\phi_2)$
$$\nabla\cdot\bm\phi=\partial_x\phi_1+\partial_y\phi_2,\qquad\nabla\times\bm\phi=\partial_x\phi_2-\partial_y\phi_1.$$\par
Let $\mathcal{T}_h=\cup{K}$ be a shape regular partition of the domain $\Omega$ (in sense of the Lemma 2.1 of \cite{Duran2008}) with respect to a union of a finite (and uniformly bounded) number of star-shaped elements. For any $K \in \mathcal{T}_h$, we let $h_K$ be the minimum diameters of circles containing $K$ and denote the mesh size $h:=\max_{K \in \mathcal{T}_h}h_K$. The element $K$ is star-shaped if and only if $\exists \bm x_0 \in K, \forall \bm x \in K$, the line segment from $\bm x$ to $\bm x_0$ lies in $K$. Let $\partial\mathcal{T}_h=\cup{F}$ be the union of all edges $F$ of $K\in \mathcal{T}_h$. we denote $h_F$ the length of edge F and the edge $F$ of $K$ satisfies that $h_F$ is greater than or equal to a uniformly constant multiply $h_K$. For all $K \in \mathcal{T}_h$, we denote by $\bm n$ the unit outward normal vectors of $K$. Furthermore, we introduce the discrete inner products
$$(\mu,v)_{\mathcal{T}_h}:=\sum_{K\in\mathcal{T}_h}(\mu,v)_{K} =\sum_{K\in\mathcal{T}_h}\int_K \mu v\ \mathrm{d}\bm x,\quad\langle\zeta,\rho\rangle_{\partial\mathcal{T}_h}:=\sum_{K\in\mathcal{T}_h}\langle\zeta,\rho\rangle_{\partial K} =\sum_{K\in\partial\mathcal{T}_h}\int_K \zeta \rho\ \mathrm{d}\bm s.$$
%	For any finite dimensional space, let $\nabla$, $\nabla^{\perp}$, $\bm \epsilon$, $\nabla \cdot$ and $\nabla\times$ denote the respective gradient, the two-dimensional curl, the symmetric part of gradient, the divergence 
%	and the circulation on each element $K\in\mathcal{T}_h$ respectively.\par

For any element $K\in\mathcal{T}_h$, $F \in \partial\mathcal{T}_h$ and any non-negative integer $j$. Let $ \Pi_j^o: \mathrm{L}^2(K)\rightarrow \mathcal{P}_j(K)$ and $ \Pi_j^{\partial}: \mathrm{L}^2(F)\rightarrow \mathcal{P}_j(F)$ be the usual $\mathrm{L}^2$ projection corresponding to two-dimensional projection $\bm\Pi_j^o: \bm{\mathrm{L}}^2(K)\rightarrow [\mathcal{P}_j(K)]^2$ and $\bm\Pi_j^{\partial}: \bm{\mathrm{L}}^2(F)\rightarrow [\mathcal{P}_j(F)]^2$. Then we can get the following approximation and boundness result \eqref{L2projectionapproximation}, therein \eqref{1} and \eqref{2} are direct results of ${\rm L}^2$ projection definition, \eqref{3} holds by \eqref{6} which holds by Theorem 3.22 of \cite{200359} and Corollary 2.1.of \cite{Duran2008}, \eqref{5} holds by Lemma 3.3 of \cite{CHEN2018643} and \eqref{4} holds by Lemma 2.3 of \cite{Duran2008}.\par
It's important to note that we use same notation $C$ throughout this paper to denote different positive constants at every occurrence which is independent of $t,h$.
\begin{lemma}
	Let $m$ be an integer with $1\leq m \leq j+1$, For all $K \in \mathcal{T}_h$, $F$is arbitrary edge of $K$, it holds
	\begin{subequations}\label{L2projectionapproximation}
		\begin{align}
			\label{1}\|\Pi_j^o v\|_{0,K}&\leq \|v\|_{0,K}  \qquad &\forall v \in \mathrm{L}^2(K), \\
			\label{2}\|\Pi_j^{\partial} v\|_{F}&\leq \|v\|_F  \qquad &\forall v \in \mathrm{L}^2(F), \\
			\label{3}\|v-\Pi_j^{\partial}v\|_{\partial K}&\leq C h^{m-1/2}_{K}|v|_{m,K}   &\qquad \forall v \in H^m(K), \\
			\label{4}|v|_{1,K}&\leq C h^{-1}_{K}|v|_{0,K}  \qquad &\forall v \in \mathcal P_j(K), \\
			\label{5}\|\nabla^s(v-\Pi_j^{o}v)\|_{K}&\leq C h^{m-s}_{K}|v|_{m,K}   \qquad& \forall v \in H^m(K), 1\leq s+1\leq m, \\
			\label{6}|v|_{\partial K}&\leq C h^{-1/2}_{K}|v|_{0,K}   \qquad &\forall v \in H^m(K).
		\end{align}
	\end{subequations}
\end{lemma}
\subsection{Helmholtz Decomposition}
Let us introduce a new unknown, so-called the shear stress. It is a useful tool for many methods:
$$
\bm\gamma:= \lambda t^{-2}(\nabla \omega-\bm\theta).
$$
Using Helmholtz decomposition, see \cite{Falk2008}, wherein $\widehat{H}^m(\Omega):=H^m(\Omega)/\mathbb{R}$, $$\bm {\mathrm{L}}^2(\Omega):=[\mathrm{L}^2(\Omega)]^2=\nabla H_0^1(\Omega) \oplus \nabla^\perp\widehat{H}^1(\Omega), $$
and  suppose $\bm\gamma$ is decomposed as $\bm\gamma=\nabla r+\nabla^\perp p$, we get the following weak formulation:\\
find $(r, \boldsymbol{\theta}, \boldsymbol{\sigma}, p, \omega) \in H_0^1(\Omega) \times\left[H_0^1(\Omega)\right]^2 \times\left(\left[\mathrm{L}^2(\Omega)\right]^{2 \times 2} \cap \mathbb{S}\right) \times\widehat{H}^1(\Omega) \times H_0^1(\Omega)$ such that
\begin{align*}
	(\nabla r, \nabla \mu) &=(g, \mu), &\forall \mu \in H_0^1(\Omega), \\
	(\mathcal{C}^{-1} \boldsymbol{\sigma}, \boldsymbol{\tau})-(\nabla \boldsymbol{\theta}, \boldsymbol{\tau}) &=0,&\forall \boldsymbol{\tau} \in\left[\mathrm{L}^2(\Omega)\right]^{2 \times 2} \cap \mathbb{S}, \\
	(\boldsymbol{\sigma}, \nabla \phi)-(\boldsymbol{\phi}, \nabla^{\perp} p) &=(\nabla r, \boldsymbol{\phi})+(\boldsymbol{f},\boldsymbol{\phi}), &\forall \boldsymbol{\phi} \in\left[H_0^1(\Omega)\right]^2, \\
	-(\boldsymbol{\theta},\nabla^{\perp} q)-\lambda^{-1} t^2(\nabla^{\perp} p,\nabla^{\perp} q) &=0,&\forall q \in H^1(\Omega) ,\\
	(\nabla \omega,\nabla s) &=(\boldsymbol{\theta},\nabla s)+\lambda^{-1}t^2(g,s),&\forall s \in H_0^1(\Omega) . 
\end{align*}\par
%\begin{remark}
	%	$\bm\gamma$ also can be decomposed $\bm\gamma=\nabla r+\nabla^\perp p$, with $r\in H_0^1(\Omega)$ and $p\in H_0^1(\Omega)$. Because we determine $p$ by $\nabla^{\perp}p=\bm\gamma-\nabla r$ and $|p|_1$ is a norm in $H_0^1(\Omega)$.
	%\end{remark}
\subsection{Continuous Form}
We introduce $\bm L=\nabla r$ and $\bm G=\nabla \omega$ and $\bm R=\lambda^{-1}t^2\nabla^{\perp} p$, then we rewrite the above formulation as
\begin{subequations}\label{continuous form}
	\begin{align}
		(\bm L,\bm M)-(\nabla r,\bm M)&=0,&\forall \bm M \in \bm{\mathrm{L}}^2(\Omega), \\
		(\bm L, \nabla \mu) &=(g, \mu), &\forall \mu \in H_0^1(\Omega), \\
		(\mathcal{C}^{-1} \boldsymbol{\sigma}, \boldsymbol{\tau})-(\nabla \boldsymbol{\theta}, \boldsymbol{\tau}) &=0, &\forall \boldsymbol{\tau} \in\left[\mathrm{L}^2(\Omega)\right]^{2 \times 2} \cap \mathbb{S}, \\
		(\lambda t^{-2} \bm R,\bm S)
		-(\nabla^{\perp}p,\bm S)&=0,&\forall \bm S \in \bm{\mathrm{L}}^2(\Omega), \\
		(\boldsymbol{\sigma}, \nabla \bm\phi)-(\boldsymbol{\phi}, \nabla^{\perp} p) &=(\bm L, \boldsymbol{\phi})+(\boldsymbol{f}, \boldsymbol{\phi}), &\forall \boldsymbol{\phi} \in\left[H_0^1(\Omega)\right]^2, \\
		-(\boldsymbol{\theta}, \nabla^{\perp} q)-(\bm R, \nabla^{\perp} q) &=0, &\forall q \in \widehat{H}^1(\Omega), \\
		(\bm G,\bm H)-( \nabla \omega,\bm H)&=0,&\forall \bm H \in \bm{\mathrm{L}}^2(\Omega), \\
		(\bm G, \nabla s) &=(\boldsymbol{\theta}, \nabla s)+\lambda^{-1}t^2(g,s), &\forall s \in H_0^1(\Omega), \\
		\bm\gamma&=\bm L+\lambda t^{-2}\bm R.
	\end{align}
\end{subequations}\par
For abbreviation, we introduce the following continuous bilinear forms:
\begin{align*}
	\mathfrak{a}(\bm L,r;\bm M,\mu)&:=(\bm L,\bm M)-(\nabla r,\bm M)+(\bm L, \nabla \mu), \\
	\mathfrak{b}(\bm \sigma,\bm R,\theta,p;\bm \tau,\bm S,\bm \phi,q)&:=\left(\mathcal{C}^{-1} \boldsymbol{\sigma}, \boldsymbol{\tau}\right)-(\nabla \boldsymbol{\theta}, \boldsymbol{\tau}) 
	+(\lambda t^{-2}\bm R,\bm S)+(\boldsymbol{\sigma}, \nabla \phi)\\
	&\qquad-(\nabla^{\perp}p,\bm S)-(\boldsymbol{\phi}, \nabla^{\perp} p)
	+(\nabla^{\perp}q,\bm R)+(\boldsymbol{\theta}, \nabla^{\perp} q). %, \\
	%	\mathfrak{a}(\bm G,\omega;\bm H,s)&:= (\bm G,\bm H)-( \nabla \omega,\bm H)+(\bm G, \nabla s).
\end{align*}
So \eqref{continuous form} can be reformulated as follows:
\begin{subequations}\label{continuousfor}      
	\begin{align} 
		\label{stage1} 
		\mathfrak{a}(\bm L,r;\bm M,\mu)&=(g, \mu), \qquad \forall (\bm M,\mu ) \in\bm{\mathrm{L}}^2(\Omega) \times H_0^1(\Omega), \\
		\label{stage2} 
		\mathfrak{b}(\bm \sigma,\bm R,\bm\theta,p;\bm \tau,\bm S,\bm \phi,q)&=(\bm L, \boldsymbol{\phi})+(\boldsymbol{f}, \boldsymbol{\phi}), \\
		\forall (\bm \tau, \bm S,\boldsymbol{\phi},q) & \in (\left[\mathrm{L}^2(\Omega)\right]^{2 \times 2}\cap \mathbb{S})\times \bm{\mathrm{L}}^2(\Omega) \times\left[H_0^1(\Omega)\right]^2\times\widehat{H}^1(\Omega),\nonumber\\ 
		\label{stage3}
		\mathfrak{a}(\bm G,\omega;\bm H,s)&= \left(\boldsymbol{\theta}, \nabla s\right)+\lambda^{-1}t^2(g,s), \qquad 
		\forall (\bm H,s)\in \bm{\mathrm{L}}^2(\Omega) \times H_0^1(\Omega), \\
		\bm\gamma &=\bm L+\lambda t^{-2}\bm R.
	\end{align}
\end{subequations}
\subsection{Well-posedness and Regularity Result} In
\cite{Falk2008}, the authors established such estimates below for the case of clamped plate, which will be used later in our proof.
\begin{theorem}\label{regularity result}
	Let $\Omega$ be a convex polygon or a smoothly bounded domain in the plane. For any $t\in(0,1]$, $\bm f\in [H^{-1}(\Omega)]^2$, and $g \in H^{-1}(\Omega)$, there exits a unique solution $(r,\bm\sigma,\bm\theta,p,\omega)\in H_0^1(\Omega)\times(\left[\mathrm{L}^2(\Omega)\right]^{2 \times 2}\cap \mathbb{S})\times\left[H_0^1(\Omega)\right]^2\times\widehat{H}^1(\Omega)\times H_0^1(\Omega) $ satisfying \eqref{stage1}, \eqref{stage2} and \eqref{stage3} and there exits a constant $C$ independent of $t,\bm f$ and $g$ such that
	\begin{equation}
		\begin{aligned}
			\|r\|_{2}+\|\bm \sigma\|_{1}+\|\bm\theta\|_{2}+\|p\|_1+t\|p\|_{2}+\|\omega\|_{2}\leq C(\|\bm f\|_0+\|g\|_{0}).
		\end{aligned}
	\end{equation}
\end{theorem}\par
In what follows we suppose that the following regularity result holds:
\begin{assumption}\label{assumption of regularity}
	Under the condition of \Cref{regularity result} and assume $\bm f\in [H^{k-1}(\Omega)]^2$, and $g \in H^{k-1}(\Omega)$, $k\geq1$, the solution of \eqref{stage1}, \eqref{stage2} and \eqref{stage3} holds
	\begin{equation}
		\begin{aligned}
			\|r\|_{k+1}+\|\bm \sigma\|_{k}+\|\bm\theta\|_{k+1}+\|p\|_k+t\|p\|_{k+1}+\|\omega\|_{k+1}\leq C(\|\bm f\|_{k-1}+\|g\|_{k-1}).
		\end{aligned}
	\end{equation}
\end{assumption}
\section{Our Method}
\label{sec:main}
\subsection{HDG Formulation}
For any integer $k\geq1,\max{(1,k-1)}\leq \ell\leq k$, we introduce the following finite dimensional function spaces:
\begin{align*}
	\mathcal{L}_h&: = \left \{\bm M_h\in\mathbf{L} ^2(\Omega ): \bm M_h|_K\in [\mathcal{P}^{k-1}(K)]^2,\forall K\in \mathcal{T}_h\right \}, \\
	\mathcal{R}_h&: = \left \{ \mathbf{\mu}_h\in \mathrm{L}^2(\Omega): \mathbf{\mu}_h|_K\in \mathcal{P}^{k}(K),\forall K\in \mathcal{T}_h\right \}, \\
	\widehat{\mathcal{R}}_h&:= \left \{ \widehat \mu_h\in \mathrm{L}^2(\mathcal{F}_h ): \widehat \mu_h|_F\in \mathcal{P}^{k-1}(F),\forall F\in \mathcal{F}_h, \widehat \mu_h|_{\partial\Omega}=0 \right \}, \\
	\Sigma_h&: = \left \{ \tau_h\in\left[\mathrm{L}^2(\Omega)\right]^{2 \times 2}: \tau_h|_K\in [\mathcal{P}^{k-1}(K)]^{2\times 2}\cap \mathbb{S},\forall K\in \mathcal{T}_h\right \}, \\
	\mathcal{Y}_h&:=\left \{ \bm \phi_h\in\mathbf{L}^{2}(\Omega ): \bm \phi_h|_K\in [\mathcal{P}^k(K)]^{2},\forall K\in \mathcal{T}_h\right \}, \\
	\widehat{\mathcal{Y}}_h&:=\left \{ \bm{\widehat \phi} _h\in \mathbf{L}^2(\mathcal{F} _h):\bm{\widehat \phi} _h|_F\in[ \mathcal{P}^{\ell}(F)]^2,\forall F\in \mathcal{F}_h, \bm{\widehat \phi}|_{\partial\Omega}=\bm 0 \right \}, \\
	\mathcal{S}_h&:=\left\{\bm S_h\in \mathbf{L}^2(\Omega):\bm S_h|_K \in [\mathcal{P}^{k-1}(K)]^2,\forall K \in \mathcal{T}_h\right\}, \\
	\mathcal{P}_h&:=\left \{ q_h\in\mathrm{L}_0^2(\Omega ): q_h|_K\in \mathcal{P}^k(K),\forall K\in \mathcal{T}_h\right \}, \\
	\widehat{\mathcal{P}}_h&:= \left\{\widehat q_h\in \mathrm{L}^2(\mathcal{F}_h ): \widehat q_h|_F\in \mathcal{P}^{k-1}(F),\forall F\in \mathcal{F}_h \right \}, \\
	\mathcal{G}_h&: = \left \{ \bm H_h\in\mathbf{L} ^2(\Omega ): \bm H_h|_K\in [\mathcal{P}^{k-1}(K)]^2,\forall K\in \mathcal{T}_h\right \} ,\\
	\mathcal{W}_h&: = \left \{ s_h\in \mathrm{L}^2(\Omega): s_h|_K\in \mathcal{P}^{k}(K),\forall K\in \mathcal{T}_h\right \}, \\
	\widehat{\mathcal{W}}_h&:= \left\{\widehat s_h\in \mathrm{L}^2(\mathcal{F}_h ): \widehat s_h|_F\in \mathcal{P}^{k-1}(F),\forall F\in \mathcal{F}_h,\widehat s_h|_{\partial\Omega}=0 \right \}.
\end{align*}\par
Our numerical schemes reads as follows.\par
\noindent\textbf{Step One:} find $(\bm L_h,r_h,\widehat r_h) \in \mathcal{L}_h \times \mathcal{R}_h \times \widehat{\mathcal{R}}_h$  {\color{black} such that}
\begin{subequations}\label{discrete formulation}
	\begin{align}
		(\bm L_h,\bm M_h)_{\mathcal T_h}+(r_h,\nabla\cdot\bm M_h)_{\mathcal T_h}
		-\langle \widehat r_h,\bm M_h\cdot\bm n\rangle_{\partial\mathcal T_h}&=0, \\
		-(\nabla\cdot\bm L_h,\mu_h)_{\mathcal T_h}
		+\langle \bm L_h\cdot\bm n,\widehat \mu_h\rangle_{\partial\mathcal T_h}
		\qquad\qquad\qquad& 
		\\ \nonumber +\langle
		\alpha_1(\Pi_{k-1}^{\partial}r_h-\widehat r_h),
		\Pi_{k-1}^{\partial}\mu_h-\widehat \mu_h
		\rangle_{\partial\mathcal T_h}
		&=(g,\mu_h)_{\mathcal T_h}.
	\end{align}
	\textbf{Step Two:} find $(\bm\sigma_h,\bm R_h,\bm\theta_h,\widehat{\bm\theta}_h,p_h,\widehat p_h) \in \Sigma_h \times \mathcal{S}_h\times \mathcal{Y}_h \times \widehat{\mathcal{Y}}_h\times\mathcal{P}_h\times\widehat{\mathcal{P}}_h$ {\color{black} such that}
	\begin{align}
		(\mathcal C^{-1}\bm\sigma_h,\bm\tau_h)_{\mathcal T_h}+(\bm\theta_h,\nabla\cdot\bm\tau)_{\mathcal T_h}
		-\langle 
		\widehat{\bm\theta}_h,\bm\tau_h\bm n
		\rangle_{\partial\mathcal T_h}&=0, \\
		(\lambda t^{-2} \bm R_h,\bm S_h)_{\mathcal T_h}
		+(p_h,\nabla\times\bm S_h)_{\mathcal T_h}
		-\langle\widehat p_h,\bm S_h\cdot\bm t\rangle_{\partial\mathcal T_h}&=0, \\
		-(\nabla\cdot\bm \sigma_h,\bm\phi_h)_{\mathcal T_h}
		+\langle
		\bm\sigma_h\bm n,\widehat{\bm\phi}_h
		\rangle_{\partial\mathcal T_h}
		+(\nabla\times\bm\phi_h,p_h)_{\mathcal T_h}
		-\langle 
		\bm\phi_h\cdot\bm t,\widehat p_h
		\rangle_{\partial\mathcal T_h}&\\
		+\langle 
		\alpha_2(\bm\Pi_{\ell}^{\partial}\bm{\theta}_h-\widehat {\bm\theta}_h),
		\bm\Pi_{\ell}^{\partial}\bm\phi_h-\widehat {\bm\phi}_h
		\rangle_{\partial\mathcal T_h}
		=(\bm L_h,\bm {\phi}_h)_{\mathcal T_h}
		+(\bm f,\bm \phi_h)_{\mathcal T_h},&\\
		-(\nabla\times\bm \theta_h,q_h)_{\mathcal T_h}
		+\langle 
		\bm\theta_h\cdot\bm t,\widehat q_h
		\rangle_{\partial\mathcal T_h}&\\	 
		-(\nabla\times\bm R_h,q_h)_{\mathcal T_h}
		+\langle\bm R_h\cdot\bm t,\widehat q_h\rangle_{\partial{\mathcal{T}_h}}
		&\\
		+\langle 
		\alpha_3(\Pi_{k-1}^{\partial}p_h-\widehat p_h),
		\Pi_{k-1}^{\partial}q_h-\widehat q_h
		\rangle_{\partial{\mathcal{T}_h}}&=0.
	\end{align}
	\textbf{Step Three:} find $(\bm G_h,\omega_h,\widehat \omega_h)\in \mathcal{G}_h\times \mathcal{W}_h \times \widehat{\mathcal{W}}_h $ {\color{black} such that}
	\begin{align}
		(\bm G_h,\bm H_h)_{\mathcal T_h}+(\omega_h,\nabla\cdot\bm H_h)_{\mathcal T_h}
		-\langle \widehat \omega_h,\bm H_h\cdot\bm n\rangle_{\partial\mathcal T_h}&=0, \\
		-(\nabla\cdot\bm G_h,s_h)_{\mathcal T_h}
		+\langle \bm G_h\cdot\bm n,\widehat s_h\rangle_{\partial\mathcal T_h}
		+\langle\alpha_1(\Pi_{k-1}^{\partial}\omega_h-\widehat \omega_h), &
		\Pi_{k-1}^{\partial}s_h-\widehat s_h
		\rangle_{\partial\mathcal T_h} \\ 
		=\lambda^{-1}t^2(g,s_h)+\langle \bm \theta_h\cdot\bm n,\widehat s_h\rangle_{\partial\mathcal T_h} & -(\nabla\cdot\bm \theta_h, s_h)_{\mathcal T_h}.
	\end{align}
	\textbf{Step Four:} we simply set
	\begin{align}
		\bm\gamma_h=\bm L_h+\lambda t^{-2}\bm R_h.
	\end{align}
\end{subequations}\par
Likewise, we introduce discrete inner products and bi-linear forms as follows:
\begin{align*}
	&\mathfrak{a}_h(\bm L_h,r_h,\widehat{r}_h;\bm M_h,\mu_{h},\widehat{\mu}_h)\\
	&\qquad:=(\bm L_h,\bm M_h)_{\mathcal T_h}+(r_h,\nabla\cdot\bm M_h)_{\mathcal T_h}-\langle \widehat r_h,\bm M_h\cdot\bm n\rangle_{\partial\mathcal T_h}\\
	&\qquad\qquad-(\nabla\cdot\bm L_h,\mu_h)_{\mathcal T_h}+\langle \bm L_h\cdot\bm n,\widehat \mu_h\rangle_{\partial\mathcal T_h}\\
	&\qquad\qquad+\langle\alpha_1(\Pi_{k-1}^{\partial}r_h-\widehat r_h),\Pi_{k-1}^{\partial}\mu_h-\widehat \mu_h
	\rangle_{\partial\mathcal T_h}, \\
	%\end{align*}
%\begin{align*}
	& \mathfrak{b}_h(\bm \sigma_h,\bm R_h,\bm \theta_h,\widehat{\bm\theta}_h,p_h,\widehat p_h ;\bm \tau_h,\bm S_h,\bm \phi_h,\widehat{\phi}_h,q_h,\widehat q_h)\\
	&\qquad: =(\mathcal C^{-1}\bm\sigma_h,\bm\tau_h)_{\mathcal T_h}+(\bm\theta_h,\nabla\cdot\bm\tau)_{\mathcal T_h}-\langle \widehat{\bm\theta}_h,\bm\tau_h\bm n\rangle_{\partial\mathcal T_h}\\
	&\qquad\qquad+(\lambda t^{-2} \bm R_h,\bm S_h)_{\mathcal T_h}-(\bm\phi_h,\nabla\cdot\bm\sigma_h)_{\mathcal T_h}+\langle\widehat{\bm\phi}_h,\bm\sigma_h\bm n\rangle_{\partial\mathcal T_h}\\
	&\qquad\qquad+\langle\alpha_2(\bm\Pi_{\ell}^{\partial}\bm{\theta}_h-\widehat {\bm\theta}_h),\bm\Pi_{\ell}^{\partial}\bm\phi_h-\widehat {\bm\phi}_h\rangle_{\partial\mathcal T_h}\\
	&\qquad\qquad+(p_h,\nabla\times\bm S_h)_{\mathcal T_h}-\langle\widehat p_h,\bm S_h\cdot\bm t\rangle_{\partial\mathcal T_h}+(\nabla\times\bm\phi_h,p_h)_{\mathcal T_h}\\
	&\qquad\qquad-\langle \bm\phi_h\cdot\bm t,\widehat p_h\rangle_{\partial\mathcal T_h}-(q_h,\nabla\times\bm R_h)_{\mathcal T_h}
	+\langle\widehat q_h,\bm R_h\cdot\bm t\rangle_{\partial\mathcal T_h}\\
	&\qquad\qquad-(\nabla\times\bm\theta_h,q_h)_{\mathcal T_h}+\langle \bm\theta_h\cdot\bm t,\widehat q_h\rangle_{\partial\mathcal T_h}\\
	&\qquad\qquad+\langle\alpha_3(\Pi_{k-1}^{\partial}p_h-\widehat p_h),\Pi_{k-1}^{\partial}q_h-\widehat q_h\rangle_{\partial\mathcal T_h}. 
	%    \end{align*}
%%    \begin{align*}
	%	& \mathfrak{a}_h(\bm G_h, \omega_h,\widehat{\omega}_h;\bm H_h,s_h,\widehat{s}_h)\\
	%	&\qquad:=(\bm G_h,\bm H_h)_{\mathcal T_h}+(\omega_h,\nabla\cdot\bm H_h)_{\mathcal T_h}-\langle \widehat \omega_h,\bm H_h\cdot\bm n\rangle_{\partial\mathcal T_h}\\
	%	&\qquad\qquad-(\nabla\cdot\bm G_h,s_h)_{\mathcal T_h}+\langle \bm G_h\cdot\bm n,\widehat s_h\rangle_{\partial\mathcal T_h}\\
	%	&\qquad\qquad+\langle\alpha_1(\Pi_{k-1}^{\partial}\omega_h-\widehat \omega_h),\Pi_{k-1}^{\partial}s_h-\widehat s_h\rangle_{\partial\mathcal T_h}.
\end{align*}
So, the HDG formulation could be rewritten in the following equivalent compact form:
\begin{subequations}\label{discretfor}
	\begin{equation}\label{discrete stage1}
		\mathfrak{a}_h(\bm L_h,r_h,\widehat{r}_h;\bm M_h,\mu_{h},\widehat{\mu}_h)=(g,\mu)_{\mathcal T_h},\qquad\forall (\bm M_h,\mu_{h},\widehat{\mu}_h)\in\mathcal{L}_h \times \mathcal{R}_h \times \widehat{\mathcal{R}}_h,
	\end{equation}
	\begin{equation}\label{discrete stage2}
		\begin{aligned}
			&\mathfrak{b}_h(\bm \sigma_h,\bm R_h,\bm \theta_h,\widehat{\bm\theta}_h,p_h,\widehat p_h ;\bm \tau_h,\bm S_h,\bm \phi_h,\widehat{\phi}_h,q_h,\widehat q_h)=(\bm L_h,\bm {\phi}_h)_{\mathcal T_h}
			+(\bm f,\bm \phi_h)_{\mathcal T_h} \\
			&\quad\qquad\qquad\qquad\forall (\bm \tau_h,\bm S_h,\bm \phi_h,\widehat{\phi}_h,q_h,\widehat q_h)
			\in \Sigma_h \times \mathcal{S}_h\times \mathcal{Y}_h \times\widehat{\mathcal{Y}}_h\times\mathcal{P}_h\times\widehat{\mathcal{P}}_h,
		\end{aligned}
	\end{equation}
	\begin{equation}\label{discrete stage3}
		\begin{aligned}	   	
			\mathfrak{a}_h(\bm G_h, \omega_h,\widehat{\omega}_h;\bm H_h,s_h,\widehat{s}_h)&= \lambda^{-1}t^2(g,s_h)+\langle \bm \theta_h\cdot\bm n,\widehat s_h\rangle_{\partial\mathcal T_h}-(\nabla\cdot\bm \theta_h,s_h)_{\mathcal T_h} \\
			&\qquad\qquad\qquad\qquad\forall (\bm H_h,s_h,\widehat{s}_h)\in  \mathcal{G}_h\times \mathcal{W}_h \times \widehat{\mathcal{W}}_h,
		\end{aligned}
	\end{equation}
	\begin{equation}
		\bm\gamma_h=\bm L_h+\lambda t^{-2}\bm R_h.
	\end{equation}	
\end{subequations}
So we consider split the processing of solving the discrete system \eqref{discrete formulation} into three problems.\par
For the choice of the stabilization coefficients, \cite{doi:10.1137/070706616},\cite{Cockburn20051067} and \cite{Nguyen20111147} give the discussion for Possion problem and the Stoke flow problem. In this paper, we take 
$$\alpha_1|_K:=h_K^{-1},\quad\alpha_2|_K:=h_K^{-1},\quad\alpha_3|_K:=h_K+t^2h_K^{-1}, \qquad\forall K \in \mathcal{T}_h. $$
In the following parts, we should also use the notation $\mathfrak h|_K=h_K$, for all $K\in\mathcal T_h$. \par
\subsection{Prior Error for Step One}
We firstly introduce a norm {\color{black} on space $\mathcal{L}_h \times \mathcal{R}_h \times \widehat{\mathcal{R}}_h$} as
$$\|(\bm M_h, \mu_h,\widehat{\mu}_h)\|^2_{\mathfrak{a}_h}=\|\bm M_h\|^2_{\mathcal{T}_h}+\|\nabla \mu_h\|^2_{\mathcal{T}_h}+\|\alpha_1^\frac{1}{2}(\Pi_{k-1}^{\partial}\mu_h-\widehat \mu_h)\|^2_{\partial\mathcal{T}_h},$$
that appears in natural way. By a standard argument, we can obtain the following lemma.
\begin{lemma}
	For all $(\bm L_h,r_h,\widehat{r}_h)\in\mathcal{L}_h \times \mathcal{R}_h \times \widehat{\mathcal{R}}_h$ , we have 
	\begin{equation}\label{discretelbb1}
		\sup_{\bm 0\neq(\bm M_h, \mu_h,\widehat{\mu}_h)\in\mathcal{L}_h \times \mathcal{R}_h \times \widehat{\mathcal{R}}_h}\frac{\mathfrak{a}_h(\bm L_h,r_h,\widehat{r}_h;\bm M_h,\mu_{h},\widehat{\mu}_h)}{\|(\bm M_h, \mu_h,\widehat{\mu}_h)\|_{\mathfrak{a}_h}}\geq C\|(\bm L_h,r_h,\widehat{r}_h)\|_{\mathfrak{a}_h}.
	\end{equation}
\end{lemma}\par
For the analysis, we state the following interpolation, see definition and more details in \cite{LI201637}. We here give some properties of the interpolation.
\begin{lemma}
	If $\mathfrak{n}$ is a positive integer that is large enough, then there exits an interpolation operator $\mathcal{I}_h: \mathrm{L}^2(\Omega)\times\mathrm{L}^2(\mathcal{F}_h)\rightarrow H_0^1(\Omega)\cap \mathcal{P}_{k+\mathfrak{n}}(\mathcal{T}_h)$ such that for all $(r_h,\widehat{r}_h)\in\mathrm{L}^2(\Omega)\times\mathrm{L}^2(\mathcal{F}_h)$, for all $K\in \mathcal{T}_h$, and for all $F\in \mathcal{F}_h$, we have 
	\begin{subequations}\label{pro}
		\begin{align}
			\label{pro3}
			(\mathcal{I}_h(r_h,\widehat{r}_h),\mu_h)_K&=(r_h,\mu)_K, &\forall \mu\in \mathcal{P}_k(K),\\
			\label{pro4}
			(\mathcal{I}_h(r_h,\widehat{r}_h),\widehat{\mu})_F&=(\widehat{r}_h,\widehat{\mu})_F, &\forall \widehat{\mu}\in \mathcal{P}_k(F),\\
			\label{pro2}
			\|\nabla\mathcal{I}_h(r_h,\widehat{r}_h)\|_{\mathcal{T}_h}&\leq C\left(\|\nabla r_h\|_{\mathcal{T}_h}+\|\mathfrak{h}^{-1/2}(r_h-\widehat{r}_h)\|_{\partial\mathcal{T}_h}\right),\\
			\label{pro1}
			\|r_h-\mathcal{I}_h(r_h,\widehat{r}_h)\|_{\mathcal{T}_h}&\leq C h\left(\|\nabla r_h\|_{\mathcal{T}_h}+\|\mathfrak{h}^{-1/2}(r_h-\widehat{r}_h)\|_{\partial\mathcal{T}_h}\right),
		\end{align}
	\end{subequations}
	where $\mathcal{P}_{k+\mathfrak{n}}(\mathcal{T}_h)=\left\{r_h\in\mathrm{L}^2(\Omega):r_h\in\mathcal{P}_{k+\mathfrak{n}}(K), \forall K \in\mathcal{T}_h\right\}$. Boldface fonts $\bm{\mathcal{I}}_h$ will be used for vector interpolation operator counterpart.
\end{lemma} 
\begin{remark}
	Actually, we can denote the operator $\mathcal{I}_h$ onto more specified spaces. For instance, we can let $\mathcal{I}_h$ onto $H_0^1(\Omega)\cap \mathcal{P}_{k+\mathfrak{n}}(\mathcal{T}_h)$, when $\widehat{r}_h|_{\partial\Omega}=0$. It trivially holds by \eqref{pro4}, which will be used for the argument of \Cref{theorem1}. Also, we can let $\mathcal{I}_h$ onto $\widehat{H}^1(\Omega)\cap \mathcal{P}_{k+\mathfrak{n}}(\mathcal{T}_h)$, when $r_h\in\mathrm{L}_0^2(\Omega)$. It holds by taking $\mu=1$ in \eqref{pro3}, which will be used for the argument of \Cref{mainlemma}.
\end{remark}
\begin{lemma}\label{lemma1}
	Let $(\bm L,r)$ be the solution of \eqref{stage1} and let $(\bm L_h,r_h,\widehat{r}_h)$ be the numerical solution of \eqref{discrete stage1}, then we have 
	\begin{equation}\label{errorprojection1}
		\begin{aligned}
			&\mathfrak{a}_h(\bm\Pi_{k-1}^o\bm L,\Pi_k^o r,\Pi_{k-1}^{\partial} r;\bm M_h,\mu_{h},\widehat{\mu}_h)\\
			&\qquad=(g,\mathcal{I}_h(\mu_h,\widehat{\mu}_h))_{\mathcal T_h}+(\bm\Pi_{k-1}^o\bm L-\bm L,\nabla\mathcal{I}_h(\mu_h,\widehat{\mu}_h))_{\mathcal T_h}\\
			&\qquad\qquad+\langle\alpha_1(\Pi_{k-1}^{\partial}\Pi_k^o r-\Pi_{k-1}^{\partial} r),\Pi_{k-1}^{\partial}\mu_h-\widehat \mu_h\rangle_{\partial\mathcal T_h}.
		\end{aligned}
	\end{equation}
\end{lemma}
\begin{proof}{\color{black} By the definition of $\mathfrak{a}_h$, we can obtain}
	\begin{align*}
		&\mathfrak{a}_h(\bm\Pi_{k-1}^o\bm L,\Pi_k^o r,\Pi_{k-1}^{\partial} r;\bm M_h,\mu_{h},\widehat{\mu}_h)\\
		&\qquad =(\bm\Pi_{k-1}^o\bm L,\bm M_h)_{\mathcal T_h}+(\Pi_k^o r,\nabla\cdot\bm M_h)_{\mathcal T_h}-\langle\Pi_{k-1}^{\partial}r,\bm M_h\cdot\bm n\rangle_{\partial\mathcal T_h}\\
		&\qquad\qquad-(\nabla\cdot\bm\Pi_{k-1}^o\bm L,\mu_h)_{\mathcal T_h}+\langle \bm\Pi_{k-1}^o\bm L\cdot\bm n,\widehat \mu_h\rangle_{\partial\mathcal T_h}\\
		&\qquad\qquad+\langle\alpha_1(\Pi_{k-1}^{\partial}\Pi_k^o r-\Pi_{k-1}^{\partial} r),\Pi_{k-1}^{\partial}\mu_h-\widehat \mu_h\rangle_{\partial\mathcal T_h}\\
		&\qquad=(\bm\Pi_{k-1}^o\bm L,\bm M_h)_{\mathcal T_h}+(\Pi_k^o r,\nabla\cdot\bm M_h)_{\mathcal T_h}-\langle\Pi_{k-1}^{\partial}r,\bm M_h\cdot\bm n\rangle_{\partial\mathcal T_h}\\
		&\qquad\qquad-(\nabla\cdot\bm\Pi_{k-1}^o\bm L,\mathcal{I}_h(\mu_h,\widehat{\mu}_h))_{\mathcal T_h}+\langle \bm\Pi_{k-1}^o\bm L\cdot\bm n,\mathcal{I}_h(\mu_h,\widehat{\mu}_h)\rangle_{\partial\mathcal T_h}\\
		&\qquad\qquad+\langle\alpha_1(\Pi_{k-1}^{\partial}\Pi_k^o r-\Pi_{k-1}^{\partial} r),\Pi_{k-1}^{\partial}\mu_h-\widehat \mu_h\rangle_{\partial\mathcal T_h}
	\end{align*}
	Applying the orthogonality of ${\rm L}^2$ projections and the properties of $\mathcal{I}_h$ \eqref{pro},
	\begin{align*}
		&\mathfrak{a}_h(\bm\Pi_{k-1}^o\bm L,\Pi_k^o r,\Pi_{k-1}^{\partial} r;\bm M_h,\mu_{h},\widehat{\mu}_h)\\
		&\qquad =(\bm L,\bm M_h)_{\mathcal T_h}+(r,\nabla\cdot\bm M_h)_{\mathcal T_h}-\langle r,\bm M_h\cdot\bm n\rangle_{\partial\mathcal T_h}\\
		&\qquad\qquad+(\bm L,\nabla\mathcal{I}_h(\mu_h,\widehat{\mu}_h))_{\mathcal T_h}+(\bm\Pi_{k-1}^o\bm L-\bm L,\nabla\mathcal{I}_h(\mu_h,\widehat{\mu}_h))_{\mathcal T_h}\\
		&\qquad\qquad+\langle\alpha_1(\Pi_{k-1}^{\partial}\Pi_k^o r-\Pi_{k-1}^{\partial} r),\Pi_{k-1}^{\partial}\mu_h-\widehat \mu_h\rangle_{\partial\mathcal T_h}.
	\end{align*}\par
	Then we use the relations of \eqref{discrete stage1}, we get
	\begin{align*}
		&\mathfrak{a}_h(\bm\Pi_{k-1}^o\bm L,\Pi_k^o r,\Pi_{k-1}^{\partial} r;\bm M_h,\mu_{h},\widehat{\mu}_h)\\
		&\qquad=(g,\mathcal{I}_h(\mu_h,\widehat{\mu}_h))_{\mathcal T_h}+(\bm\Pi_{k-1}^o\bm L-\bm L,\nabla\mathcal{I}_h(\mu_h,\widehat{\mu}_h))_{\mathcal T_h}\\
		&\qquad\qquad+\langle\alpha_1(\Pi_{k-1}^{\partial}\Pi_k^o r-\Pi_{k-1}^{\partial} r),\Pi_{k-1}^{\partial}\mu_h-\widehat \mu_h\rangle_{\partial\mathcal T_h}.
	\end{align*}
	This completes the proof.
\end{proof}\par
\begin{theorem}\label{theorem1}
	Under condition of  \Cref{lemma1} , it holds 
	\begin{equation}
		\|\bm L-\bm L_h\|_{\mathcal{T}_h}+\|\nabla r-\nabla r_h\|_{\mathcal{T}_h}\leq Ch^k(\|r\|_{k+1}+\|g\|_{k-1}).
	\end{equation}
\end{theorem}\par
\begin{proof}
	We introduce the follow expressions to make the argument more concise:
	\begin{align*} \color{black}
		\bm \xi_{\bm L}:= \bm \Pi_{k-1}^o  \bm L-\bm L_h,
		\qquad\xi_{r}:= \Pi_k^o r-r_h,
		\qquad\xi_{\widehat r}:= \Pi_{k-1}^{\partial} r-\widehat{r}_h.
	\end{align*}
	Then, we get                   
	\begin{align*}
		&\mathfrak{a}_h(\bm\xi_{\bm L},\xi_{r},\widehat{\xi}_r;\bm M_h,\mu_{h},\widehat{\mu}_h)\\
		&\qquad=(g,\mathcal{I}_h(\mu_h,\widehat{\mu}_h)-\mu_h)_{\mathcal T_h}+(\bm\Pi_{k-1}^o\bm L-\bm L,\nabla\mathcal{I}_h(\mu_h,\widehat{\mu}_h))_{\mathcal T_h}\\
		&\qquad\qquad+\langle\alpha_1(\Pi_{k-1}^{\partial}\Pi_k^o r-\Pi_{k-1}^{\partial} r),\Pi_{k-1}^{\partial}\mu_h-\widehat \mu_h\rangle_{\partial\mathcal T_h}\\
		&\qquad:=E_1+E_2+E_3.
	\end{align*}\par
	For the first term, we have the following result by Cauchy-Schwartz’s inequality and \eqref{3} and the properties of $\mathcal{I}_h$ \eqref{pro1}
	\begin{align*}
		E_1&=(g-\Pi^o_{k-1}g,\mathcal{I}_h(\mu_h,\widehat{\mu}_h)-\mu_h)_{\mathcal T_h}\\
		&\leq Ch^{k}\|g\|_{k-1}(\|\nabla \mu_h\|_{\mathcal{T}_h}+\|\mathfrak{h}^{-1/2}(\mu_h-\widehat{\mu}_h)\|_{\partial\mathcal{T}_h})\\
		&\leq Ch^{k}\|g\|_{k-1}(\|\nabla \mu_h\|_{\mathcal{T}_h}+\|\alpha_1^{1/2}(\Pi_{k-1}^{\partial}\mu_h-\widehat{\mu}_h)\|_{\partial\mathcal{T}_h}).
	\end{align*}
	%		For the first term, we have the following estimate by Cauch-Schwarz’s inequality, \eqref{3} and \eqref{5} to get
	%		\begin{align*}
		%			E_1&=\langle (\bm\Pi_{k-1}^o\bm L-\bm L)\cdot\bm n,\widehat \mu_h-\Pi_{k-1}^{\partial}\mu_h\rangle_{\partial\mathcal T_h} \\
		%			&\quad \ +\langle (\bm\Pi_{k-1}^o\bm L-\bm L)\cdot\bm n,\Pi_{k-1}^{\partial}\mu_h-\mu_h\rangle_{\partial\mathcal T_h}\\
		%			&\leq \|\alpha_1^\frac{1}{2}(\Pi_{k-1}^{\partial}\mu_h-\widehat \mu_h)\|_{\partial\mathcal{T}_h}\|\alpha_1^{-\frac{1}{2}}(\bm\Pi_{k-1}^o\bm L-\bm L)\|_{\partial\mathcal{T}_h}\\
		%			&\quad \ +\|\Pi_{k-1}^{\partial}\mu_h-\mu_h\|_{\partial\mathcal{T}_h}\|\bm\Pi_{k-1}^o\bm L-\bm L\|_{\partial\mathcal{T}_h}\\
		%			&\leq Ch^k\|\bm L\|_k(\|\nabla \mu_h\|_{\mathcal{T}_h}+\|\alpha_1^\frac{1}{2}(\Pi_{k-1}^{\partial}\mu_h-\widehat \mu_h)\|_{\partial\mathcal{T}_h}).
		%		\end{align*}
	For the second term, we use the same argument as $E_1$ with the properties of $\mathcal{I}_h$ \eqref{pro2}
	\begin{align*}
		E_2&\leq Ch^k\|\bm L\|_k(\|\nabla \mu_h\|_{\mathcal{T}_h}+\|\alpha_1^{1/2}(\Pi_{k-1}^{\partial}\mu_h-\widehat{\mu}_h)\|_{\partial\mathcal{T}_h})
	\end{align*}
	For the third term, we obtain by Cauchy-Schwartz’s inequality, \eqref{2} and \eqref{5}:
	\begin{align*}
		E_3&\leq \|\alpha_1^{1/2}(\Pi_{k-1}^{\partial}\Pi_k^o r-\Pi_{\ell}^{\partial} r)\|_{\partial\mathcal{T}_h}\|\alpha_1^{1/2}(\Pi_{k-1}^{\partial}\mu_h-\widehat \mu_h)\|_{\partial\mathcal{T}_h}\\
		&\leq C h^k\|r\|_{k+1}\|\alpha_1^{\frac{1}{2}}(\Pi_{k-1}^{\partial}\mu_h-\widehat \mu_h)\|_{\partial\mathcal{T}_h}.
	\end{align*}\par
	Then, using the \eqref{discretelbb1}, we obtain the boundness for $\|(\bm\xi_{\bm L},\xi_{r},\widehat{\xi}_r)\|_{\mathfrak{a}_h}$:
	\begin{align*}
		\|(\bm\xi_{\bm L},\xi_{r},\widehat{\xi}_r)\|_{\mathfrak{a}_h}\leq Ch^k(\|r\|_{k+1}+\|g\|_{k-1}).
	\end{align*}
	Finally, after using the triangle inequality, we finish the proof.
\end{proof}\par
\subsection{Prior Error for Step Two}
Now, we mainly consider the second problem \eqref{discrete stage2}. Again, we introduce a norm that appears naturally in the analysis.\par
Let $(\bm\sigma,\bm R,\bm \theta, p) $ be the solution of \eqref{stage2} and let $(\bm\sigma_h,\bm R_h,\bm\theta_h,\widehat{\bm \theta}_h,p_h,\widehat{p}_h)$ be the numerical solution of \eqref{discrete stage2}.\par
%\begin{lemma}[Piecewise Korn's Inequality\cite{doi:10.1137/S0036142902401311}] Under the shape regular condition of $\mathcal{T}_h$, for any $\bm v_p\in [H^1(\mathcal{T}_h)]^2$, it holds
	%	\begin{equation}\label{localkorn}
		%		\|\nabla v_p\|^2_{\mathcal{T}_h}\leq C(\|\bm{\epsilon}(v_p)\|_{\mathcal{T}_h}+\sum\limits_{F\in\partial\mathcal{T}_h/\partial\Omega}h_F^{-1}\|\bm\Pi_1^{\partial}\llbracket v_p \rrbracket\|_F^2+\sup\limits_{\substack{\bm m\in \mathbb{RM}(\Omega)\\ \|\bm m\|_{\Gamma_s}=1\\ \int_{\Gamma_s}\bm m ds =0}}\langle v_p,\bm m\rangle^2_{\Gamma_s}),
		%	\end{equation}
	%	where $\Gamma_s$ is a measurable subset of $\partial\Omega$ with the measure of  $\Gamma_s$ being greater than zero and $\mathbb{RM}(\Omega)$ is the space of (infinitesimal) rigid motions on $\Omega$ defined by
	%	\begin{equation}\nonumber
		%		\mathbb{RM}(\Omega):=\left\{ \bm a + b 
		%		\left[
		%		\begin{array}{cc}
			%			0 & -1\\
			%			1 & 0
			%		\end{array}
		%		\right]\bm x: \bm a \in \mathbb{R}^2, b\in \mathbb{R}\right\}
		%	\end{equation}
	%	for $x\in\Omega$.\par
	%\end{lemma}\par
%Now by \eqref{localkorn}), we are ready to get
We first introduce the HDG-Korn's inequality and the HDG-Poincar\'e’s inequality, see Lemma 3.3 of \cite{Chen1}.
\begin{lemma}[HDG-Poincar\'e’s Inequality]
	For any $\bm\theta_h\in \mathcal{Y}_h, \widehat{\bm\theta}_h\in\widehat{\mathcal{Y}}_h$, it holds
	\begin{equation}\label{HDGpoincare}
		\|\bm\theta_h\|^2_{\mathcal{T}_h}\leq C(\|\nabla\bm\theta_h\|^2_{\mathcal{T}_h}+\|\mathfrak{h}^{-\frac{1}{2}}(\bm\Pi_{\ell}^{\partial}\bm{\theta}_h-\widehat {\bm\theta}_h)\|^2_{\partial\mathcal{T}_h}).
	\end{equation}
\end{lemma}
\begin{lemma}[HDG-Korn’s Inequality]
	For any $\bm\theta_h\in \mathcal{Y}_h,\widehat{\bm\theta}_h\in\widehat{\mathcal{Y}}_h$ and sufficiently small $h$, it holds
	\begin{equation}\label{HDGkorn}
		\|\nabla\bm\theta_h\|^2_{\mathcal{T}_h}\leq C(\|\bm{\epsilon}(\bm\theta_h)\|^2_{\mathcal{T}_h}+\|\mathfrak{h}^{-\frac{1}{2}}(\bm\Pi_{\ell}^{\partial}\bm{\theta}_h-\widehat {\bm\theta}_h)\|^2_{\partial\mathcal{T}_h}).
	\end{equation}
\end{lemma}\par
%It should be noted that \eqref{HDGkorn} holds only for $k\ge 1, \ell=\max(k-1,1)$, whereas \eqref{HDGpoincare} holds for any $k\ge 1, \ell = k, k - 1$ and $k=\ell=0$.\par
We next introduce a semi-norm in the analysis of our method,  $\|\cdot\|_{\mathfrak{b}_h}: \Sigma_h \times \mathcal{S}_h\times \mathcal{Y}_h \times \widehat{\mathcal{Y}}_h\times\mathcal{P}_h\times\widehat{\mathcal{P}}_h \rightarrow \mathbb{R}$. Furthermore, we clarify that $\|\cdot\|_{\mathfrak{b}_h}$ is indeed a norm,  then state the priori error in the norm $\|\cdot\|_{\mathfrak{b}_h}$.
\begin{lemma}
	For all $(\bm\sigma_h,\bm R_h,\bm\theta_h,\widehat{\bm\theta}_h,p_h,\widehat p_h) \in \Sigma_h \times \mathcal{S}_h\times \mathcal{Y}_h \times \widehat{\mathcal{Y}}_h\times\mathcal{P}_h\times\widehat{\mathcal{P}}_h$, the definition of 
	\begin{equation}\label{norm}
		\begin{aligned}
			\|(\bm\sigma_h,\bm R_h,\bm\theta_h,\widehat{\bm\theta}_h,p_h,\widehat p_h)\|^2_{\mathfrak{b}_h}&=\|\bm \sigma_h\|^2_{\mathcal{T}_h}+t^{-2}\|\bm R_h\|^2_{\mathcal{T}_h}+\|\nabla\bm\theta_h\|^2_{\mathcal{T}_h}+\|\alpha_2^\frac{1}{2}(\bm\Pi_{\ell}^{\partial}\bm{\theta}_h-\widehat {\bm\theta}_h)\|^2_{\partial\mathcal{T}_h}\\
			&\qquad+t^2\|\nabla^{\perp}p_h\|^2_{\mathcal{T}_h}+\|\alpha_3^\frac{1}{2}(\Pi_{k-1}^{\partial} p_h-\widehat{p}_h)\|^2_{\partial\mathcal{T}_h},
		\end{aligned}
	\end{equation}
	is a norm for spaces $\Sigma_h \times \mathcal{S}_h\times \mathcal{Y}_h \times \widehat{\mathcal{Y}}_h\times\mathcal{P}_h\times\widehat{\mathcal{P}}_h$.
\end{lemma}\par
\begin{proof}
	It is enough to show \eqref{norm} is a norm for spaces $\Sigma_h \times \mathcal{S}_h\times \mathcal{Y}_h \times \widehat{\mathcal{Y}}_h\times\mathcal{P}_h\times\widehat{\mathcal{P}}_h$, assuming that $\|(\bm\sigma_h,\bm R_h,\bm\theta_h,\widehat{\bm\theta}_h,p_h,\widehat p_h)\|_{\mathfrak{b}_h}=0$,  that we show  $\bm\sigma_h,\bm R_h,\bm\theta_h,\widehat{\bm\theta}_h,p_h,\widehat p_h$ vanish. Thanks to \eqref{norm}, we only need to show $\bm\theta_h,\widehat{\bm\theta}_h,p_h,\widehat p_h$ vanish because $\|\bm \sigma_h\|_{\mathcal{T}_h}=0, \|\bm R_h\|_{\mathcal{T}_h}=0$ directly imply $\bm \sigma_h=\bm 0, \bm R_h=\bm 0$ for $(\bm{\sigma}_h,\bm R_h)\in\Sigma_h \times \mathcal{S}_h$.\par
	It follows readily that $\nabla\bm\theta_h=\bm 0, \nabla^{\perp}p_h=\bm 0 $ in $\mathcal T_h$ and $\bm\Pi_{\ell}^{\partial}\bm{\theta}_h=\widehat {\bm\theta}_h, \Pi_{k-1}^{\partial} p_h=\widehat{p}_h$ on $\partial\mathcal T_h$. So $\bm\theta_h$ is constant in $\mathcal{T}_h$ and $\bm\Pi_{\ell}^{\partial}\bm\theta_h=\bm 0$ on $\partial \Omega$ because $\widehat{\bm\theta}_h = \bm 0$ on $\partial \Omega$. We can conclude that $\bm\theta_h=\bm 0$ in $\mathcal T_h$ and $\widehat{\bm\theta}_h=\bm 0$ on $\partial\mathcal T_h$. For the scaler variable $p_h,\widehat{p}_h$, we note that $p_h\in\mathrm{L}_0^2(\Omega )$. After a simple integration combining that $p_h$ is also constant in $\mathcal{T}_h$, now reads $p_h=0$ in $\mathcal{T}_h$ and $\widehat{p}_h=0$ on $\partial\mathcal T_h$. This completes the proof.
\end{proof}\par
Following the well known theory of \cite{zbMATH03340747}, the discrete LBB condition is necessary for well-posedness of our method which makes $\mathfrak{b}_h$ coercive.\par
\begin{lemma}[Discrete LBB Condition]
	For all $(\bm \sigma_h, \bm R_h,\bm \theta_h, \widehat{\bm \theta}_h,p_h,\widehat{p}_h) \in \Sigma_h \times\mathcal{S}_h\times  \mathcal{Y}_h \times  \widehat{\mathcal{Y}}_h\times\mathcal{Q}_h\times\widehat{\mathcal{Q}}_h$, we have 
	\begin{equation}\label{lbbinequality}
		\begin{aligned}
			\sup_{\color{black}\bm 0\neq (\bm \tau_h, \bm S_h,\bm \phi_h, \widehat{\bm \phi}_h,q_h,\widehat{q}_h)\atop
				\in \Sigma_h \times\mathcal{S}_h\times  \mathcal{Y}_h \times  \widehat{\mathcal{Y}}_h\times\mathcal{Q}_h\times\widehat{\mathcal{Q}}_h}&\frac{\mathfrak{b}_h(\bm \sigma_h, \bm R_h,\bm \theta_h, \widehat{\bm \theta}_h,p_h,\widehat{p}_h;\bm \tau_h, \bm S_h,\bm \phi_h, \widehat{\bm \phi}_h,q_h,\widehat{q}_h)}{\|(\bm \tau_h, \bm S_h,\bm \phi_h, \widehat{\bm \phi}_h,q_h,\widehat{q}_h)\|_{\mathfrak{b}_h}}\\
			&\qquad\geq C\|(\bm\sigma_h,\bm R_h,\bm\theta_h,\widehat{\bm\theta}_h,p_h,\widehat p_h)\|_{\mathfrak{b}_h}.
		\end{aligned}
	\end{equation}
\end{lemma}
\begin{proof}
	For any fixed $(\bm \sigma_h, \bm R_h,\bm \theta_h, \widehat{\bm \theta}_h,p_h,\widehat{p}_h) \in \Sigma_h \times\mathcal{S}_h\times  \mathcal{Y}_h \times  \widehat{\mathcal{Y}}_h\times\mathcal{Q}_h\times\widehat{\mathcal{Q}}_h$, let $\beta_1,\beta_2$ be two constants that will be specified below, we take $\bm \tau_h=\bm \sigma_h+\beta_1\bm{\epsilon}(\bm\theta_h)$, $\bm S_h=\bm R_h+\beta_2\lambda^{-1}t^2\nabla^{\perp}p_h$, $\bm\phi_h=\bm\theta_h$, $\widehat{\bm\phi}_h=\widehat{\bm\theta}_h$, $q_h=p_h$, $\widehat{q}_h=\widehat{p}_h$ and we get
	\begin{align*}
		&\mathfrak{b}_h(\bm \sigma_h, \bm R_h,\bm \theta_h, \widehat{\bm \theta}_h,p_h,\widehat{p}_h;\bm \tau_h, \bm S_h,\bm \phi_h, \widehat{\bm \phi}_h,q_h,\widehat{q}_h)\\
		&\qquad=(\mathcal C^{-1}\bm\sigma_h,\bm \sigma_h+\beta_1\bm{\epsilon}(\bm\theta_h))_{\mathcal T_h}+(\bm\theta_h,\nabla\cdot(\bm \sigma_h+\beta_1\bm{\epsilon}(\bm\theta_h)))_{\mathcal T_h}\\
		&\qquad\qquad-\langle \widehat{\bm\theta}_h,(\bm \sigma_h+\beta_1\bm{\epsilon}(\bm\theta_h))\bm n\rangle_{\partial\mathcal T_h}+(\lambda t^{-2} \bm R_h,\bm R_h+\beta_2\lambda^{-1}t^2\nabla^{\perp}p_h)_{\mathcal T_h}\\
		&\qquad\qquad-(\bm\theta_h,\nabla\cdot\bm\sigma_h)_{\mathcal T_h}+\langle\widehat{\bm\theta}_h,\bm\sigma_h\bm n\rangle_{\partial\mathcal T_h}+\langle\alpha_2(\bm\Pi_{\ell}^{\partial}\bm{\theta}_h-\widehat {\bm\theta}_h),\bm\Pi_{\ell}^{\partial}\bm\theta_h-\widehat {\bm\theta}_h\rangle_{\partial\mathcal T_h}\\
		&\qquad\qquad+(p_h,\nabla\times(\bm R_h+\beta_2\lambda^{-1}t^2\nabla^{\perp}p_h))_{\mathcal T_h}-\langle\widehat p_h,(\bm R_h+\beta_2\lambda^{-1}t^2\nabla^{\perp}p_h)\cdot\bm t\rangle_{\partial\mathcal T_h}\\
		&\qquad\qquad+(\nabla\times\bm\theta_h,p_h)_{\mathcal T_h}-\langle \bm\theta_h\cdot\bm t,\widehat p_h\rangle_{\partial\mathcal T_h}-(p_h,\nabla\times\bm R_h)_{\mathcal T_h}
		+\langle\widehat p_h,\bm R_h\cdot\bm t\rangle_{\partial\mathcal T_h}\\
		&\qquad\qquad-(\nabla\times\bm\theta_h,p_h)_{\mathcal T_h}+\langle \bm\theta_h\cdot\bm t,\widehat p_h\rangle_{\partial\mathcal T_h}+\langle\alpha_3(\Pi_{k-1}^{\partial}p_h-\widehat p_h),\Pi_{k-1}^{\partial}p_h-\widehat p_h\rangle_{\partial\mathcal T_h}.
	\end{align*}
	Then applying the orthogonality of ${\rm L}^2$ projections, we get
	\begin{align*}
		&\mathfrak{b}_h(\bm \sigma_h, \bm R_h,\bm \theta_h, \widehat{\bm \theta}_h,p_h,\widehat{p}_h;\bm \tau_h, \bm S_h,\bm \phi_h, \widehat{\bm \phi}_h,q_h,\widehat{q}_h)\\
		&\qquad=(\mathcal C^{-1}\bm\sigma_h,\bm \sigma_h)_{\mathcal T_h}+(\mathcal C^{-1}\bm\sigma_h,\beta_1\bm{\epsilon}(\bm\theta_h))_{\mathcal T_h}+(-\beta_1\bm{\epsilon}(\bm\theta_h),\bm{\epsilon}(\bm\theta_h))_{\mathcal T_h}\\
		&\qquad\qquad+\langle\bm\Pi_{\ell}^{\partial}\bm\theta_h-\widehat{\bm \theta}_h,\beta_1\bm{\epsilon}(\bm\theta_h)\bm n\rangle_{\partial\mathcal T_h}+(\lambda t^{-2}\bm R_h,\bm R_h)_{\mathcal T_h}+(\bm R_h,\beta_2\nabla^{\perp}p_h)_{\mathcal T_h}\\
		&\qquad\qquad-(\nabla^{\perp}p_h,\beta_2\lambda^{-1}t^2\nabla^{\perp}p_h)_{\mathcal T_h}+\langle\Pi_{k-1}^{\partial}p_h-\widehat{p}_h,\beta_2\lambda^{-1}t^2\nabla^{\perp}p_h\cdot\bm t\rangle_{\partial\mathcal T_h}\\
		&\qquad\qquad+\langle\alpha_2(\bm\Pi_{\ell}^{\partial}\bm{\theta}_h-\widehat {\bm\theta}_h),\bm\Pi_{\ell}^{\partial}\bm\theta_h-\widehat {\bm\theta}_h\rangle_{\partial\mathcal T_h}+\langle\alpha_3(\Pi_{k-1}^{\partial}p_h-\widehat p_h),\Pi_{k-1}^{\partial}p_h-\widehat p_h\rangle_{\partial\mathcal T_h}.
	\end{align*}
	Next, using Cauchy-Schwartz inequality and \eqref{key} with $\beta_1<0,\beta_2<0$, we get
	\begin{align*}
		&\mathfrak{b}_h(\bm \sigma_h, \bm R_h,\bm \theta_h, \widehat{\bm \theta}_h,p_h,\widehat{p}_h;\bm \tau_h, \bm S_h,\bm \phi_h, \widehat{\bm \phi}_h,q_h,\widehat{q}_h)\\
		&\qquad\geq\|\bm \sigma_h\|^2_{\mathcal{T}_h}+\beta_1\|\bm \sigma_h\|_{\mathcal{T}_h}\|\bm{\epsilon}(\bm\theta_h)\|_{\mathcal{T}_h}-\beta_1\|\bm{\epsilon}(\bm\theta_h)\|^2_{\mathcal{T}_h}\\
		&\qquad\qquad+\beta_1\|\mathfrak{h}^{-\frac{1}{2}}(\bm\Pi_{\ell}^{\partial}\bm{\theta}_h-\widehat {\bm\theta}_h)\|_{\partial\mathcal{T}_h}\|\mathfrak{h}^{\frac{1}{2}}\bm{\epsilon}(\bm\theta_h)\|_{\partial\mathcal{T}_h}+\lambda t^{-2}\|\bm R_h\|^2_{\mathcal{T}_h}\\
		&\qquad\qquad+\beta_2 t^{-1}\|\bm R_h\|_{\mathcal{T}_h}\cdot t\|\nabla^{\perp}p_h\|_{\mathcal{T}_h}-\beta_2\lambda^{-1} t^2\|\nabla^{\perp}p_h\|^2_{\mathcal{T}_h}\\
		&\qquad\qquad+\beta_2t\| \mathfrak{h}^{-\frac{1}{2}}(\Pi_{k-1}^{\partial}p_h-\widehat p)\|_{\partial\mathcal{T}_h}\cdot \lambda^{-1}t\|\mathfrak{h}^{\frac{1}{2}}\nabla^{\perp}p_h\|_{\partial\mathcal{T}_h}\\
		&\qquad\qquad+\|\alpha_2^\frac{1}{2}(\bm\Pi_{\ell}^{\partial}\bm{\theta}_h-\widehat {\bm\theta}_h)\|^2_{\partial\mathcal{T}_h}+\|\alpha_3^\frac{1}{2}(\Pi_{k-1}^{\partial}p_h-\widehat p_h)\|^2_{\partial\mathcal{T}_h}.
	\end{align*}
	Using \eqref{6} and arithmetic and geometric means inequality, we get
	\begin{align*}
		&\mathfrak{b}_h(\bm \sigma_h, \bm R_h,\bm \theta_h, \widehat{\bm \theta}_h,p_h,\widehat{p}_h;\bm \tau_h, \bm S_h,\bm \phi_h, \widehat{\bm \phi}_h,q_h,\widehat{q}_h)\\
		&\qquad\geq\|\bm \sigma_h\|^2_{\mathcal{T}_h}+\|\alpha_2^\frac{1}{2}(\bm\Pi_{\ell}^{\partial}\bm{\theta}_h-\widehat {\bm\theta}_h)\|^2_{\partial\mathcal{T}_h}-\frac{\beta_1}{2}(\|\bm{\epsilon}(\bm\theta_h)\|^2_{\mathcal{T}_h}+\|\mathfrak{h}^{-\frac{1}{2}}(\bm\Pi_{\ell}^{\partial}\bm{\theta}_h-\widehat {\bm\theta}_h)\|^2_{\partial\mathcal{T}_h})\\
		&\qquad\qquad+C\beta_1(\|\bm \sigma_h\|^2_{\mathcal{T}_h}+\|\mathfrak{h}^{-\frac{1}{2}}(\bm\Pi_{\ell}^{\partial}\bm{\theta}_h-\widehat {\bm\theta}_h)\|^2_{\partial\mathcal{T}_h})\\
		&\qquad\qquad+\lambda t^{-2}\|\bm R_h\|^2_{\mathcal{T}_h}+\|\alpha_3^\frac{1}{2}(\Pi_{k-1}^{\partial} p_h-\widehat{p}_h)\|^2_{\partial\mathcal{T}_h}-\frac{\beta_2}{2}\lambda^{-1}t^2\|\nabla^{\perp}p_h\|^2_{\mathcal{T}_h}\\
		&\qquad\qquad +C\beta_2(\lambda t^{-2}\|\bm R_h\|^2_{\mathcal{T}_h}+t^2\|\mathfrak{h}^{-\frac{1}{2}}(\Pi_{k-1}^{\partial} p_h-\widehat{p}_h)\|^2_{\partial\mathcal{T}_h}).
	\end{align*}
	Using HDG-Korn’s inequality \eqref{HDGkorn}, it holds
	\begin{align*}
		&\mathfrak{b}_h(\bm \sigma_h, \bm R_h,\bm \theta_h, \widehat{\bm \theta}_h,p_h,\widehat{p}_h;\bm \tau_h, \bm S_h,\bm \phi_h, \widehat{\bm \phi}_h,q_h,\widehat{q}_h)\\
		&\qquad\geq(1+C\beta_1)(\|\bm \sigma_h\|^2_{\mathcal{T}_h}+\|\alpha_2^\frac{1}{2}(\bm\Pi_{\ell}^{\partial}\bm{\theta}_h-\widehat {\bm\theta}_h)\|^2)-\frac{\beta_1}{2}\|\nabla\bm\theta_h\|^2_{\mathcal{T}_h}\\
		&\qquad\qquad+(1+C\beta_2)(\lambda t^{-2}\|\bm R_h\|^2_{\mathcal{T}_h}+\|\alpha_3^\frac{1}{2}(\Pi_{k-1}^{\partial} p_h-\widehat{p}_h)\|^2_{\partial\mathcal{T}_h})-\frac{\beta_2}{2}\lambda^{-1}t^2\|\nabla^{\perp}p_h\|^2_{\mathcal{T}_h}.
	\end{align*}\par
	Next, we take $1+C\beta_1>0,1+C\beta_2>0$ to conclude that
	\begin{align*}
		&\mathfrak{b}_h(\bm \sigma_h, \bm R_h,\bm \theta_h, \widehat{\bm \theta}_h,p_h,\widehat{p}_h;\bm \tau_h, \bm S_h,\bm \phi_h, \widehat{\bm \phi}_h,q_h,\widehat{q}_h)\\
		&\qquad\geq C\|(\bm \sigma_h, \bm R_h,\bm \theta_h, \widehat{\bm \theta}_h,p_h,\widehat{p}_h)\|^2_{\mathfrak{b}_h}.
	\end{align*}\par
	With   $\bm \tau_h=\bm \sigma_h+\beta_1\bm{\epsilon}(\bm\theta_h)$, 
	$\bm S_h=\bm R_h+\beta_2\lambda^{-1}t^2\nabla^{\perp}p_h$,
	$\bm\phi_h=\bm\theta_h$,
	$\widehat{\bm\phi}_h=\widehat{\bm\theta}_h$ and
	$q_h=p_h$,$\widehat{q}_h=\widehat{p}_h$, the triangle inequality implies
	\begin{align*}
		&\|(\bm \tau_h, \bm S_h,\bm \phi_h, \widehat{\bm \phi}_h,q_h,\widehat{q}_h)\|_{\mathfrak{b}_h}\\
		&\qquad\leq\max\{1+|\beta_1|,1+|\beta_2|\}\|(\bm \sigma_h, \bm R_h,\bm \theta_h, \widehat{\bm \theta}_h,p_h,\widehat{p}_h)\|_{\mathfrak{b}_h}\\
		&\qquad\leq C\|(\bm \sigma_h, \bm R_h,\bm \theta_h, \widehat{\bm \theta}_h,p_h,\widehat{p}_h)\|_{\mathfrak{b}_h}.
	\end{align*}
	Then the result \eqref{lbbinequality} follows immediately.
\end{proof}\par
\begin{lemma}
	Let $(\bm \sigma,\bm R,\bm\theta,p) \in (\left[\mathrm{L}^2(\Omega)\right]^{2 \times 2}\cap \mathbb{S})\times \left[\mathrm{L}^2(\Omega)\right]^2 \times \left[H_0^1(\Omega)\right]^2\times \widehat{H}^1(\Omega)$ be the weak solution to the \eqref{stage2}, then for all $(\bm \tau_h, \bm S_h,\bm \phi_h, \widehat{\bm \phi}_h,q_h,\widehat{q}_h) \in \mathcal{T}_h \times\mathcal{S}_h\times  \mathcal{Y}_h \times  \widehat{\mathcal{Y}}_h\times\mathcal{Q}_h\times\widehat{\mathcal{Q}}_h$, it holds 
	\allowdisplaybreaks
	\begin{equation}\label{errorprojection}
		\begin{aligned}
			&\mathfrak{b}_h(\bm \Pi_{k-1}^o\bm \sigma,\bm \Pi_{k-1}^o \bm R,\bm \Pi_{k}^o \bm \theta,\bm\Pi_{\ell}^{\partial}\bm \theta,\Pi_k^o p,\Pi_{k-1}^{\partial} p;\bm \tau_h, \bm S_h,\bm\phi_h,\widehat{\bm \phi}_h,q_h,\widehat{q}_h)\\
			&\qquad=(\bm L+\bm f, \bm{\mathcal{I}}_h(\bm\phi_h,\widehat{\bm\phi}_h))_{\mathcal T_h}+(\nabla\bm{\mathcal{I}}_h(\bm\phi_h,\widehat{\bm\phi}_h),\bm \Pi_{k-1}^o\bm\sigma-\bm\sigma)_{\mathcal T_h}\\
			&\qquad\qquad+\langle\alpha_2(\bm\Pi_{\ell}^{\partial}\bm \Pi_{k}^o\bm{\theta}-\bm\Pi_{\ell}^{\partial}\bm \theta),\bm\Pi_{\ell}^{\partial}\bm\phi_h-\widehat {\bm\phi}_h\rangle_{\partial\mathcal T_h}-\langle (\bm\phi_h-\bm\Pi^o_{k-1} \bm \phi_h)\cdot\bm t, \Pi_{k-1}^{\partial}p-p\rangle_{\partial{\mathcal{T}_h}}\\
			&\qquad\qquad+(\nabla\times\bm{\mathcal{I}}_h(\bm\phi_h,\widehat{\bm\phi}_h),\bm \Pi_k^o p-p)_{\mathcal T_h}-\langle (\bm{\mathcal{I}}_h(\bm\phi_h,\widehat{\bm\phi}_h)-\bm\phi_h)\cdot\bm t,\Pi_{k}^o p-p\rangle_{\partial\mathcal T_h}\\
			&\qquad\qquad+(\nabla^{\perp}\mathcal{I}_h(q_h,\widehat q_h),\bm \Pi_{k-1}^o\bm R-\bm R)_{\mathcal T_h}+(\bm \Pi_{k}^o\bm\theta-\bm \theta,\nabla^{\perp}\mathcal{I}_h(q_h,\widehat q_h))_{\mathcal T_h}\\
			&\qquad\qquad+\langle\alpha_3(\Pi_{k-1}^{\partial}\Pi_k^op-\Pi_{k-1}^{\partial} p),\Pi_{k-1}^{\partial}q_h-\widehat q_h\rangle_{\partial\mathcal T_h}.
		\end{aligned}
	\end{equation}
\end{lemma}
\begin{proof}
	By the definition of $\mathfrak{b}_h$, we get
	\begin{align*}
		&\mathfrak{b}_h(\bm \Pi_{k-1}^o\bm \sigma,\bm \Pi_{k-1}^o \bm R,\bm \Pi_{k}^o \bm \theta,\bm\Pi_{\ell}^{\partial}\bm \theta,\Pi_k^o p,\Pi_{k-1}^{\partial} p;\bm \tau_h, \bm S_h,\bm\phi_h,\widehat{\bm \phi}_h,q_h,\widehat{q}_h)\\
		&\qquad=(\mathcal C^{-1}\bm\Pi_{k-1}^o\bm\sigma,\bm\tau_h)_{\mathcal T_h}+(\bm\Pi_k^o\bm\theta,\nabla\cdot\bm\tau_h)_{\mathcal T_h}-\langle \bm\Pi_{\ell}^{\partial}\bm\theta,\bm\tau_h\bm n\rangle_{\partial\mathcal T_h}\\
		&\qquad\qquad+(\lambda t^{-2} \bm \Pi_{k-1}^o\bm R,\bm S_h)_{\mathcal T_h}-(\bm\phi_h,\nabla\cdot\bm \Pi_{k-1}^o\bm\sigma)_{\mathcal T_h}+\langle\widehat{\bm\phi}_h,\bm \Pi_{k-1}^o\bm\sigma\bm n\rangle_{\partial\mathcal T_h}\\
		&\qquad\qquad+\langle\alpha_2(\bm\Pi_{\ell}^{\partial}\bm \Pi_{k}^o\bm{\theta}-\bm\Pi_{\ell}^{\partial}\bm \theta),\bm\Pi_{\ell}^{\partial}\bm\phi_h-\widehat {\bm\phi}_h\rangle_{\partial\mathcal T_h}\\
		&\qquad\qquad+(\Pi_k^o p,\nabla\times\bm S_h)_{\mathcal T_h}-\langle\Pi_{k-1}^{\partial} p,\bm S_h\cdot\bm t\rangle_{\partial\mathcal T_h}+(\nabla\times\bm\phi_h,\bm \Pi_k^o p)_{\mathcal T_h}\\
		&\qquad\qquad-\langle \bm\phi_h\cdot\bm t,\Pi_{k-1}^{\partial} p\rangle_{\partial\mathcal T_h}-(q_h,\nabla\times\bm \Pi_{k-1}^o\bm R)_{\mathcal T_h}+\langle\widehat q_h,\bm \Pi_{k-1}^o\bm R\cdot\bm t\rangle_{\partial\mathcal T_h}\\
		&\qquad\qquad-(\nabla\times\bm \Pi_{k}^o\bm\theta,q_h)_{\mathcal T_h}+\langle \bm \Pi_{k}^o\bm\theta\cdot\bm t,\widehat q_h\rangle_{\partial\mathcal T_h}\\
		&\qquad\qquad+\langle\alpha_3(\Pi_{k-1}^{\partial}\Pi_k^op-\Pi_{k-1}^{\partial} p),\Pi_{k-1}^{\partial}q_h-\widehat q_h\rangle_{\partial\mathcal T_h}.
	\end{align*}
	then we introduce $\mathcal{I}_h$ to get
	\begin{align*}
		&\mathfrak{b}_h(\bm \Pi_{k-1}^o\bm \sigma,\bm \Pi_{k-1}^o \bm R,\bm \Pi_{k}^o \bm \theta,\bm\Pi_{\ell}^{\partial}\bm \theta,\Pi_k^o p,\Pi_{k-1}^{\partial} p;\bm \tau_h, \bm S_h,\bm\phi_h,\widehat{\bm \phi}_h,q_h,\widehat{q}_h)\\
		&\qquad=(\mathcal C^{-1}\bm\Pi_{k-1}^o\bm\sigma,\bm\tau_h)_{\mathcal T_h}+(\bm\Pi_k^o\bm\theta,\nabla\cdot\bm\tau_h)_{\mathcal T_h}-\langle \bm\Pi_{\ell}^{\partial}\bm\theta,\bm\tau_h\bm n\rangle_{\partial\mathcal T_h}\\
		&\qquad\qquad+(\lambda t^{-2} \bm \Pi_{k-1}^o\bm R,\bm S_h)_{\mathcal T_h}-(\bm{\mathcal{ I}}_h(\bm\phi_h,\widehat{\bm\phi}_h),\nabla\cdot\bm \Pi_{k-1}^o\bm\sigma)_{\mathcal T_h}\\
		&\qquad\qquad+\langle\bm{\mathcal{ I}}_h(\bm\phi_h,\widehat{\bm\phi}_h),\bm \Pi_{k-1}^o\bm\sigma\bm n\rangle_{\partial\mathcal T_h}+\langle\alpha_2(\bm\Pi_{\ell}^{\partial}\bm \Pi_{k}^o\bm{\theta}-\bm\Pi_{\ell}^{\partial}\bm \theta),\bm\Pi_{\ell}^{\partial}\bm\phi_h-\widehat {\bm\phi}_h\rangle_{\partial\mathcal T_h}\\
		&\qquad\qquad+(\Pi_k^o p,\nabla\times\bm S_h)_{\mathcal T_h}-\langle\Pi_{k-1}^{\partial} p,\bm S_h\cdot\bm t\rangle_{\partial\mathcal T_h}-\langle \bm\phi_h\cdot\bm t, \Pi_{k-1}^{\partial}p-p\rangle_{\partial{\mathcal{T}_h}}\\
		&\qquad\qquad+(\nabla\times\bm{\mathcal{ I}}_h(\bm\phi_h,\widehat{\bm\phi}_h),\bm \Pi_k^o p-p)_{\mathcal T_h}-(\bm{\mathcal{I}}_h(\bm\phi_h,\widehat{\bm\phi}_h),\nabla^{\perp}p)_{\mathcal T_h}\\
		&\qquad\qquad-\langle \bm{\mathcal{\bm I}}_h(\bm\phi_h,\widehat{\bm\phi}_h)\cdot\bm t,\Pi_{k}^o p-p\rangle_{\partial\mathcal T_h}-(\mathcal{I}_h(q_h,\widehat q_h),\nabla\times\bm \Pi_{k-1}^o\bm R)_{\mathcal T_h}\\
		&\qquad\qquad+\langle\mathcal{I}_h(q_h,\widehat{q}_h),\bm \Pi_{k-1}^o\bm R\cdot\bm t\rangle_{\partial\mathcal T_h}-(\nabla\times\bm \Pi_{k}^o\bm\theta,\mathcal{I}_h(q_h,\widehat q_h))_{\mathcal T_h}\\
		&\qquad\qquad+\langle \bm \Pi_{k}^o\bm\theta\cdot\bm t,\mathcal{I}_h(q_h,\widehat q_h)\rangle_{\partial\mathcal T_h}+\langle\alpha_3(\Pi_{k-1}^{\partial}\Pi_k^op-\Pi_{k-1}^{\partial} p),\Pi_{k-1}^{\partial}q_h-\widehat q_h\rangle_{\partial\mathcal T_h}.
	\end{align*}
	Then apply the orthogonality of projections ${\rm L}^2$  and $\mathcal{I}_h$ such as \eqref{pro3} and \eqref{pro4} to get 
	\begin{align*}
		&\mathfrak{b}_h(\bm \Pi_{k-1}^o\bm \sigma,\bm \Pi_{k-1}^o \bm R,\bm \Pi_{k}^o \bm \theta,\bm\Pi_{\ell}^{\partial}\bm \theta,\Pi_k^o p,\Pi_{k-1}^{\partial} p;\bm \tau_h, \bm S_h,\bm\phi_h,\widehat{\bm \phi}_h,q_h,\widehat{q}_h)\\
		&\qquad=(\mathcal C^{-1}\bm\sigma,\bm\tau_h)_{\mathcal T_h}-(\nabla\bm\theta,\bm\tau_h)_{\mathcal T_h}+(\lambda t^{-2}\bm R,\bm S_h)_{\mathcal T_h}\\
		&\qquad\qquad+(\nabla\bm{\mathcal{I}}_h(\bm\phi_h,\widehat{\bm\phi}_h),\bm \Pi_{k-1}^o\bm\sigma-\bm\sigma)_{\mathcal T_h}+(\nabla\bm{\mathcal{I}}_h(\bm\phi_h,\widehat{\bm\phi}_h),\bm\sigma)_{\mathcal T_h}\\
		&\qquad\qquad+\langle\alpha_2(\bm\Pi_{\ell}^{\partial}\bm \Pi_{k}^o\bm{\theta}-\bm\Pi_{\ell}^{\partial}\bm \theta),\bm\Pi_{\ell}^{\partial}\bm\phi_h-\widehat {\bm\phi}_h\rangle_{\partial\mathcal T_h}-\langle (\bm\phi_h-\bm\Pi^o_{k-1} \bm \phi_h)\cdot\bm t, \Pi_{k-1}^{\partial}p-p\rangle_{\partial{\mathcal{T}_h}}\\
		&\qquad\qquad-(\nabla^{\perp}p,\bm S_h)_{\mathcal T_h}+(\nabla\times\bm{\mathcal{I}}_h(\bm\phi_h,\widehat{\bm\phi}_h),\bm \Pi_k^o p-p)_{\mathcal T_h}-(\bm{\mathcal{I}}_h(\bm\phi_h,\widehat{\bm\phi}_h),\nabla^{\perp}p)_{\mathcal T_h}\\
		&\qquad\qquad-\langle \bm{\mathcal{I}}_h(\bm\phi_h,\widehat{\bm\phi}_h)\cdot\bm t,\Pi_{k}^o p-p\rangle_{\partial\mathcal T_h}+(\nabla^{\perp}\mathcal{I}_h(q_h,\widehat q_h),\bm \Pi_{k-1}^o\bm R-\bm R)_{\mathcal T_h}\\
		&\qquad\qquad+(\nabla^{\perp}\mathcal{I}_h(q_h,\widehat q_h),\bm R)_{\mathcal T_h}+(\bm \Pi_{k}^o\bm\theta-\bm \theta,\nabla^{\perp}\mathcal{I}_h(q_h,\widehat q_h))_{\mathcal T_h}\\
		&\qquad\qquad+(\bm \theta,\nabla^{\perp}\mathcal{I}_h(q_h,\widehat q_h))_{\mathcal T_h}+\langle\alpha_3(\Pi_{k-1}^{\partial}\Pi_k^op-\Pi_{k-1}^{\partial} p),\Pi_{k-1}^{\partial}q_h-\widehat q_h\rangle_{\partial\mathcal T_h}.
	\end{align*}
	The last step is straightforward consequence of \eqref{stage2}
	\begin{align*}
		&\mathfrak{b}_h(\bm \Pi_{k-1}^o\bm \sigma,\bm \Pi_{k-1}^o \bm R,\bm \Pi_{k}^o \bm \theta,\bm\Pi_{\ell}^{\partial}\bm \theta,\Pi_k^o p,\Pi_{k-1}^{\partial} p;\bm \tau_h, \bm S_h,\bm\phi_h,\widehat{\bm \phi}_h,q_h,\widehat{q}_h)\\
		&\qquad=(\nabla\bm{\mathcal{I}}_h(\bm\phi_h,\widehat{\bm\phi}_h),\bm \Pi_{k-1}^o\bm\sigma-\bm\sigma)_{\mathcal T_h}+(\bm{\mathcal{I}}_h(\bm\phi_h,\widehat{\bm\phi}_h),\bm L+\bm f)_{\mathcal T_h}\\
		&\qquad\qquad+\langle\alpha_2(\bm\Pi_{\ell}^{\partial}\bm \Pi_{k}^o\bm{\theta}-\bm\Pi_{\ell}^{\partial}\bm \theta),\bm\Pi_{\ell}^{\partial}\bm\phi_h-\widehat {\bm\phi}_h\rangle_{\partial\mathcal T_h}-\langle (\bm\phi_h-\bm\Pi^o_{k-1} \bm \phi_h)\cdot\bm t, \Pi_{k-1}^{\partial}p-p\rangle_{\partial{\mathcal{T}_h}}\\
		&\qquad\qquad+(\nabla\times\bm{\mathcal{I}}_h(\bm\phi_h,\widehat{\bm\phi}_h),\bm \Pi_k^o p-p)_{\mathcal T_h}-\langle (\bm{\mathcal{I}}_h(\bm\phi_h,\widehat{\bm\phi}_h)-\bm\phi_h)\cdot\bm t,\Pi_{k}^o p-p\rangle_{\partial\mathcal T_h}\\
		&\qquad\qquad+(\nabla^{\perp}\mathcal{I}_h(q_h,\widehat q_h),\bm \Pi_{k-1}^o\bm R-\bm R)_{\mathcal T_h}+(\bm \Pi_{k}^o\bm\theta-\bm \theta,\nabla^{\perp}\mathcal{I}_h(q_h,\widehat q_h))_{\mathcal T_h}\\
		&\qquad\qquad+\langle\alpha_3(\Pi_{k-1}^{\partial}\Pi_k^op-\Pi_{k-1}^{\partial} p),\Pi_{k-1}^{\partial}q_h-\widehat q_h\rangle_{\partial\mathcal T_h}.
	\end{align*}
\end{proof}\par
Also, we introduce the follow expressions to make the argument more concise.
\begin{align*}
	&\bm \xi_{\bm \sigma}:= \bm \Pi_{k-1}^o  \bm \sigma -\bm \sigma_h,\quad &&\bm \xi_{\bm R}:= \bm \Pi_{k-1}^o \bm R-  \bm R_h,\quad&\bm \xi_{\bm \theta}:=\bm \Pi_{k}^o \bm \theta- \bm \theta_h, \\
	&\xi_{p}:= \Pi_k^o p-p_h,\quad&&\bm \xi_{\widehat{\bm \theta}}:=  \bm\Pi_{\ell}^{\partial}\bm \theta-\widehat{\bm \theta}_h,\quad&\xi_{\widehat p}:= \Pi_{k-1}^{\partial} p-\widehat{p}_h.
\end{align*}\par                    
\begin{lemma}\label{mainlemma}
	Let $(\bm\sigma,\bm R,\bm \theta, p) $ be the solution of \eqref{stage2} and let $(\bm\sigma_h,\bm R_h,\bm\theta_h,\widehat{\bm \theta}_h,p_h,\widehat{p}_h)$ be the numerical solution of \eqref{discrete stage2},Then, under the
	regularity condition, the following error estimate holds:
	\begin{equation}
		\begin{aligned}
			\|(\bm\xi_{\bm\sigma},\bm\xi_{\bm R},\bm\xi_{\bm\theta},\bm\xi_{\widehat{\bm\theta}},\xi_p,\xi_{\widehat p})\|_{\mathfrak{b}_h}\leq C h^k(\|\bm f\|_{k-1}+\|g\|_{k-1}+\|r\|_{k+1}+\|\bm \sigma \|_k+\|\bm \theta\|_{k+1}+t\|p\|_{k+1}).
		\end{aligned}
	\end{equation}
\end{lemma}
\begin{proof}
	By the equation \eqref{errorprojection}, we get the following error equation
	\begin{align*}
		&\mathfrak{b}_h( \bm \xi_{\bm \sigma},\bm \xi_{\bm R},\bm \xi_{\bm \theta}, \bm \xi_{\widehat{\bm \theta}}, \xi_{p}, \xi_{\widehat p};\bm \tau_h, \bm S_h,\bm\phi_h,\widehat{\bm \phi}_h,q_h,\widehat{q}_h)\\
		&\qquad=[(\bm L-\bm L_h, \bm{\mathcal{I}}_h(\bm\phi_h,\widehat{\bm\phi}_h))_{\mathcal T_h}+(\bm f-\bm\Pi_{k-1}^o \bm f, \bm{\mathcal{I}}_h(\bm\phi_h,\widehat{\bm\phi}_h)-\bm\phi_h)_{\mathcal T_h}\\
		&\qquad\qquad+(\nabla\bm{\mathcal{I}}_h(\bm\phi_h,\widehat{\bm\phi}_h),\bm \Pi_{k-1}^o\bm\sigma-\bm\sigma)_{\mathcal T_h}-\langle (\bm\phi_h-\bm\Pi^o_{k-1} \bm \phi_h)\cdot\bm t, \Pi_{k-1}^{\partial}p-p\rangle_{\partial{\mathcal{T}_h}}\\
		&\qquad\qquad+(\nabla\times\bm{\mathcal{I}}_h(\bm\phi_h,\widehat{\bm\phi}_h),\bm \Pi_k^o p-p)_{\mathcal T_h}-\langle (\bm{\mathcal{I}}_h(\bm\phi_h,\widehat{\bm\phi}_h)-\bm\phi_h)\cdot\bm t,\Pi_{k}^o p-p\rangle_{\partial\mathcal T_h}]\\
		&\qquad\qquad+\langle\alpha_2(\bm\Pi_{\ell}^{\partial}\bm \Pi_{k}^o\bm{\theta}-\bm\Pi_{\ell}^{\partial}\bm \theta),\bm\Pi_{\ell}^{\partial}\bm\phi_h-\widehat {\bm\phi}_h\rangle_{\partial\mathcal T_h}\\
		&\qquad\qquad+[(\nabla^{\perp}\mathcal{I}_h(q_h,\widehat q_h),\bm \Pi_{k-1}^o\bm R-\bm R)_{\mathcal T_h}+(\bm \Pi_{k}^o\bm\theta-\bm \theta,\nabla^{\perp}\mathcal{I}_h(q_h,\widehat q_h))_{\mathcal T_h}]\\
		&\qquad\qquad+\langle\alpha_3(\Pi_{k-1}^{\partial}\Pi_k^op-\Pi_{k-1}^{\partial} p),\Pi_{k-1}^{\partial}q_h-\widehat q_h\rangle_{\partial\mathcal T_h}\\
		&\qquad:=E_4+E_5+E_6+E_7.
	\end{align*}\par
	For the term $E_4$, we use the same argument as $E_1,E_2$ with the estimate $ \langle (\bm\phi_h-\bm\Pi^o_{k-1} \bm \phi_h)\cdot\bm t, \Pi_{k-1}^{\partial}p-p\rangle_{\partial{\mathcal{T}_h}} \leq Ch^k\|\nabla \bm\phi_h\|_{\mathcal{T}_h}|p|_k $ to get the following estimate
	\begin{align*}
		E_4&\leq Ch^k(\|\bm f\|_{k-1}+\|g\|_{k-1}+\|r\|_{k+1}+\|\bm\sigma\|_k+\|p\|_k+\|\bm R\|_k)\\
		&\qquad\qquad\cdot(\|\nabla \bm\phi_h\|_{\mathcal{T}_h}+\|\alpha_2^{1/2}(\bm\Pi_{\ell}^{\partial}\bm\phi_h-\widehat{\bm\phi}_h)\|_{\partial\mathcal{T}_h}).
	\end{align*}
	For the term $E_5$, it holds
	\begin{align*}
		E_5&=\langle\alpha_2(\bm\Pi_{\ell}^{\partial}\bm \Pi_{k}^o \bm \theta-\bm\Pi_{\ell}^{\partial}\bm \theta),\bm\Pi_{\ell}^{\partial}\bm \phi_h-\widehat{\bm \phi}_h\rangle_{\partial\mathcal{T}_h} \\
		&\leq\|\alpha_2^{1/2}(\bm \Pi_{k}^o \bm \theta-\bm \theta)\|_{\partial\mathcal{T}_h}\|\alpha_2^{1/2}(\bm\Pi_{\ell}^{\partial}\bm{\phi}_h-\widehat {\bm\phi}_h)\|^2_{\partial\mathcal{T}_h} \\
		&\leq C h^k\|\bm\theta\|_{k+1}\|\alpha_2^{1/2}(\bm\Pi_{\ell}^{\partial}\bm{\phi}_h-\widehat {\bm\phi}_h)\|_{\partial\mathcal{T}_h}.
	\end{align*}
	%		and for the term $E_6$, we also use the same argument with $E_1, E_2$,
	%		\begin{align*}
		%			E_3&=\langle \widehat{\bm \phi}_h-\bm \phi_h,(\bm \Pi_{k-1}^o\bm \sigma-\bm \sigma)\bm n\rangle_{\partial\mathcal{T}_h} \\
		%			&=\langle \widehat{\bm \phi}_h-\bm\Pi_{\ell}^{\partial}\bm \phi_h,(\bm \Pi_{k-1}^o\bm \sigma-\bm \sigma)\bm n\rangle_{\partial\mathcal{T}_h}+\langle \bm\Pi_{\ell}^{\partial}\bm \phi_h-\bm \phi_h,(\bm \Pi_{k-1}^o\bm \sigma-\bm \sigma)\bm n\rangle_{\partial\mathcal{T}_h} \\
		%			&\text{(using Cauchy-Schwartz’s inequality and \eqref{5})}\\
		%			&\leq C h^k\|\bm \sigma \|_k(\|\mathfrak{h}^{-\frac{1}{2}}(\bm\Pi_{\ell}^{\partial}\bm \phi_h-\widehat{\bm \phi}_h)\|_{\partial\mathcal{T}_h}+\|\mathfrak{h}^{-\frac{1}{2}}(\bm\Pi_{\ell}^{\partial}\bm \phi_h-\bm \phi_h)\|_{\partial\mathcal{T}_h}) \\
		%			&\text{(using \eqref{3})}\\
		%			&\leq C h^k\|\bm \sigma \|_k(\|\alpha_2^{\frac{1}{2}}(\bm\Pi_{\ell}^{\partial}\bm \phi_h-\widehat{\bm \phi}_h)\|_{\partial\mathcal{T}_h}+\|\nabla\bm\phi_h\|_{\mathcal{T}_h}). 
		%		\end{align*}
	%		For the term $E_4$, the arguments is similar with $E_3$ to get 
	%		\begin{align*}
		%			E_4&=\langle \widehat{q}_h-q_h,(\bm \Pi_{k-1}^o \bm R-\bm R)\cdot \bm t\rangle_{\partial\mathcal{T}_h} \\
		%			&\leq C h^k t^{-1}\|\bm R \|_k(\|\alpha_3^{\frac{1}{2}}(\Pi_{k-1}^{\partial}q_h-\widehat{q}_h)\|_{\partial\mathcal{T}_h}+t\|\nabla^{\perp}q_h\|_{\mathcal{T}_h}). 
		%		\end{align*}
	For the term $E_6$, the main steps of analysis is a little bit different with $E_1,E_2$.  So we prove the details
	\begin{align*}
		E_6&=(\nabla^{\perp}\mathcal{I}_h(q_h,\widehat q_h),\bm \Pi_{k-1}^o\bm R-\bm R)_{\mathcal T_h}+(\bm \Pi_{k}^o\bm\theta-\bm \theta,\nabla^{\perp}\mathcal{I}_h(q_h,\widehat q_h))_{\mathcal T_h}\\
		&=(t\nabla^{\perp}\mathcal{I}_h(q_h,\widehat q_h),t^{-1}(\bm \Pi_{k-1}^o\bm R-\bm R))_{\mathcal T_h}+(\mathfrak{h}^{-1}(\bm \Pi_{k}^o\bm\theta-\bm \theta),\mathfrak{h}\nabla^{\perp}\mathcal{I}_h(q_h,\widehat q_h))_{\mathcal T_h}\\
		&\text{(using Cauchy-Schwartz’s inequality and \eqref{pro2} and \eqref{5})}\\
		&\leq C h^k\|\bm\theta\|_{k+1}(\|\mathfrak{h}\nabla^{\perp} q_h\|_{\mathcal{T}_h}+\|\mathfrak{h}^{1/2}(q_h-\widehat{q}_h)\|_{\partial\mathcal{T}_h})\\
		&\qquad+C h^kt^{-1}\|\bm R\|_k (t\|\nabla^{\perp} q_h\|_{\mathcal{T}_h}+\|t\mathfrak{h}^{-1/2}(q_h-\widehat{q}_h)\|_{\partial\mathcal{T}_h})\\
		&\text{(using \eqref{3})}\\
		&\leq C h^k\|\bm\theta\|_{k+1}(\|\mathfrak{h}\nabla^{\perp} q_h\|_{\mathcal{T}_h}+\|\alpha_3^{1/2}(\Pi^{\partial}_{k-1}q_h-\widehat{q}_h)\|_{\partial\mathcal{T}_h})\\
		&\qquad+C h^kt^{-1}\|\bm R\|_k (t\|\nabla^{\perp} q_h\|_{\mathcal{T}_h}+\|\alpha_3^{1/2}(\Pi^{\partial}_{k-1}q_h-\widehat{q}_h)\|_{\partial\mathcal{T}_h}).
	\end{align*}\par
	%		\begin{align*}
		%			E_5&=\langle \widehat{q}_h-q_h,(\bm \Pi_{k}^o \bm \theta-\bm \theta)\cdot \bm t\rangle_{\partial\mathcal{T}_h} \\
		%			&=\langle \mathfrak{h}^{\frac{1}{2}}(\widehat{q}_h-\Pi_{k-1}^{\partial}q_h),\mathfrak{h}^{-\frac{1}{2}}(\bm \Pi_{k}^o\bm \theta-\bm \theta)\cdot\bm t\rangle_{\partial\mathcal{T}_h}+\langle \mathfrak{h}^{\frac{1}{2}}(\Pi_{k-1}^{\partial}q_h-q_h),\mathfrak{h}^{-\frac{1}{2}}(\bm \Pi_{k}^o\bm \theta-\bm \theta)\cdot\bm t\rangle_{\partial\mathcal{T}_h} \\
		%			&\text{(repeatedly using Cauchy-Schwartz’s inequality and \eqref{5})}\\
		%			&\leq C h^k\|\bm \theta \|_{k+1}(\|\mathfrak{h}^{\frac{1}{2}}(\Pi_{k-1}^{\partial}q_h-\widehat{q}_h)\|_{\partial\mathcal{T}_h}+\|\mathfrak{h}^{\frac{1}{2}}(\Pi_{k-1}^{\partial}q_h-q_h)\|_{\partial\mathcal{T}_h}) \\
		%			&\text{(repeatedly using \eqref{3})}\\
		%			&\leq C h^k\|\bm \theta\|_{k+1}(\|\mathfrak{h}^{\frac{1}{2}}(\Pi_{k-1}^{\partial}q_h-\widehat{q}_h)\|_{\partial\mathcal{T}_h}+\|\mathfrak{h}\nabla^{\perp}q_h\|_{\mathcal{T}_h}).
		%		\end{align*}\par
	To deal with the term $\|\mathfrak{h}\nabla^{\perp}q_h\|_{\mathcal{T}_h}$, we proceed as follows. Firstly it's ready to get
	\begin{align*}
		\|\mathfrak{h}\nabla^{\perp}q_h\|_{\mathcal{T}_h}&=\frac{(\nabla^{\perp}q_h,\mathfrak{h}^2\nabla^{\perp}q_h)_{\mathcal{T}_h}}{\|\mathfrak{h}\nabla^{\perp}q_h\|_{\mathcal{T}_h}}\leq
		\sup_{\bm 0\neq\bm\theta_h\in\mathcal{Y}_h}\frac{(\nabla^{\perp}q_h,\bm\theta_h)_{\mathcal{T}_h}}{\|\mathfrak{h}^{-1}\bm\theta_h\|_{\mathcal{T}_h}},
	\end{align*}
	Next, we work on the numerator in the above expression. We have
	\begin{align*}
		(\nabla^{\perp}q_h,\bm\theta_h)&=\mathfrak{b}_h( \bm 0,\bm 0,\bm \theta_h, \bm 0, 0, 0;\bm \tau_h, \bm S_h,\bm\phi_h,\widehat{\bm \phi}_h,q_h,\widehat{q}_h)\\
		&\qquad-(\bm\theta_h,\nabla\cdot\bm\tau)_{\mathcal T_h}-\langle\alpha_2(\bm\Pi_{\ell}^{\partial}\bm{\theta}_h,\bm\Pi_{\ell}^{\partial}\bm\phi_h-\widehat {\bm\phi}_h\rangle_{\partial\mathcal T_h}+\langle\ \bm \theta_h\cdot \bm t , \Pi_{k-1}^\partial q_h-\widehat{q}_h\rangle_{\partial\mathcal{T}_h}.
	\end{align*}
	Last, by the triangle inequality, \eqref{4}, \eqref{2}, \eqref{6} and Cauchy-Schwartz’s inequality, we conclude that
	\begin{align*}
		\|\mathfrak h\nabla^{\perp}q_h\|_{\mathcal{T}_h}&\leq\sup_{\bm 0\neq\bm\theta_h\atop\in\mathcal{Y}_h}\frac{(\bm\theta_h,\nabla\cdot\bm\tau_h)_{\mathcal T_h}+\langle\alpha_2(\bm\Pi_{\ell}^{\partial}\bm{\theta}_h,\bm\Pi_{\ell}^{\partial}\bm\phi_h-\widehat {\bm\phi}_h\rangle_{\partial\mathcal T_h}+\langle\ \bm \theta_h\cdot \bm t ,\Pi_{k-1}^\partial q_h-\widehat{q}_h\rangle_{\partial\mathcal{T}_h}}{\|\mathfrak{h}^{-1}\bm\theta_h\|_{\mathcal{T}_h}}\\
		&\qquad+\sup_{\bm 0\neq\bm\theta_h\in\mathcal{Y}_h}\frac{\mathfrak{b}_h( \bm 0,\bm 0,\bm \theta_h, \bm 0, 0, 0;\bm \tau_h, \bm S_h,\bm\phi_h,\widehat{\bm \phi}_h,q_h,\widehat{q}_h)}{\|\mathfrak{h}^{-1}\bm\theta_h\|_{\mathcal{T}_h}}\\
		%&\qquad+\sup_{0\neq\bm\theta_h\atop\in\mathcal{Y}_h}\frac{(\bm\theta_h,\nabla\cdot\bm\tau_h)_{\mathcal T_h}+\langle\alpha_2(\bm\Pi_{\ell}^{\partial}\bm{\theta}_h,\bm\Pi_{\ell}^{\partial}\bm\phi_h-\widehat {\bm\phi}_h\rangle_{\partial\mathcal T_h}+\langle\ \bm \theta_h\cdot \bm t ,\Pi_{k-1}^\partial q_h-\widehat{q}_h\rangle_{\partial\mathcal{T}_h}}{\|\mathfrak{h}^{-1}\bm\theta_h\|_{\mathcal{T}_h}}\\
		%&\text{(using \eqref{2}, \eqref{6} and Cauchy-Schwarz’s inequality)}\\
		&\leq C\sup_{\bm 0\neq\bm\theta_h\in\mathcal{Y}_h}\frac{\mathfrak{b}_h( \bm 0,\bm 0,\bm \theta_h, \bm 0, 0, 0;\bm \tau_h, \bm S_h,\bm\phi_h,\widehat{\bm \phi}_h,q_h,\widehat{q}_h)}{\|\nabla\bm\theta_h\|_{\mathcal{T}_h}}\\&\qquad+C(\|\bm \tau_h\|_{\mathcal{T}_h}+\|\alpha_2^\frac{1}{2}(\bm\Pi_{\ell}^{\partial}\bm{\phi}_h-\widehat {\bm\phi}_h)\|_{\partial\mathcal{T}_h}+\|\alpha_3^\frac{1}{2}(\Pi_{k-1}^{\partial} q_h-\widehat{q}_h)\|_{\partial\mathcal{T}_h}) \\
		&\leq C\sup_{\color{black}\bm 0\neq (\bm \sigma_h, \bm R_h,\bm \theta_h, \widehat{\bm \theta}_h,p_h,\widehat{p}_h)\atop\in \Sigma_h \times\mathcal{S}_h\times \mathcal{Y}_h \times \widehat{\mathcal{Y}}_h\times\mathcal{Q}_h\times\widehat{\mathcal{Q}}_h}\frac{\mathfrak{b}_h(\bm \sigma_h, \bm R_h,\bm \theta_h, \widehat{\bm \theta}_h,p_h,\widehat{p}_h;\bm \tau_h, \bm S_h,\bm \phi_h, \widehat{\bm \phi}_h,q_h,\widehat{q}_h)}{\|(\bm \sigma_h, \bm R_h,\bm \theta_h, \widehat{\bm \theta}_h,p_h,\widehat{p}_h)\|_{\mathfrak{b}_h}}\\
		&\qquad+C(\|\bm \tau_h\|_{\mathcal{T}_h}+\|\alpha_2^\frac{1}{2}(\bm\Pi_{\ell}^{\partial}\bm{\phi}_h-\widehat {\bm\phi}_h)\|_{\partial\mathcal{T}_h}+\|\alpha_3^\frac{1}{2}(\Pi_{k-1}^{\partial} q_h-\widehat{q}_h)\|_{\partial\mathcal{T}_h}).
	\end{align*}
	%		For the term $E_6$, it holds
	%		\begin{align*}
		%			E_6&=\langle\alpha_2(\bm\Pi_{\ell}^{\partial}\bm \Pi_{k}^o \bm \theta-\bm\Pi_{\ell}^{\partial}\bm \theta),\bm\Pi_{\ell}^{\partial}\bm \phi_h-\widehat{\bm \phi}_h\rangle_{\partial\mathcal{T}_h} \\
		%			&\leq\|\alpha_2^\frac{1}{2}(\bm \Pi_{k}^o \bm \theta-\bm \theta)\|_{\partial\mathcal{T}_h}\|\alpha_2^\frac{1}{2}(\bm\Pi_{\ell}^{\partial}\bm{\phi}_h-\widehat {\bm\phi}_h)\|^2_{\partial\mathcal{T}_h} \\
		%			&\leq C h^k\|\bm\theta\|_{k+1}\|\alpha_2^\frac{1}{2}(\bm\Pi_{\ell}^{\partial}\bm{\phi}_h-\widehat {\bm\phi}_h)\|_{\partial\mathcal{T}_h}.
		%		\end{align*}
	As for the term $E_7$, we use the same arguments as $E_6$ to get
	\begin{align*}
		E_7&=\langle\alpha_3(\Pi_{k-1}^{\partial}\Pi_k^o p-\Pi_{k-1}^{\partial}p),\Pi_{k-1}^{\partial}q_h-\widehat{q}_h\rangle_{\partial\mathcal{T}_h} \\
		&\leq Ch^k t\|p\|_{k+1}\|\alpha_3^\frac{1}{2}(\Pi_{k-1}^{\partial}q_h-\widehat {q}_h)\|_{\partial\mathcal{T}_h}.
	\end{align*}
	%		For the term $E_8$, We use \eqref{HDGpoincare} and \cref{theorem1} and Cauchy-Schwartz’s inequality to get
	%		\begin{align*}
		%			E_8=(\bm L-\bm L_h,\bm\phi_h)_{\mathcal{T}_h}&\leq\|\bm L-\bm L_h\|_{\mathcal{T}_h}\|\bm\phi_h\|_{\mathcal{T}_h}\\
		%			&\leq Ch^k\|r\|_{k+1}(\|\nabla\bm\phi_h\|_{\mathcal{T}_h}+\|h_K^{-\frac{1}{2}}(\bm\Pi_{\ell}^{\partial}\bm{\phi}_h-\widehat {\bm\phi}_h)\|_{\partial\mathcal{T}_h}).
		%		\end{align*}
	Combining the above estimates of $E_i$, we get
	\begin{align*}
		&\quad \ \mathfrak{b}_h( \bm \xi_{\bm \sigma},\bm e_{\bm R},\bm \xi_{\bm \theta}, \bm \xi_{\widehat{\bm \theta}}, \xi_{p}, \xi_{\widehat p};\bm \tau_h, \bm S_h,\bm\phi_h,\widehat{\bm \phi}_h,q_h,\widehat{q}_h)\\
		& \leq C h^k(\|r\|_{k+1}+\|\bm \sigma \|_k+t^{-1}\|\bm R \|_k+\|\bm \theta\|_{k+1}+t\|p\|_{k+1})\\
		& \qquad \cdot (\|\bm \tau_h\|_{\mathcal{T}_h}+\|\nabla\bm\phi_h\|_{\mathcal{T}_h}+\|\alpha_2^\frac{1}{2}(\bm\Pi_{\ell}^{\partial}\bm{\phi}_h-\widehat {\bm\phi}_h)\|_{\partial\mathcal{T}_h}+t\|\nabla^{\perp}q_h\|_{\mathcal{T}_h}+\|\alpha_3^\frac{1}{2}(\Pi_{k-1}^{\partial} q_h-\widehat{q}_h)\|_{\partial\mathcal{T}_h})\\
		&\quad \ +\sup_{\bm 0\neq (\bm \sigma_h, \bm R_h,\bm \theta_h, \widehat{\bm \theta}_h,p_h,\widehat{p}_h)\atop\in \Sigma_h \times\mathcal{S}_h\times \mathcal{Y}_h \times \widehat{\mathcal{Y}}_h\times\mathcal{Q}_h\times\widehat{\mathcal{Q}}_h}\frac{\mathfrak{b}_h(\bm \sigma_h, \bm R_h,\bm \theta_h, \widehat{\bm \theta}_h,p_h,\widehat{p}_h;\bm \tau_h, \bm S_h,\bm \phi_h, \widehat{\bm \phi}_h,q_h,\widehat{q}_h)}{\|(\bm \sigma_h, \bm R_h,\bm \theta_h, \widehat{\bm \theta}_h,p_h,\widehat{p}_h)\|_{\mathfrak{b}_h}}\cdot C h^k\|\bm \theta \|_{k+1}.
	\end{align*}
	Finally, by discrete LBB condition \eqref{lbbinequality}, we could prove the desired conclusion.
\end{proof}\par
In light of \Cref{mainlemma} and the triangle inequality, we easily obtain the following error estimate:
\begin{theorem}\label{theorem2}
	Under the condition of \Cref{mainlemma}, it holds 
	\begin{equation}
		\begin{aligned}
			&\|\bm \sigma-\bm \sigma_h\|_{\mathcal{T}_h}+t^{-1}\|\bm R-\bm R_h\|_{\mathcal{T}_h}+\|\nabla\bm\theta-\nabla\bm\theta_h\|_{\mathcal{T}_h}+t\|\nabla^{\perp}p-\nabla^{\perp}p_h\|_{\mathcal{T}_h}\\
			&\qquad\leq C h^k(\|\bm f\|_{k-1}+\|g\|_{k-1}+\|r\|_{k+1}+\|\bm \sigma \|_k+\|\bm \theta\|_{k+1}+t\|p\|_{k+1}).
		\end{aligned}
	\end{equation}
\end{theorem}
\subsection{Main Result}
The third problem is also a normal Possion problem, so it's ready to get the prior estimate by using the same argument as \eqref{discrete stage1}.
\begin{theorem}\label{theorem3}
	Let $(\bm G,\omega)$ be the solution of \eqref{stage3} and let $(\bm G_h,\omega_h,\widehat{\omega}_h)$ be numerical the solution of \eqref{discrete stage3}, then we have 
	\begin{equation}
		\begin{aligned}
			&\|\bm G-\bm G_h\|_{\mathcal{T}_h}+\|\nabla \omega-\nabla \omega_h\|_{\mathcal{T}_h}\\
			&\qquad\leq Ch^k(\|\bm f\|_{k-1}+\|g\|_{k-1}+\|r\|_{k+1}+\|\bm \sigma \|_k+\|\bm \theta\|_{k+1}+t\|p\|_{k+1}+\|\omega\|_{k+1}).
		\end{aligned}
	\end{equation}
\end{theorem}\par
For the shear stress $\bm\gamma$,  we directly have the following estimate
\begin{equation}\label{streeestimate}
	\begin{aligned}
		t\|\bm\gamma-\bm\gamma_h\|_{\mathcal{T}_h}&\leq t\|\bm L-\bm L_h\|_{\mathcal{T}_h}+\lambda t^{-1}\|\bm R-\bm R_h\|_{\mathcal{T}_h}\\
		&\leq  C h^k(\|\bm f\|_{k-1}+\|g\|_{k-1}+\|r\|_{k+1}+\|\bm \sigma \|_k+\|\bm \theta\|_{k+1}+t\|p\|_{k+1}).
	\end{aligned}
\end{equation}
Under the \Cref{assumption of regularity} and combining  \Cref{theorem1,theorem2,theorem3} with \eqref{streeestimate}, the main result is straightforward.
\begin{theorem}[Main Result]\label{mainresult}
	Let $(\bm L,r,\bm\sigma,\bm R,\bm\theta,p,\bm G,\omega,\bm\gamma)$ be the solution of \eqref{continuousfor}. Let $(\bm L_h,r_h,\widehat{r}_h,\bm\sigma_h,\bm R_h,\bm\theta_h,\widehat{\bm \theta}_h,p_h,\widehat{p}_h,\bm G_h,\omega_h,\widehat{\omega}_h,\bm\gamma_h)$ be the numerical solution of \eqref{discretfor}, then we have 
	\begin{equation}
		\begin{aligned}
			&\|\bm L-\bm L_h\|_{\mathcal{T}_h}+\|\nabla r-\nabla r_h\|_{\mathcal{T}_h}+\|\bm \sigma-\bm \sigma_h\|_{\mathcal{T}_h}+t^{-1}\|\bm R-\bm R_h\|_{\mathcal{T}_h}+\|\nabla\bm\theta-\nabla\bm\theta_h\|_{\mathcal{T}_h}\\
			&\qquad+t\|\nabla^{\perp}p-\nabla^{\perp}p_h\|_{\mathcal{T}_h}+\|\bm G-\bm G_h\|_{\mathcal{T}_h}+\|\nabla \omega-\nabla \omega_h\|_{\mathcal{T}_h}+t\|\bm\gamma-\bm\gamma_h\|_{\mathcal{T}_h}\\
			&\qquad\leq C h^k(\|f\|_{k-1}+\|g\|_{k-1}).
		\end{aligned}
	\end{equation}
\end{theorem}
\section{${\rm L}_2$ Error Estimates}
We introduce the following co-problems to help us establish the ${\rm L}_2$ error estimates, here $ \bm e_{\bm L}=\bm L-\bm L_h, \bm e_{\bm \sigma}=\bm \sigma-\bm \sigma_h,\bm e_{\bm R}=\bm R-\bm R_h$.
\begin{subequations}\label{repro}
	\begin{align}
		(\bm\aleph,\bm M)-(\nabla \iota,\bm M)&=(\bm e_{\bm L},\bm M),&\forall \bm M \in \left[\mathrm{L}^2(\Omega)\right]^2, \\
		(\bm \aleph, \nabla \mu) &=(-\xi_r, \mu), &\forall \mu \in H_0^1(\Omega), \\
		\left(\mathcal{C}^{-1} \bm\zeta, \bm\tau\right)+(\nabla \bm\psi, \bm\tau) &=0, &\forall \bm \tau \in\left[\mathrm{L}^2(\Omega)\right]^{2 \times 2} \cap \mathbb{S}, \\
		(\lambda t^{-2} \bm \Upsilon,\bm S)
		+(\nabla^{\perp}\rho,\bm S)&=0,&\forall \bm S \in \left[\mathrm{L}^2(\Omega)\right]^2, \\
		(\bm\zeta, \nabla \phi)-(\bm\phi, \nabla^{\perp} \rho) &=(-\bm\xi_{\bm \theta}, \bm\phi), &\forall \bm\phi \in\left[H_0^1(\Omega)\right]^2, \\
		-(\bm\psi, \nabla^{\perp} q)-(\bm \Upsilon, \nabla^{\perp} q) &=0, &\forall q \in \widehat{H}^1(\Omega),
		\\
		(\bm\gimel,\bm H)+( \nabla \varkappa,\bm H)&=0,&\forall \bm H \in \left[\mathrm{L}^2(\Omega)\right]^2, \\
		(\bm \gimel, \nabla s) &=(-\xi_\omega,s), &\forall s \in H_0^1(\Omega).
	\end{align}
\end{subequations}\par
Wehn $\Omega$ is  a convex polygon, by \Cref{regularity result} we have
\begin{align}
	\|\bm \aleph\|_1+\|\iota\|_2&\leq C (\|\xi_{r}\|_0+\|\bm e_{\bm L}\|_0),\label{co-regularity1}\\
	\|\bm \zeta\|_1+t^{-1}\|\Upsilon\|_1+\|\bm\psi\|_2+\|\rho\|_1+t\|\rho\|_2&\leq C \|\bm\xi_{\bm\theta}\|_0,\label{co-regularity2}\\
	\|\bm \gimel\|_1+\|\varkappa\|_2&\leq C \|\xi_{\omega}\|_0.\label{co-regularity3}
\end{align}\par
\begin{theorem}\label{L2mainresult} When $\Omega$ is convex and 
	under the condition of \Cref{mainlemma}, it holds
	\begin{align}
		\|r-r_h\|_{\mathcal{T}_h}&\leq Ch^{k+1}(\|r\|_{k+1}+\|g\|_{k-1}) ,\label{l21}\\
		\|\bm\theta-\bm\theta_h\|_{\mathcal{T}_h}&\leq Ch^{k+1} (\|g\|_{k-1}+\|\bm f\|_{k-1}+\|r\|_{k+1}+\|\bm \sigma \|_k+\|\bm\theta\|_{k+1}+\|p\|_k+t\|p\|_{k+1}),\label{l22}\\
		\|\omega-\omega_h\|_{\mathcal{T}_h}& \leq Ch^{k+1}(\|g\|_{k-1}+\|\bm f\|_{k-1}+\|r\|_{k+1}+\|\bm \sigma \|_k\\
		&\qquad\qquad\qquad+\|\bm\theta\|_{k+1}+\|p\|_k+t\|p\|_{k+1}+\|\omega\|_{k+1}).\label{l23}\nonumber
	\end{align}
	Moreover, under the \Cref{assumption of regularity}, we have
	\begin{align}
		\|r-r_h\|_{\mathcal{T}_h}
		+\|\bm\theta-\bm\theta_h\|_{\mathcal{T}_h}
		+\|\omega-\omega_h\|_{\mathcal{T}_h}\le C h^{k+1}(\|f\|_{k-1}+\|g\|_{k-1}).
	\end{align}
\end{theorem}
\begin{proof}
	We only give a proof of \eqref{l21} and \eqref{l22}, because the proof of \eqref{l23} is similar with \eqref{l21}.\par
	\textbf{Proof of \eqref{l21}:} Using the relation of \eqref{repro} and the orthogonality of projections ${\rm L}^2$, we get
	\begin{align*}
		\|\xi_r\|^2_{\mathcal{T}_h}+\|\bm e_{\bm L}\|^2_{\mathcal{T}_h}&=(\bm e_{\bm L},\bm e_{\bm L})_{\mathcal{T}_h}+(\xi_r-\mathcal{I}_h(\xi_r,\xi_{\widehat{r}}),\xi_r)_{\mathcal{T}_h}+(\mathcal{I}_h(\xi_r,\xi_{\widehat{r}}),\xi_r)_{\mathcal{T}_h}\\
		&=(\xi_r-\mathcal{I}_h(\xi_r,\xi_{\widehat{r}}),\xi_r)_{\mathcal{T}_h}+(\bm\aleph,\bm e_{\bm L})-(\nabla \iota,\bm e_{\bm L})-(\bm \aleph, \nabla \mathcal{I}_h(\xi_r,\xi_{\widehat{r}}))\\
		&=(\xi_r-\mathcal{I}_h(\xi_r,\xi_{\widehat{r}}),\xi_r)_{\mathcal{T}_h}+(\bm\aleph-\bm\Pi^o_{0}\bm\aleph,\bm e_{\bm L})_{\mathcal{T}_h}+(\bm\Pi^o_{0}\bm\aleph,\bm e_{\bm L})_{\mathcal{T}_h}\\
		&\qquad-(\nabla \iota,\bm e_{\bm L})_{\mathcal{T}_h}-(\bm \aleph-\bm\Pi^o_{0}\bm\aleph, \nabla \mathcal{I}_h(\xi_r,\xi_{\widehat{r}}))_{\mathcal{T}_h}-(\bm\Pi^o_{0}\bm\aleph, \nabla \mathcal{I}_h(\xi_r,\xi_{\widehat{r}}))_{\mathcal{T}_h},
	\end{align*}
	then, we mainly consider the following three terms 
	\begin{align*}
		&(\bm\Pi^o_{0}\bm\aleph,\bm e_{\bm L})_{\mathcal{T}_h}-(\nabla\iota,\bm e_{\bm L})_{\mathcal{T}_h}-(\bm\Pi^o_{0}\bm\aleph, \nabla \mathcal{I}_h(\xi_r,\xi_{\widehat{r}}))_{\mathcal{T}_h}\\
		&\qquad(\text{by $\nabla\cdot\bm\Pi^o_{0}=0$ and \eqref{pro3}, \eqref{pro4} and integration by parts})\\
		&\qquad=(\bm\Pi^o_{0}\bm\aleph,\bm e_{\bm L})_{\mathcal{T}_h}+(\iota,\nabla\cdot\bm e_{\bm L})_{\mathcal{T}_h}-\langle\bm e_{\bm L}\cdot \bm n,\iota\rangle_{\partial{\mathcal{T}_h}}-\langle\bm\Pi^o_{0}\bm\aleph\cdot t,  \xi_{\widehat{r}}\rangle_{\partial{\mathcal{T}_h}},\\
		&\qquad=(\bm\Pi^o_{0}\bm\aleph,\bm L-\bm L_h)_{\mathcal{T}_h}+(\iota,\nabla\cdot(\bm L-\bm L_h))_{\mathcal{T}_h}\\
		&\qquad\qquad-\langle\bm L-\bm L_h\cdot \bm n,\iota\rangle_{\partial{\mathcal{T}_h}}-\langle\bm\Pi^o_{0}\bm\aleph\cdot t,  \Pi_{k-1}^{\partial} r-\widehat{r}_h\rangle_{\partial{\mathcal{T}_h}}\\
		&\qquad(\text{by the orthogonality of projections ${\rm L}^2$ and  integration by parts})\\
		&\qquad=(\bm\Pi^o_{0}\bm\aleph,\bm L)_{\mathcal{T}_h}-(\bm\Pi^o_{0}\bm\aleph,\bm L_h)_{\mathcal{T}_h}-(\nabla\iota,\bm L)_{\mathcal{T}_h}\\
		&\qquad\qquad-(\iota-\Pi^o_k\iota,\nabla\cdot\bm L_h)_{\mathcal{T}_h}-(\Pi^o_k\iota,\nabla\cdot\bm L_h)_{\mathcal{T}_h}+\langle\bm L_h\cdot \bm n,\iota-\Pi^\partial_{k-1}\iota\rangle_{\partial{\mathcal{T}_h}}\\
		&\qquad\qquad+\langle\bm L_h\cdot \bm n,\Pi^\partial_{k-1}\iota\rangle_{\partial{\mathcal{T}_h}}-(\bm\Pi^o_{0}\bm\aleph,\nabla r)_{\mathcal{T}_h}+\langle\bm\Pi^o_{0}\bm\aleph\cdot t,\widehat{r}_h\rangle_{\partial{\mathcal{T}_h}}\\
		&\qquad(\text{by the continuous form \eqref{continuous form} and the discrete form \eqref{discrete formulation}})\\
		&\qquad=(g,\iota-\Pi^o_k\iota)-\langle\bm L_h\cdot \bm n,\iota-\Pi^\partial_{k-1}\iota\rangle_{\partial{\mathcal{T}_h}}-(\iota-\Pi^o_k\iota,\nabla\cdot\bm L_h)_{\mathcal{T}_h}\\
		&\qquad\qquad-\langle\alpha_1(\Pi_{k-1}^{\partial}r_h-\widehat r_h),\Pi_{k-1}^{\partial}\Pi^o_k\iota-\Pi^\partial_{k-1}\iota\rangle_{\partial\mathcal T_h}.
	\end{align*}
	So the above equation leads to
	\begin{align*}
		\|\xi_r\|^2_{\mathcal{T}_h}+\|\bm e_{\bm L}\|^2_{\mathcal{T}_h}
		&=(\xi_r-\mathcal{I}_h(\xi_r,\xi_{\widehat{r}}),\xi_r)_{\mathcal{T}_h}+(\bm\aleph-\bm\Pi^o_{0}\bm\aleph,\bm e_{\bm L})_{\mathcal{T}_h}-(\bm \aleph-\bm\Pi^o_{0}\bm\aleph, \nabla \mathcal{I}_h(\xi_r,\xi_{\widehat{r}}))_{\mathcal{T}_h}\\
		&\qquad+(g-\Pi^o_{k-1}g,\iota-\Pi^o_k\iota)-\langle\bm L_h\cdot \bm n,\iota-\Pi^\partial_{k-1}\iota\rangle_{\partial{\mathcal{T}_h}}-(\iota-\Pi^o_k\iota,\nabla\cdot\bm L_h)_{\mathcal{T}_h}\\
		&\qquad+\langle\alpha_1(\Pi_{k-1}^{\partial}\xi_r-\xi_{\widehat{r}}),\Pi_{k-1}^{\partial}\Pi^o_k\iota-\Pi^\partial_{k-1}\iota\rangle_{\partial\mathcal T_h}\\
		&\qquad-\langle\alpha_1(\Pi_{k}^o r-r),\Pi_{k-1}^{\partial}\Pi^o_k\iota-\Pi^\partial_{k-1}\iota\rangle_{\partial\mathcal T_h}.
	\end{align*}
	We denote 
	\begin{align*}
		\mathbb{E}_1&:=|(\xi_r-\mathcal{I}_h(\xi_r,\xi_{\widehat{r}}),\xi_r)_{\mathcal{T}_h}-(\bm \aleph-\bm\Pi^o_{0}\bm\aleph, \nabla \mathcal{I}_h(\xi_r,\xi_{\widehat{r}}))_{\mathcal{T}_h}|\\
		\mathbb{E}_2&:=|(\bm\aleph-\bm\Pi^o_{0}\bm\aleph,\bm e_{\bm L})_{\mathcal{T}_h}+(g-\Pi^o_{k-1}g,\iota-\Pi^o_k\iota)\\
		&\qquad-\langle\bm L_h\cdot \bm n,\iota-\Pi^\partial_{k-1}\iota\rangle_{\partial{\mathcal{T}_h}}-(\iota-\Pi^o_k\iota,\nabla\cdot\bm L_h)_{\mathcal{T}_h}|\\
		\mathbb{E}_3&:=|\langle\alpha_1(\Pi_{k-1}^{\partial}\xi_r-\xi_{\widehat{r}}),\Pi_{k-1}^{\partial}\Pi^o_k\iota-\Pi^\partial_{k-1}\iota\rangle_{\partial\mathcal T_h}\\
		&\qquad-\langle\alpha_1(\Pi_{k}^o r-r),\Pi_{k-1}^{\partial}\Pi^o_k\iota-\Pi^\partial_{k-1}\iota\rangle_{\partial\mathcal T_h}|.
	\end{align*}\par
	By Cauchy-Schwartz's inequality, \eqref{5}, \eqref{pro1}, \eqref{pro2}, \eqref{co-regularity1} and \Cref{theorem1}, it holds
	\begin{align*}
		\mathbb{E}_1&\leq C h(\|\xi_r\|_{\mathcal{T}_h}+\|\bm\aleph\|_1)(\|\nabla \xi_r\|_{\mathcal{T}_h}+\|\mathfrak{h}^{-1/2}(\xi_r-\xi_{\widehat{r}})\|_{\partial\mathcal{T}_h})\\
		&\leq C h^{k+1}(\|\xi_r\|_{\mathcal{T}_h}+\|\bm e_{\bm L}\|_{\mathcal{T}_h})(\|r\|_{k+1}+\|g\|_{k-1}), \qquad(\text{by \eqref{co-regularity1} and \Cref{theorem1}}).
	\end{align*}
	By Cauchy-Schwartz's inequality and the orthogonality of projections ${\rm L}^2$ and \eqref{5} and \eqref{6}, it holds 
	\begin{align*}
		\mathbb{E}_2&\leq C h\|\bm\aleph\|_1\|\bm e_{\bm L}\|_{\mathcal{T}_h}
		+C h^{k+1}\|g\|_{k-1}\|\iota\|_2+Ch\|\iota\|_2\|\bm\xi_{\bm L}\|_{\mathcal{T}_h}\\
		&\leq ch^{k+1}(\|\xi_r\|_{\mathcal{T}_h}+\|\bm e_{\bm L}\|_{\mathcal{T}_h})(\|r\|_{k+1}+\|g\|_{k-1}), \qquad(\text{also by \eqref{co-regularity1} and \Cref{theorem1}}).
	\end{align*}
	Using the same argument as $\mathbb{E}_2$, we also get
	\begin{align*}
		\mathbb{E}_3&\leq C h\|\iota\|_2(h^k\|r\|_{k+1}+\|\alpha_1^{1/2}(\Pi_{k-1}^{\partial}\xi_r-\xi_{\widehat{r}})\|_{\partial\mathcal T_h})\\
		&\leq ch^{k+1}(\|\xi_r\|_{\mathcal{T}_h}+\|\bm e_{\bm L}\|_{\mathcal{T}_h})(\|r\|_{k+1}+\|g\|_{k-1}).
	\end{align*}
	So we can prove the desired conclusion
	\begin{align*}
		\|\xi_r\|_{\mathcal{T}_h}+\|\bm e_{\bm L}\|_{\mathcal{T}_h}\leq ch^{k+1}(\|r\|_{k+1}+\|g\|_{k-1})
	\end{align*}\par
	\textbf{Proof of \eqref{l22}:} We give the proof of \eqref{l22} which is nearly same as \eqref{l21}, but there are some details should be clarified. By \eqref{repro}, it holds
	\begin{align*}
		\|\xi_{\bm \theta}\|^2_{\mathcal{T}_h}&=(\xi_{\bm \theta}-\bm{\mathcal{I}}_h(\xi_{\bm\theta},\xi_{\widehat{\bm \theta}}),\xi_{\bm \theta})_{\mathcal{T}_h}+(\bm{\mathcal{I}}_h(\xi_{\bm\theta},\xi_{\widehat{\bm \theta}}),\xi_{\bm \theta})_{\mathcal{T}_h}\\
		&=(\xi_{\bm \theta}-\bm{\mathcal{I}}_h(\xi_{\bm\theta},\xi_{\widehat{\bm \theta}}),\xi_{\bm \theta})_{\mathcal{T}_h}+\left(\mathcal{C}^{-1} \bm\zeta, \bm e_{\bm\sigma}\right)_{\mathcal{T}_h}+(\nabla \bm\psi, \bm e_{\bm\sigma})+(\lambda t^{-2} \bm \Upsilon,\bm e_{\bm R})_{\mathcal{T}_h}\\
		&\qquad+(\nabla^{\perp}\rho,\bm e_{\bm R})_{\mathcal{T}_h}-(\bm\zeta, \nabla \bm{\mathcal{I}}_h(\xi_{\bm\theta},\xi_{\widehat{\bm \theta}}))_{\mathcal{T}_h}+(\bm{\mathcal{I}}_h(\xi_{\bm\theta},\xi_{\widehat{\bm \theta}}), \nabla^{\perp} \rho)_{\mathcal{T}_h}\\
		&\qquad-(\bm\psi, \nabla^{\perp} \mathcal{I}_h(\xi_p,\xi_{\widehat{p}}))_{\mathcal{T}_h}-(\bm \Upsilon, \nabla^{\perp} \mathcal{I}_h(\xi_p,\xi_{\widehat{p}}))_{\mathcal{T}_h} \\
		&(\text{by the orthogonality of projections ${\rm L}^2$ and  integration by parts})\\
		&=(\xi_{\bm \theta}-\bm{\mathcal{I}}_h(\xi_{\bm\theta},\xi_{\widehat{\bm \theta}}),\xi_{\bm \theta})_{\mathcal{T}_h}+\left(\mathcal{C}^{-1} \bm\zeta-\mathcal{C}^{-1}\bm\Pi^o_{0}\bm\zeta, \bm e_{\bm\sigma}\right)_{\mathcal{T}_h}+(\mathcal{C}^{-1}\Pi^o_{0}\bm\zeta, \bm e_{\bm\sigma})_{\mathcal{T}_h}\\
		&\qquad+(\nabla \bm\psi, \bm e_{\bm\sigma})_{\mathcal{T}_h}+(\lambda t^{-2}( \bm \Upsilon-\bm\Pi^o_0\bm \Upsilon),\bm e_{\bm R})_{\mathcal{T}_h}+(\lambda t^{-2} \bm\Pi^o_0\bm \Upsilon,\bm e_{\bm R})_{\mathcal{T}_h}\\
		&\qquad+(\nabla^{\perp}\rho,\bm e_{\bm R})_{\mathcal{T}_h}-(\bm\zeta-\bm\Pi^o_{0}\bm\zeta, \nabla \bm{\mathcal{I}}_h(\xi_{\bm\theta},\xi_{\widehat{\bm \theta}}))_{\mathcal{T}_h}-(\bm\Pi^o_{0}\bm\zeta, \nabla \bm{\mathcal{I}}_h(\xi_{\bm\theta},\xi_{\widehat{\bm \theta}}))_{\mathcal{T}_h}\\
		&\qquad+(\bm{\mathcal{I}}_h(\xi_{\bm\theta},\xi_{\widehat{\bm \theta}}), \nabla^{\perp} \rho)_{\mathcal{T}_h}-(\bm\psi, \nabla^{\perp} \mathcal{I}_h(\xi_p,\xi_{\widehat{p}}))_{\mathcal{T}_h}\\
		&\qquad-(\bm \Upsilon-\bm{\Pi}^o_{0}\bm\Upsilon, \nabla^{\perp} \mathcal{I}_h(\xi_p,\xi_{\widehat{p}}))_{\mathcal{T}_h}-(\bm{\Pi}^o_{0}\bm\Upsilon, \nabla^{\perp} \mathcal{I}_h(\xi_p,\xi_{\widehat{p}}))_{\mathcal{T}_h}.
	\end{align*}\par
	To deduce the desired estimate, we focus on the following terms
	\begin{align*}
		&(\mathcal{C}^{-1}\bm\Pi^o_{0}\bm\zeta, \bm e_{\bm\sigma})_{\mathcal{T}_h}+(\nabla \bm\psi, \bm e_{\bm\sigma})+(\lambda t^{-2} \bm\Pi^o_0\bm \Upsilon,\bm e_{\bm R})_{\mathcal{T}_h}+(\nabla^{\perp}\rho,\bm e_{\bm R})_{\mathcal{T}_h}-(\bm\Pi^o_{0}\bm\zeta, \nabla \bm{\mathcal{I}}_h(\xi_{\bm\theta},\xi_{\widehat{\bm \theta}}))_{\mathcal{T}_h}\\
		&\qquad+(\bm{\mathcal{I}}_h(\xi_{\bm\theta},\xi_{\widehat{\bm \theta}}), \nabla^{\perp} \rho)_{\mathcal{T}_h}-(\bm\psi, \nabla^{\perp} \mathcal{I}_h(\xi_p,\xi_{\widehat{p}}))_{\mathcal{T}_h}-(\bm{\Pi}^o_{0}\bm\Upsilon, \nabla^{\perp} \mathcal{I}_h(\xi_p,\xi_{\widehat{p}}))_{\mathcal{T}_h}\\
		&\qquad(\text{by $\nabla\cdot\bm\Pi^o_{0}=0$ and \eqref{pro3}, \eqref{pro4} and integration by parts})\\
		&\qquad =(\mathcal{C}^{-1}\bm\Pi^o_{0}\bm\zeta, \bm\sigma)_{\mathcal{T}_h}-(\mathcal{C}^{-1}\bm\Pi^o_{0}\bm\zeta, \bm\sigma_h)_{\mathcal{T}_h}+(\nabla \bm\psi,  \bm\sigma)_{\mathcal{T}_h}-(\nabla \bm\psi,\bm\sigma_h)_{\mathcal{T}_h}\\
		&\qquad\qquad+(\lambda t^{-2} \bm\Pi^o_0\bm \Upsilon,\bm R)_{\mathcal{T}_h}-(\lambda t^{-2} \bm\Pi^o_0\bm \Upsilon,\bm R_h)_{\mathcal{T}_h}+(\nabla^{\perp}\rho,\bm R)_{\mathcal{T}_h}-(\nabla^{\perp}\rho,\bm R_h)_{\mathcal{T}_h}\\
		&\qquad\qquad-\langle\bm\Pi^o_{0}\bm\zeta n, \bm\theta\rangle_{\partial\mathcal{T}_h}+\langle\bm\Pi^o_{0}\bm\zeta n, \widehat{\bm\theta}_h\rangle_{\partial\mathcal{T}_h}+(\bm{\mathcal{I}}_h(\xi_{\bm\theta},\xi_{\widehat{\bm \theta}}), \nabla^{\perp} \rho-\nabla^{\perp} \Pi^o_{k}\rho)_{\mathcal{T}_h}\\
		&\qquad\qquad+(\bm\theta, \nabla^{\perp} \Pi^o_{k}\rho)_{\mathcal{T}_h}-(\bm\theta_h, \nabla^{\perp} \Pi^o_{k}\rho)_{\mathcal{T}_h}-(\bm\psi-\bm\Pi^o_{k}\bm\psi, \nabla^{\perp} \mathcal{I}_h(\xi_p,\xi_{\widehat{p}}))_{\mathcal{T}_h}\\
		&\qquad\qquad-(\bm\Pi^o_{k}\bm\psi, \nabla^{\perp}p)_{\mathcal{T}_h}-(\nabla\times\bm\Pi^o_{k}\bm\psi, p_h)_{\mathcal{T}_h}-\langle\bm\Pi^o_{k}\bm\psi\cdot t,\Pi^\partial_{k-1}p-p\rangle_{\partial{\mathcal{T}_h}}\\
		&\qquad\qquad+\langle\bm\Pi^o_{k}\bm\psi\cdot t,\widehat{p}_h\rangle_{\partial{\mathcal{T}_h}}-\langle\bm{\Pi}^o_{0}\bm\Upsilon\cdot t,p\rangle_{\partial\mathcal{T}_h}+\langle\bm{\Pi}^o_{0}\bm\Upsilon\cdot t,\widehat{p}_h\rangle_{\partial\mathcal{T}_h}\\
		&\qquad(\text{by the orthogonality of projections ${\rm L}^2$ and  integration by parts})\\
		&\qquad=(\mathcal{C}^{-1}\bm\sigma, \Pi^o_{0}\bm\zeta)_{\mathcal{T}_h}+(\bm\sigma,\nabla \bm\psi  )_{\mathcal{T}_h}+(\lambda t^{-2}\bm R,\bm\Pi^o_0\bm \Upsilon)_{\mathcal{T}_h}+(\bm R,\nabla^{\perp}\rho)_{\mathcal{T}_h}-(\bm\Pi^o_{0}\bm\zeta,\nabla\theta)_{\mathcal{T}_h}\\
		&\qquad\qquad+(\bm\theta,\nabla^{\perp}\rho)_{\mathcal{T}_h}-(\bm\psi,\nabla^{\perp}p)-(\bm{\Pi}^o_{0}\bm\Upsilon,\nabla^{\perp}p)_{\mathcal{T}_h}-(\bm\theta,\nabla^{\perp}\rho-\nabla^{\perp} \Pi^o_{k}\rho)_{\mathcal{T}_h}\\
		&\qquad\qquad-(\bm\Pi^o_{k}\bm\psi-\bm\psi,\nabla^{\perp}p)_{\mathcal{T}_h}-\langle\bm\psi, \bm\sigma_h n\rangle_{\partial{\mathcal{T}_h}}-\langle \rho, \bm R_h t\rangle_{\partial{\mathcal{T}_h}}\\
		&\qquad\qquad-(\mathcal{C}^{-1}\bm\Pi^o_{0}\bm\zeta, \bm\sigma_h)_{\mathcal{T}_h}+(\bm\Pi^o_k\bm\psi,\nabla\cdot\bm \sigma_h)_{\mathcal{T}_h}-(\lambda t^{-2} \bm\Pi^o_0\bm \Upsilon,\bm R_h)_{\mathcal{T}_h}+(\Pi^o_k\rho,\nabla\times\bm R_h)\\
		&\qquad\qquad-(\nabla\cdot\bm\Pi^o_0\bm\zeta,\bm\theta_h)_{\mathcal{T}_h}+\langle\bm\Pi^o_{0}\bm\zeta n, \widehat{\bm\theta}_h\rangle_{\partial\mathcal{T}_h}-(\bm\theta_h, \nabla^{\perp} \Pi^o_{k}\rho)_{\mathcal{T}_h}-(\nabla\times\bm\Pi^o_{k}\bm\psi, p_h)_{\mathcal{T}_h}\\
		&\qquad\qquad+\langle\bm\Pi^o_{k}\bm\psi\cdot t,\widehat{p}_h\rangle_{\partial{\mathcal{T}_h}}-(\nabla\times\bm{\Pi}^o_{0}\bm\Upsilon,p_h)_{\mathcal{T}_h}+\langle\bm{\Pi}^o_{0}\bm\Upsilon\cdot t,\widehat{p}_h\rangle_{\partial\mathcal{T}_h}-\langle\bm\Pi^o_{k}\bm\psi\cdot t,\Pi^\partial_{k-1}p-p\rangle_{\partial{\mathcal{T}_h}}\\
		&\qquad\qquad+(\bm{\mathcal{I}}_h(\xi_{\bm\theta},\xi_{\widehat{\bm \theta}}), \nabla^{\perp} \rho-\nabla^{\perp} \Pi^o_{k}\rho)_{\mathcal{T}_h}-(\bm\psi-\bm\Pi^o_{k}\bm\psi, \nabla^{\perp} \mathcal{I}_h(\xi_p,\xi_{\widehat{p}}))_{\mathcal{T}_h}\\
		&\qquad(\text{by the continuous form \eqref{continuous form} and the discrete form \eqref{discrete formulation}})\\
		&\qquad=(\bm L+\bm f,\bm\psi)_{\mathcal{T}_h}-(\bm L_h+\bm f,\bm\Pi^o_k\bm\psi)_{\mathcal{T}_h}+\langle\bm\theta_h\cdot t,\Pi^\partial_{k-1}\rho-\Pi^o_k\rho\rangle_{\partial{\mathcal{T}_h}}\\
		&\qquad\qquad-\langle\bm\Pi^o_{k}\bm\psi\cdot t,\Pi^\partial_{k-1}p-p\rangle_{\partial{\mathcal{T}_h}}-(\bm\theta,\nabla^{\perp}\rho-\nabla^{\perp} \Pi^o_{k}\rho)_{\mathcal{T}_h}-(\bm\Pi^o_{k}\bm\psi-\bm\psi,\nabla^{\perp}p)_{\mathcal{T}_h}\\
		&\qquad\qquad+(\bm{\mathcal{I}}_h(\xi_{\bm\theta},\xi_{\widehat{\bm \theta}}), \nabla^{\perp} \rho-\nabla^{\perp} \Pi^o_{k}\rho)_{\mathcal{T}_h}-(\bm\psi-\bm\Pi^o_{k}\bm\psi, \nabla^{\perp} \mathcal{I}_h(\xi_p,\xi_{\widehat{p}}))_{\mathcal{T}_h}\\
		&\qquad\qquad+\langle\alpha_2(\bm\Pi_{\ell}^{\partial}\bm{\theta}_h-\widehat {\bm\theta}_h),\bm\Pi_{\ell}^{\partial}\bm\Pi^o_k\bm\psi-\bm\Pi_{\ell}^{\partial}\bm\psi\rangle_{\partial\mathcal T_h}+\langle\alpha_3(\Pi_{k-1}^{\partial}p_h-\widehat p_h),\Pi_{k-1}^{\partial}\Pi^o_k\rho-\Pi_{k-1}^{\partial}\rho\rangle_{\partial{\mathcal{T}_h}}
	\end{align*}
	so, we can prove
	\begin{align*}
		\|\xi_{\bm \theta}\|^2_{\mathcal{T}_h}&=(\xi_{\bm \theta}-\bm{\mathcal{I}}_h(\xi_{\bm\theta},\xi_{\widehat{\bm \theta}}),\xi_{\bm \theta})_{\mathcal{T}_h}+\left(\mathcal{C}^{-1} \bm\zeta-\mathcal{C}^{-1}\Pi^o_{0}\bm\zeta, \bm e_{\bm\sigma}\right)_{\mathcal{T}_h}+(\lambda t^{-2}( \bm \Upsilon-\Pi^o_0\bm \Upsilon),\bm e_{\bm R})_{\mathcal{T}_h}\\
		&\qquad-(\bm\zeta-\bm\Pi^o_{0}\bm\zeta, \nabla \bm{\mathcal{I}}_h(\xi_{\bm\theta},\xi_{\widehat{\bm \theta}}))_{\mathcal{T}_h}-(\bm \Upsilon-\bm{\Pi}^o_{0}\bm\Upsilon, \nabla^{\perp} \mathcal{I}_h(\xi_p,\xi_{\widehat{p}}))_{\mathcal{T}_h}+(\bm L+\bm f,\bm\psi)_{\mathcal{T}_h}\\
		&\qquad-(\bm L_h+\bm f,\bm\Pi^o_k\bm\psi)_{\mathcal{T}_h}+\langle\bm\theta_h\cdot t,\Pi^\partial_{k-1}\rho-\rho\rangle_{\partial{\mathcal{T}_h}}-\langle\bm\Pi^o_{k}\bm\psi\cdot t,\Pi^\partial_{k-1}p-p\rangle_{\partial{\mathcal{T}_h}}\\
		&\qquad-(\bm\theta,\nabla^{\perp}\rho-\nabla^{\perp} \Pi^o_{k}\rho)_{\mathcal{T}_h}-(\bm\Pi^o_{k}\bm\psi-\bm\psi,\nabla^{\perp}p)_{\mathcal{T}_h}+(\bm{\mathcal{I}}_h(\xi_{\bm\theta},\xi_{\widehat{\bm \theta}}), \nabla^{\perp} \rho-\nabla^{\perp} \Pi^o_{k}\rho)_{\mathcal{T}_h}\\
		&\qquad-(\bm\psi-\bm\Pi^o_{k}\bm\psi, \nabla^{\perp} \mathcal{I}_h(\xi_p,\xi_{\widehat{p}}))_{\mathcal{T}_h}+\langle\alpha_2(\bm\Pi_{\ell}^{\partial}\bm{\theta}_h-\widehat {\bm\theta}_h),\bm\Pi_{\ell}^{\partial}\bm\Pi^o_k\bm\psi-\bm\Pi_{\ell}^{\partial}\bm\psi\rangle_{\partial\mathcal T_h}\\
		&\qquad+\langle\alpha_3(\Pi_{k-1}^{\partial}p_h-\widehat p_h),\Pi_{k-1}^{\partial}\Pi^o_k\rho-\Pi_{k-1}^{\partial}\rho\rangle_{\partial{\mathcal{T}_h}}.
	\end{align*}
	Likely we also denote 
	\begin{align*}
		\mathbb{E}_4&:=|(\xi_{\bm \theta}-\bm{\mathcal{I}}_h(\xi_{\bm\theta},\xi_{\widehat{\bm \theta}}),\xi_{\bm \theta})_{\mathcal{T}_h}-(\bm \Upsilon-\bm{\Pi}^o_{0}\bm\Upsilon, \nabla^{\perp} \mathcal{I}_h(\xi_p,\xi_{\widehat{p}}))_{\mathcal{T}_h}-(\bm\zeta-\bm\Pi^o_{0}\bm\zeta, \nabla \bm{\mathcal{I}}_h(\xi_{\bm\theta},\xi_{\widehat{\bm \theta}}))_{\mathcal{T}_h}\\
		&\qquad+(\bm{\mathcal{I}}_h(\xi_{\bm\theta},\xi_{\widehat{\bm \theta}}), \nabla^{\perp} \rho-\nabla^{\perp} \Pi^o_{k}\rho)_{\mathcal{T}_h}-(\bm\psi-\bm\Pi^o_{k}\bm\psi, \nabla^{\perp} \mathcal{I}_h(\xi_p,\xi_{\widehat{p}}))_{\mathcal{T}_h}|\\
		\mathbb{E}_5&:=|(\bm L+\bm f,\bm\psi)_{\mathcal{T}_h}-(\bm L_h+\bm f,\bm\Pi^o_k\bm\psi)_{\mathcal{T}_h}+\left(\mathcal{C}^{-1} \bm\zeta-\mathcal{C}^{-1}\Pi^o_{0}\bm\zeta, \bm e_{\bm\sigma}\right)_{\mathcal{T}_h}\\
		&\qquad+(\lambda t^{-2}( \bm \Upsilon-\Pi^o_0\bm \Upsilon),\bm e_{\bm R})_{\mathcal{T}_h}-(\bm\theta,\nabla^{\perp}\rho-\nabla^{\perp} \Pi^o_{k}\rho)_{\mathcal{T}_h}-(\bm\Pi^o_{k}\bm\psi-\bm\psi,\nabla^{\perp}p)_{\mathcal{T}_h}|\\
		\mathbb{E}_6&:=|\langle\bm\theta_h\cdot t,\Pi^\partial_{k-1}\rho-\rho\rangle_{\partial{\mathcal{T}_h}}-\langle\bm\Pi^o_{k}\bm\psi\cdot t,\Pi^\partial_{k-1}p-p\rangle_{\partial{\mathcal{T}_h}}|\\
		\mathbb{E}_7&:=|\langle\alpha_2(\bm\Pi_{\ell}^{\partial}\bm{\theta}_h-\widehat {\bm\theta}_h),\bm\Pi_{\ell}^{\partial}\bm\Pi^o_k\bm\psi-\bm\Pi_{\ell}^{\partial}\bm\psi\rangle_{\partial\mathcal T_h}+\langle\alpha_3(\Pi_{k-1}^{\partial}p_h-\widehat p_h),\Pi_{k-1}^{\partial}\Pi^o_k\rho-\Pi_{k-1}^{\partial}\rho\rangle_{\partial{\mathcal{T}_h}}|.
	\end{align*}\par
	By Cauchy-Schwartz's inequality, \eqref{5}, \eqref{pro1}, \eqref{pro2}, \eqref{co-regularity2}, \Cref{theorem2} and HDG-Poincare's inequality \eqref{HDGpoincare}, it holds
	\begin{align*}
		\mathbb{E}_4&\leq C h(\|\xi_{\bm \theta}\|_{\mathcal{T}_h}+\|\rho\|_1+\|\bm\zeta\|_1)(\|\nabla\xi_{\bm \theta}\|_{\mathcal{T}_h}+\|\mathfrak{h}^{-1/2}(\bm\Pi^\partial_{\ell}\xi_{\bm\theta}-\xi_{\widehat{\bm \theta}})\|_{\mathcal{T}_h})\\
		&\qquad+C h t^{-1}\|\bm\Upsilon\|_1\cdot t(\|\nabla\xi_p\|_{\mathcal{T}_h}+\|\mathfrak{h}^{-1/2}(\xi_p-\xi_{\widehat{p}})\|_{\mathcal{T}_h})\\
		&\qquad+C h\|\bm\psi\|_2(\|\mathfrak{h}\nabla\xi_p\|_{\mathcal{T}_h}+\|\mathfrak{h}^{1/2}(\xi_p-\xi_{\widehat{p}})\|_{\mathcal{T}_h})\\
		&\leq Ch^{k+1}\|\xi_{\bm \theta}\|_{\mathcal{T}_h}(\|r\|_{k+1}+|\bm f|_{k-1}+\|\bm \sigma \|_k+\|\bm \theta\|_{k+1}+t\|p\|_{k+1})\\
		&(\text{using the same argument for $E_6$ and by \eqref{co-regularity2} and \Cref{theorem2}}).
	\end{align*}
	We note $(\bm\Pi^o_{k-1}\bm L+\bm\Pi^o_{k-1}\bm f, \bm\psi-\bm\Pi^o_k\bm\psi)=0$ and by Cauchy-Schwartz's inequality, the orthogonality of projections ${\rm L}^2$, \eqref{5}, \eqref{key} and \eqref{l21}, it holds 
	\begin{align*}
		\mathbb{E}_5&\leq C h^{k+1}(\|\bm\psi\|_2+\|\rho\|_1)(\|r\|_{k+1}+\|g\|_{k-1}+|\bm L|_{k}+\|\bm f\|_{k-1}+\|\bm\theta\|_{k+1}+\|p\|_k)\\
		&\qquad+C h\|\bm\zeta\|_1\|\bm e_{\bm\sigma}\|_0+C h\|\bm \Upsilon\|_1\|\bm e_{\bm R}\|_0\\
		&\leq C h^{k+1}\|\xi_{\bm \theta}\|_{\mathcal{T}_h}(\|g\|_{k-1}+\|\bm f\|_{k-1}+\|r\|_{k+1}+|\bm L|_{k}+\|\bm \sigma \|_k+t^{-1}\|\bm R \|_k+\|\bm\theta\|_{k+1}+\|p\|_k).
	\end{align*}
	By the orthogonality of projections ${\rm L}^2$, it holds
	\begin{align*}
		\mathbb{E}_6&=|-\langle(\xi_{\bm\theta}-\bm\Pi^o_{k-1}\xi_{\bm\theta})\cdot t,\Pi^\partial_{k-1}\rho-\rho\rangle_{\partial{\mathcal{T}_h}}+\langle(\bm\theta-\bm\Pi^o_{k}\bm\theta)\cdot t,\Pi^\partial_{k-1}\rho-\rho\rangle_{\partial{\mathcal{T}_h}}\\
		&\qquad-\langle(\bm\Pi^o_{k}\bm\psi-\bm\psi)\cdot t,\Pi^\partial_{k-1}p-p\rangle_{\partial{\mathcal{T}_h}}|.
	\end{align*}
	So by \eqref{6} and \eqref{co-regularity2} and \eqref{mainlemma}
	\begin{align*}
		\mathbb{E}_6&\leq Ch\|\nabla\xi_{\bm\theta}\|_{\mathcal{T}_h}\|\rho\|_1+Ch^{k+1}\|\bm\theta\|_{k+1}\|\rho\|_{1}+Ch^{k+1}\|\bm\psi\|_{2}\|p\|_{k}\\
		&\leq C h^{k+1}\|\xi_{\bm\theta}\|_{\mathcal{T}_h}(\|r\|_{k+1}+|\bm f|_{k-1}+\|\bm \sigma \|_k+t^{-1}\|\bm R \|_k+\|\bm \theta\|_{k+1}+t\|p\|_{k+1}+\|p\|_{k}).
	\end{align*}
	By the orthogonality of projections ${\rm L}^2$, it holds
	\begin{align*}
		\mathbb{E}_7&=|\langle\alpha_2(\bm\Pi_{\ell}^{\partial}\xi_{\bm{\theta}}-\xi_{\widehat {\bm\theta}}),\bm\Pi_{\ell}^{\partial}\bm\Pi^o_k\bm\psi-\bm\Pi_{\ell}^{\partial}\bm\psi\rangle_{\partial\mathcal T_h}-\langle\alpha_2(\bm\Pi_{k}^{o}\bm{\theta}- \bm\theta),\bm\Pi_{\ell}^{\partial}\bm\Pi^o_k\bm\psi-\bm\Pi_{\ell}^{\partial}\bm\psi\rangle_{\partial\mathcal T_h}\\
		&\qquad+\langle\alpha_3(\Pi_{k-1}^{\partial}\xi_{p}-\xi_{\widehat p}),\Pi_{k-1}^{\partial}\Pi^o_k\rho-\Pi_{k-1}^{\partial}\rho\rangle_{\partial{\mathcal{T}_h}}-\langle\alpha_3(\Pi_{k}^{o}p-p),\Pi_{k-1}^{\partial}\Pi^o_k\rho-\Pi_{k-1}^{\partial}\rho\rangle_{\partial{\mathcal{T}_h}}|.
	\end{align*}
	so we get by \eqref{5}, \eqref{6} and \Cref{mainlemma}
	\begin{align*}
		\mathbb{E}_7&\leq C h\|\bm\psi\|_2\|\alpha_2^{1/2}(\bm\Pi_{\ell}^{\partial}\xi_{\bm{\theta}}-\xi_{\widehat {\bm\theta}})\|_{\partial\mathcal{T}_h}+Ch^{k+1}\|\bm\theta\|_{k+1}\|\bm\psi\|_2\\
		&\qquad+Ch(\|\rho\|_1+t\|\rho\|_2)\|\alpha_3^{1/2}(\Pi_{k-1}^{\partial}\xi_{p}-\xi_{\widehat p})\|_{\partial\mathcal{T}_h}+Ch^{k+1}(\|p\|_k+t\|p\|_{k+1})(\|\rho\|_1+t\|\rho\|_2)\\
		&\leq Ch^{k+1}\|\xi_{\bm\theta}\|_{\mathcal{T}_h}(\|r\|_{k+1}+|\bm f|_{k-1}+\|\bm \sigma \|_k+t^{-1}\|\bm R \|_k+\|\bm \theta\|_{k+1}+t\|p\|_{k+1}+\|p\|_{k}).
	\end{align*}\par
	Combining the estimate of $\mathbb{E}_i$, we immediately derive \eqref{l22}, thus we completes the proof.
\end{proof}
\section{Implementation}\label{implemention}

%\subsection{Implementation Notes}
For \eqref{discrete stage1} and \eqref{discrete stage3}, we arrange the interior unknowns and boundary unknowns separately, then we can get the following linear system:
\begin{align}\label{51}
\begin{pmatrix}
	A_{11} &A_{12} \\
	A_{12}^T &-A_{22}
\end{pmatrix}
\begin{pmatrix}
	x_1 \\
	x_2
\end{pmatrix}
=
\begin{pmatrix}
	b_1 \\
	b_2
\end{pmatrix},
\end{align}
where $A_{11}$ is piecewise diagonal, that means the inverse of $A_{11}$ can be obtained efficiently. Therefore, instead of solving \eqref{51}, we are going to solve the Schur complement system
\begin{align} \label{52}
(A_{22}+A_{12}^TA_{11}^{-1}A_{12})x_2=-b_2+A_{12}^TA_{11}^{-1}b_1,
\end{align}
which is an SPD system, and any AMG solver or AMG-preconditioner solver can solve this system efficiently.
After solving \eqref{52}, we obtain $x_1$ by
\begin{align}\label{53}
x_1=A_{11}^{-1}(b_1-A_{12}x_2).
\end{align}
For \eqref{discrete stage2}, we still arrange the interior unknowns and boundary unknowns separately. Therefore, we still have the system of form \eqref{51}, and we also turn to solve \eqref{52} and \eqref{53}, the different is that here \eqref{52} is not an SPD system, the system \eqref{52} is a saddle-point system with the form
\begin{align}\label{54}
\begin{pmatrix}
	B_{11} &B_{12} \\
	B_{12}^T &-B_{22}
\end{pmatrix}
\begin{pmatrix}
	\widehat{\theta} \\
	\widehat p
\end{pmatrix}
=
\begin{pmatrix}
	c_1 \\
	c_2
\end{pmatrix},
\end{align}
This time the inverse of $B_{11}$ can not be obtained efficiently. Still, we can solve the system 
\begin{align} \label{55}
(B_{22}+B_{12}^TB_{11}^{-1}B_{12})\widehat p=-c_2+B_{12}^TB_{11}^{-1}c_1,
\end{align}
by iterative methods, Conjugate Gradient (CG) method for example. Remember that we do not need to compute $B_{11}^{-1}$, but we can use AMG-preconditioner solver to get matrix-vector multiplication $B_{11}^{-1}v$ for any vector $v$ that matches the dimensions.

At last we simply use the following formulation to recover $\widehat{\theta}$:
\begin{align}\nonumber
\widehat \theta=B_{11}^{-1}(c_1-B_{12}\widehat p).
\end{align}

We notice that in all steps of our calculation, the time is mainly spent on solving the system \eqref{55}.

\section{Numerical experiments}\label{section6}
This section provides some numerical results to verify the performance of the HDG scheme. All examples are coded in C++ with the library Eigen\cite{eigenweb} and the library Hypre\cite{10.1007/3-540-47789-6_66}.\par
\subsection{Numerical Results}\label{analytical solution}
We compute a square plate with continuous solution to show the numerical results. This solution is taken from \cite{HU2008464}. The domain $\Omega$ is simple unit square $(0, 1)^2$, the corresponding parameters are taken as $E = 1.0, \nu= 0.3$ and $\kappa=\frac 5 6$. Also, we mainly consider the hard clamped boundary. The continuous solution $\bm\theta,\omega$ is of the form
{%\scriptsize
\begin{equation}\nonumber
	\left\{
	\begin{aligned}
		&\bm\theta=(100y^3(y-1)^3x^2(x-1)^2(2x-1), 100x^3(x-1)^3y^2(y-1)^2(2y-1))^T, \\
		&\omega=100\bigg(\frac{1}{3}x^3(x-1)^3y^3(y-1)^3-\frac{2t^2}{5(1-\nu)}[y^3(y-1)^3x(x-1)(5x^2-5x+1)\\ 
		&\quad \ +x^3(x-1)^3y(y-1)(5y^2-5y+1)] \bigg).
	\end{aligned}
	\right.
\end{equation}}
Therefore, the body force and transverse loading are
{%\footnotesize
\begin{equation}\nonumber
	\left\{
	\begin{aligned}
		&\bm f=(0,0)^T, \\
		&g = (x^3(x-1)^3(5y^2-5y+1)+y^3(y-1)^3(5x^2-5x+1)\\
		& \quad \ +x(x-1)y(y-1)(5x^2-5x+1)(5y^2-5y+1)).
	\end{aligned}
	\right.
\end{equation}}\par
We mainly compute three different thick plates with the pane thickness $t=1$, $t=0.1$ and $t=0.01$. We display the results in following tables. From the results we can see that our HDG scheme, with $k=1,2,3$, yields optimal convergence rates which are uniform with respect to the plane thickness $t$ as $(k+1)$-th order rate for $\|\bm\theta-\bm\theta_h\|_{\mathcal{T}_h}, \|\omega-\omega_h\|_{\mathcal{T}_h}$  and $k$-th order for $\|\bm\sigma-\bm\sigma_h\|_{\mathcal{T}_h}$, $\|\bm\gamma-\bm\gamma_h\|_{\mathcal{T}_h}$. These are conformable to the theoretical results of \Cref{mainresult} and \Cref{L2mainresult}. We denote ``Iter'' in the tables equals iterations of linear system \eqref{discrete stage2}. Here we use a Schur complement CG iteration method to solve the linear system.
\begin{figure}[H]
\centering
\subfigure[$4\times4$]{
	\includegraphics[width=0.46\linewidth]{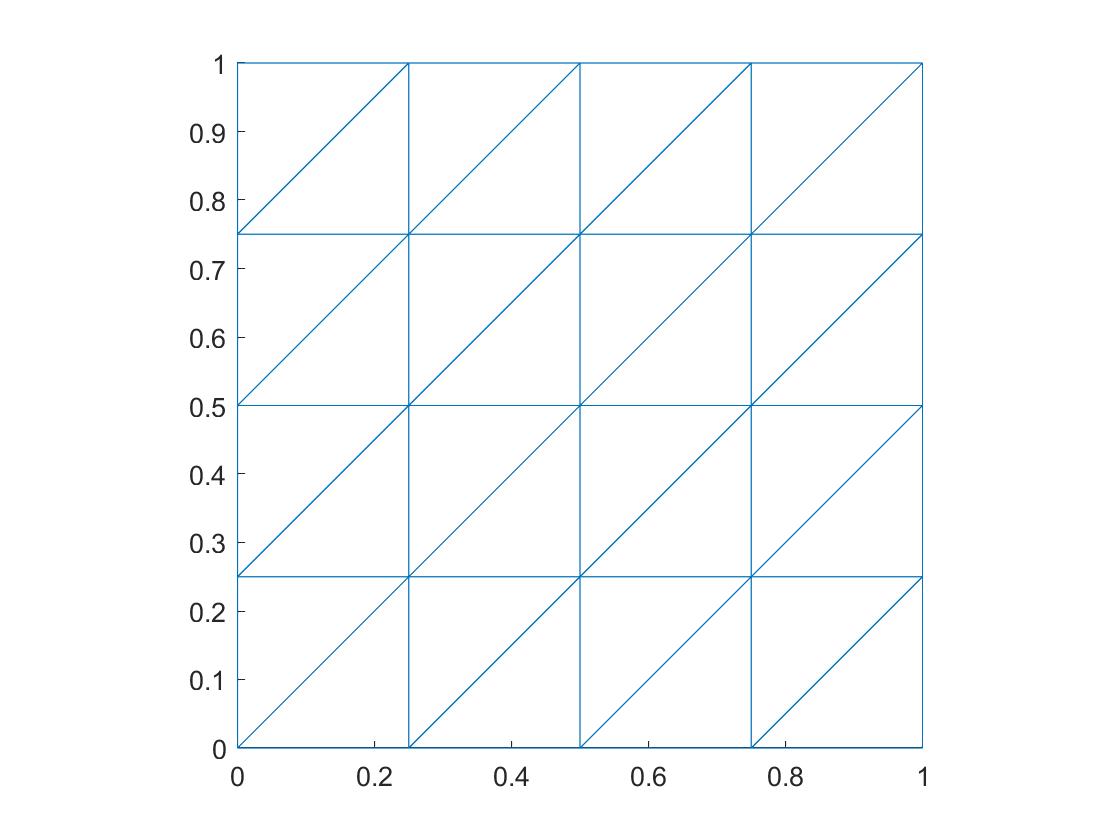} 
}\quad\subfigure[$8\times8$]{
	\includegraphics[width=0.46\linewidth]{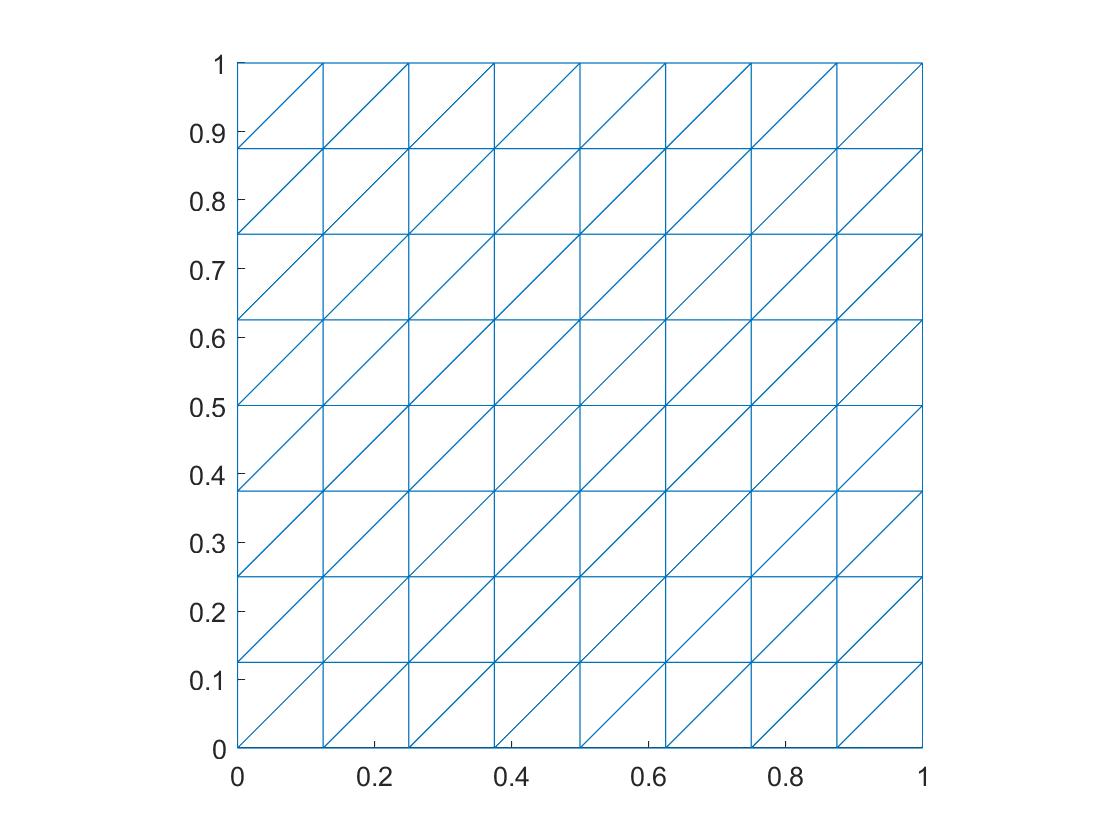}
}
\caption{Triangle meshes for square plate}
\end{figure} 

\begin{table}[H] \tiny
\centering
\resizebox{\textwidth}{!}{
	\begin{tabular}{c|c|c|c|c|c|c|c|c|c|c}
		\Xhline{1pt}
		\multirow{2}{*}{$k$}&\multirow{2}{*}{Mesh}&\multirow{2}{*}{Iter}&\multicolumn{2}{c|}{$\|\bm\theta-\bm\theta_h\|_{\mathcal{T}_h}$}&\multicolumn{2}{c|}{$t\|\bm\gamma-\bm\gamma_h\|_{\mathcal{T}_h}$}&\multicolumn{2}{c|}{$\|\bm\sigma-\bm\sigma_h\|_{\mathcal{T}_h}$}&\multicolumn{2}{c}{$\|\omega-\omega_h\|_{\mathcal{T}_h}$}\\
		\cline{4-11}
		& & &Error&Rate&Error&Rate&Error&Rate&Error&Rate\\
		\hline
		\multirow{7}{*}{1}
		&$2\times2$	&	12	&	1.4101E-02	&	-	&	1.1267E-01	&	-	&	1.1518E-02	&	-	&	4.2883E-02	&	-	\\
		&$4\times4$	&	24	&	5.7887E-03	&	1.28 	&	9.5251E-02	&	0.24 	&	6.1928E-03	&	0.90 	&	1.9591E-02	&	1.13 	\\
		&$8\times8$	&	32	&	1.8465E-03	&	1.65 	&	5.3090E-02	&	0.84 	&	3.4310E-03	&	0.85 	&	5.7566E-03	&	1.77 	\\
		&$16\times16$	&	41	&	4.9623E-04	&	1.90 	&	2.7260E-02	&	0.96 	&	1.7645E-03	&	0.96 	&	1.5051E-03	&	1.94 	\\
		&$32\times32$	&	43	&	1.2645E-04	&	1.97 	&	1.3721E-02	&	0.99 	&	8.8843E-04	&	0.99 	&	3.8070E-04	&	1.98 	\\
		&$64\times64$	&	44	&	3.1768E-05	&	1.99 	&	6.8721E-03	&	1.00 	&	4.4499E-04	&	1.00 	&	9.5459E-05	&	2.00 	\\
		&$128\times128$	&	42	&	7.9517E-06	&	2.00 	&	3.4375E-03	&	1.00 	&	2.2259E-04	&	1.00 	&	2.3882E-05	&	2.00 	\\			
		\Xhline{1pt}
		
		\multirow{7}{*}{2}&$2\times2$	&	39	&	6.3940E-03	&	-	&	7.7201E-02	&	-	&	6.5495E-03	&	-	&	1.3415E-02	&	-	\\
		&$4\times4$	&	56	&	1.1447E-03	&	2.48 	&	2.6054E-02	&	1.57 	&	2.5770E-03	&	1.35 	&	2.7421E-03	&	2.29 	\\
		&$8\times8$	&	71	&	1.8700E-04	&	2.61 	&	7.1550E-03	&	1.86 	&	7.2715E-04	&	1.83 	&	3.8361E-04	&	2.84 	\\
		&$16\times16$	&	65	&	3.9848E-05	&	2.23 	&	1.8381E-03	&	1.96 	&	1.9269E-04	&	1.92 	&	4.9976E-05	&	2.94 	\\
		&$32\times32$	&	64	&	7.7231E-06	&	2.37 	&	4.6384E-04	&	1.99 	&	5.2196E-05	&	1.88 	&	6.3649E-06	&	2.97 	\\
		&$64\times64$	&	64	&	1.2083E-06	&	2.68 	&	1.1638E-04	&	1.99 	&	1.3966E-05	&	1.90 	&	8.0278E-07	&	2.99 	\\
		&$128\times128$	&	58	&	1.6762E-07	&	2.85 	&	2.9141E-05	&	2.00 	&	3.6333E-06	&	1.94 	&	1.0079E-07	&	2.99 	\\
		
		\Xhline{1pt}
		\multirow{7}{*}{3}&$2\times2$	&	69	&	1.9516E-03	&	-	&	3.6223E-02	&	-	&	3.1544E-03	&	-	&	5.6861E-03	&	-	\\
		&$4\times4$	&	104	&	2.2076E-04	&	3.14 	&	5.9261E-03	&	2.61 	&	6.0651E-04	&	2.38 	&	4.1451E-04	&	3.78 	\\
		&$8\times8$	&	112	&	1.8750E-05	&	3.56 	&	7.9693E-04	&	2.89 	&	8.7173E-05	&	2.80 	&	2.7813E-05	&	3.90 	\\
		&$16\times16$	&	114	&	1.2109E-06	&	3.95 	&	1.0241E-04	&	2.96 	&	1.1194E-05	&	2.96 	&	1.7942E-06	&	3.95 	\\
		&$32\times32$	&	106	&	7.6609E-08	&	3.98 	&	1.2919E-05	&	2.99 	&	1.4082E-06	&	2.99 	&	1.1350E-07	&	3.98 	\\
		&$64\times64$	&	123	&	4.8394E-09	&	3.98 	&	1.6197E-06	&	3.00 	&	1.7639E-07	&	3.00 	&	7.1268E-09	&	3.99 	\\
		&$128\times128$	&	143	&	3.0459E-10	&	3.99 	&	2.0268E-07	&	3.00 	&	2.2066E-08	&	3.00 	&	4.4628E-10	&	4.00 	\\

		\Xhline{1pt}
\end{tabular}}
\caption{Results for $t=1$ on triangle meshes}
\label{t=1:table1}
\end{table}

\begin{table}[H]\tiny
\centering
\resizebox{\textwidth}{!}{
	\begin{tabular}{c|c|c|c|c|c|c|c|c|c|c}
		\Xhline{1pt}
		\multirow{2}{*}{$k$}&\multirow{2}{*}{Mesh}&\multirow{2}{*}{Iter}&\multicolumn{2}{c|}{$\|\bm\theta-\bm\theta_h\|_{\mathcal{T}_h}$}&\multicolumn{2}{c|}{$t\|\bm\gamma-\bm\gamma_h\|_{\mathcal{T}_h}$}&\multicolumn{2}{c|}{$\|\bm\sigma-\bm\sigma_h\|_{\mathcal{T}_h}$}&\multicolumn{2}{c}{$\|\omega-\omega_h\|_{\mathcal{T}_h}$}\\
		\cline{4-11}
		& & &Error&Rate&Error&Rate&Error&Rate&Error&Rate\\
		\hline
		\multirow{7}{*}{1}
		&$2\times2$	&	12	&	1.3699E-02	&	-	&	1.1321E-02	&	-	&	1.1401E-02	&	-	&	3.2301E-03	&	-	\\
		&$4\times4$	&	25	&	5.7317E-03	&	1.26 	&	8.7692E-03	&	0.37 	&	6.2085E-03	&	0.88 	&	9.6966E-04	&	1.74 	\\
		&$8\times8$	&	37	&	1.8425E-03	&	1.64 	&	5.1815E-03	&	0.76 	&	3.4492E-03	&	0.85 	&	2.6533E-04	&	1.87 	\\
		&$16\times16$	&	46	&	4.7357E-04	&	1.96 	&	2.7151E-03	&	0.93 	&	1.7605E-03	&	0.97 	&	6.4744E-05	&	2.03 	\\
		&$32\times32$	&	47	&	1.2240E-04	&	1.95 	&	1.3715E-03	&	0.99 	&	8.8750E-04	&	0.99 	&	1.6085E-05	&	2.01 	\\
		&$64\times64$	&	43	&	3.1427E-05	&	1.96 	&	6.8718E-04	&	1.00 	&	4.4486E-04	&	1.00 	&	4.0289E-06	&	2.00 	\\
		&$128\times128$	&	42	&	7.9288E-06	&	1.99 	&	3.4375E-04	&	1.00 	&	2.2257E-04	&	1.00 	&	1.0082E-06	&	2.00 	\\
		
		\Xhline{1pt}
		
		\multirow{8}{*}{2}&$2\times2$	&	23	&	6.0449E-03	&	-	&	7.3006E-03	&	-	&	6.3656E-03	&	-	&	9.8520E-04	&	-	\\
		&$4\times4$	&	54	&	8.9741E-04	&	2.75 	&	2.4783E-03	&	1.56 	&	2.4921E-03	&	1.35 	&	1.6980E-04	&	2.54 	\\
		&$8\times8$	&	79	&	1.3264E-04	&	2.76 	&	6.9056E-04	&	1.84 	&	7.2188E-04	&	1.79 	&	2.6108E-05	&	2.70 	\\
		&$16\times16$	&	80	&	2.1277E-05	&	2.64 	&	1.8113E-04	&	1.93 	&	1.8802E-04	&	1.94 	&	3.4754E-06	&	2.91 	\\
		&$32\times32$	&	72	&	4.6672E-06	&	2.19 	&	4.6613E-05	&	1.96 	&	4.8548E-05	&	1.95 	&	4.4193E-07	&	2.98 	\\
		&$64\times64$	&	63	&	9.8039E-07	&	2.25 	&	1.1742E-05	&	1.99 	&	1.3160E-05	&	1.88 	&	5.5482E-08	&	2.99 	\\
		&$128\times128$	&	68	&	1.5742E-07	&	2.64 	&	2.9265E-06	&	2.00 	&	3.5492E-06	&	1.89 	&	6.9422E-09	&	3.00 	\\
		
		\Xhline{1pt}
		\multirow{7}{*}{3}&$2\times2$	&	38	&	1.6684E-03	&	-	&	3.5916E-03	&	-	&	3.0420E-03	&	-	&	3.9487E-04	&	-	\\
		&$4\times4$	&	79	&	1.5251E-04	&	3.45 	&	5.7646E-04	&	2.64 	&	5.8831E-04	&	2.37 	&	5.3478E-05	&	2.88 	\\
		&$8\times8$	&	117	&	1.1239E-05	&	3.76 	&	7.7665E-05	&	2.89 	&	8.3829E-05	&	2.81 	&	3.8806E-06	&	3.78 	\\
		&$16\times16$	&	118	&	7.8941E-07	&	3.83 	&	1.0125E-05	&	2.94 	&	1.0817E-05	&	2.95 	&	2.5079E-07	&	3.95 	\\
		&$32\times32$	&	111	&	6.1963E-08	&	3.67 	&	1.2886E-06	&	2.97 	&	1.3788E-06	&	2.97 	&	1.5797E-08	&	3.99 	\\
		&$64\times64$	&	131	&	4.5209E-09	&	3.78 	&	1.6191E-07	&	2.99 	&	1.7501E-07	&	2.98 	&	9.8904E-10	&	4.00 	\\
		&$128\times128$	&	142	&	2.9906E-10	&	3.92 	&	2.0266E-08	&	3.00 	&	2.2017E-08	&	2.99 	&	6.1839E-11	&	4.00 	\\
		\Xhline{1pt}
\end{tabular}}
\caption{Results for $t=0.1$ on triangle meshes}
\label{t=0.1:table1}
\end{table}

\begin{table}[H]\tiny
\centering
\resizebox{\textwidth}{!}{
	\begin{tabular}{c|c|c|c|c|c|c|c|c|c|c}
		\Xhline{1pt}
		\multirow{2}{*}{$k$}&\multirow{2}{*}{Mesh}&\multirow{2}{*}{Iter}&\multicolumn{2}{c|}{$\|\bm\theta-\bm\theta_h\|_{\mathcal{T}_h}$}&\multicolumn{2}{c|}{$t\|\bm\gamma-\bm\gamma_h\|_{\mathcal{T}_h}$}&\multicolumn{2}{c|}{$\|\bm\sigma-\bm\sigma_h\|_{\mathcal{T}_h}$}&\multicolumn{2}{c}{$\|\omega-\omega_h\|_{\mathcal{T}_h}$}\\
		\cline{4-11}
		& & &Error&Rate&Error&Rate&Error&Rate&Error&Rate\\
		\hline
		\multirow{7}{*}{1}
		&$2\times2$	&	10	&	1.3671E-02	&	-	&	1.1367E-03	&	-	&	1.1387E-02	&	-	&	2.8682E-03	&	-	\\
		&$4\times4$	&	17	&	5.6979E-03	&	1.26 	&	7.9353E-04	&	0.52 	&	6.2099E-03	&	0.87 	&	8.5458E-04	&	1.75 	\\
		&$8\times8$	&	24	&	1.9548E-03	&	1.54 	&	4.2548E-04	&	0.90 	&	3.4801E-03	&	0.84 	&	2.4544E-04	&	1.80 	\\
		&$16\times16$	&	25	&	5.4575E-04	&	1.84 	&	2.2217E-04	&	0.94 	&	1.7659E-03	&	0.98 	&	6.3364E-05	&	1.95 	\\
		&$32\times32$	&	31	&	1.3826E-04	&	1.98 	&	1.2302E-04	&	0.85 	&	8.8334E-04	&	1.00 	&	1.5686E-05	&	2.01 	\\
		&$64\times64$	&	39	&	3.2883E-05	&	2.07 	&	6.6637E-05	&	0.88 	&	4.4217E-04	&	1.00 	&	3.7339E-06	&	2.07 	\\
		&$128\times128$	&	44	&	7.6144E-06	&	2.11 	&	3.4206E-05	&	0.96 	&	2.2172E-04	&	1.00 	&	8.7447E-07	&	2.09 	\\
		
		\Xhline{1pt}
		
		\multirow{8}{*}{2}
		&$2\times2$	&	17	&	6.1159E-03	&	-	&	7.1643E-04	&	-	&	6.3847E-03	&	-	&	9.5754E-04	&	-	\\
		&$4\times4$	&	30	&	9.1901E-04	&	2.73 	&	2.6662E-04	&	1.43 	&	2.4867E-03	&	1.36 	&	1.7457E-04	&	2.46 	\\
		&$8\times8$	&	37	&	1.2542E-04	&	2.87 	&	9.5883E-05	&	1.48 	&	7.1873E-04	&	1.79 	&	2.6780E-05	&	2.70 	\\
		&$16\times16$	&	42	&	1.6965E-05	&	2.89 	&	2.9814E-05	&	1.69 	&	1.8900E-04	&	1.93 	&	3.5208E-06	&	2.93 	\\
		&$32\times32$	&	50	&	2.2717E-06	&	2.90 	&	7.5139E-06	&	1.99 	&	4.8174E-05	&	1.97 	&	4.4598E-07	&	2.98 	\\
		&$64\times64$	&	65	&	3.0847E-07	&	2.88 	&	1.8125E-06	&	2.05 	&	1.2117E-05	&	1.99 	&	5.6030E-08	&	2.99 	\\
		&$128\times128$	&	78	&	4.7096E-08	&	2.71 	&	5.2227E-07	&	1.80 	&	3.0327E-06	&	2.00 	&	7.0184E-09	&	3.00 	\\

		\Xhline{1pt}
		\multirow{7}{*}{3}
		&$2\times2$	&	23	&	1.6804E-03	&	-	&	3.7706E-04	&	-	&	3.0238E-03	&	-	&	3.9919E-04	&	-	\\
		&$4\times4$	&	43	&	1.4371E-04	&	3.55 	&	7.7838E-05	&	2.28 	&	5.8361E-04	&	2.37 	&	5.4096E-05	&	2.88 	\\
		&$8\times8$	&	55	&	1.0479E-05	&	3.78 	&	1.3820E-05	&	2.49 	&	8.3668E-05	&	2.80 	&	3.9387E-06	&	3.78 	\\
		&$16\times16$	&	67	&	6.8187E-07	&	3.94 	&	1.9520E-06	&	2.82 	&	1.0824E-05	&	2.95 	&	2.5486E-07	&	3.95 	\\
		&$32\times32$	&	89	&	4.3044E-08	&	3.99 	&	2.0906E-07	&	3.22 	&	1.3636E-06	&	2.99 	&	1.6061E-08	&	3.99 	\\
		&$64\times64$	&	125	&	2.7120E-09	&	3.99 	&	2.0610E-08	&	3.34 	&	1.7047E-07	&	3.00 	&	1.0058E-09	&	4.00 	\\
		&$128\times128$	&	181	&	1.8206E-10	&	3.90 	&	2.5070E-09	&	3.04 	&	2.1290E-08	&	3.00 	&	6.2890E-11	&	4.00 	\\
		
		\Xhline{1pt}
\end{tabular}}
\caption{Results for $t=0.01$ on triangle meshes}
\label{t=0.01:table1}
\end{table}

\begin{figure}[H]
\centering
\subfigure[$4\times4$]{
	\includegraphics[width=0.45\linewidth]{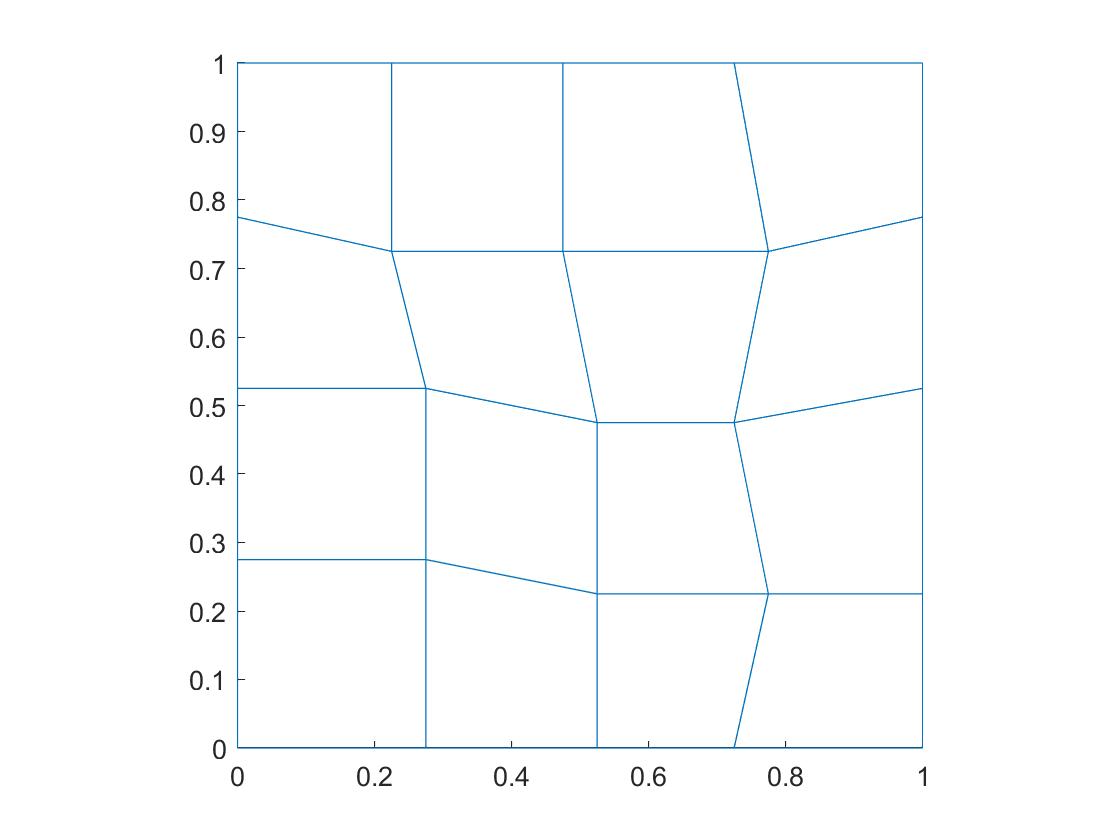} 
}\quad\subfigure[$8\times8$]{
	\includegraphics[width=0.45\linewidth]{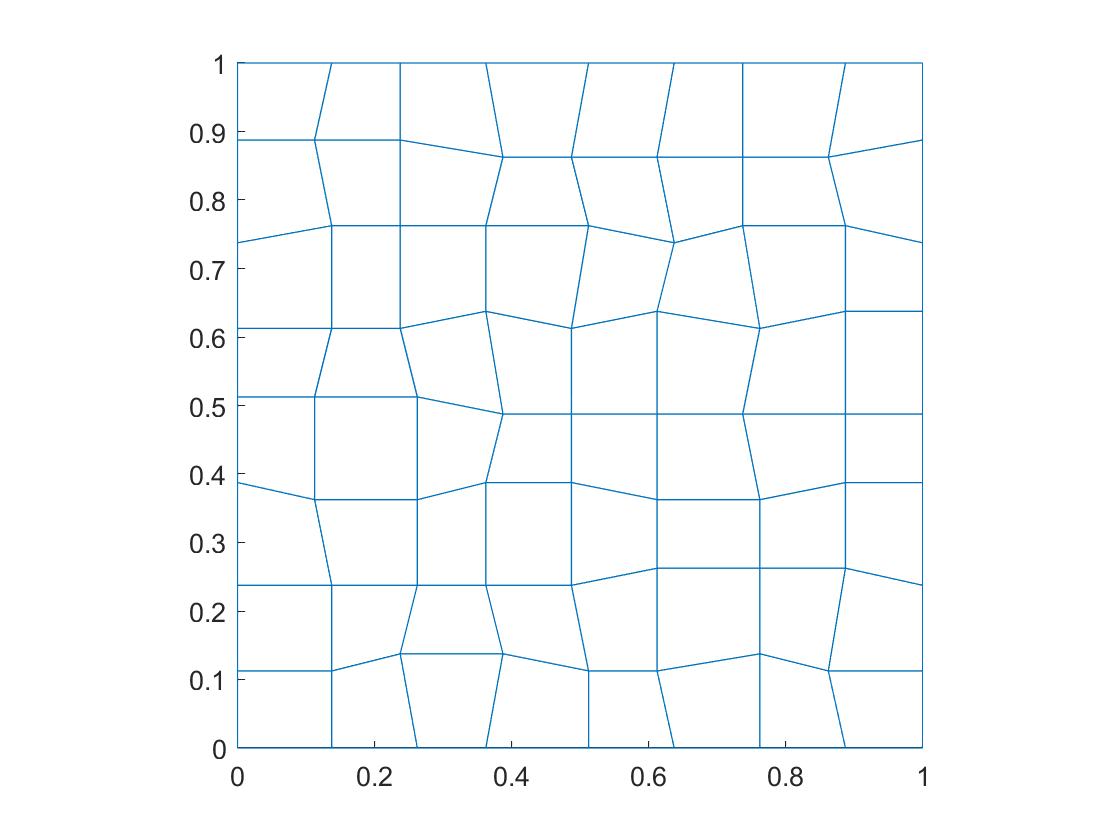}
}
\caption{Quadrangle meshes for square plate}
\end{figure}

\begin{table}[H]\tiny
\centering
\resizebox{\textwidth}{!}{
	\begin{tabular}{c|c|c|c|c|c|c|c|c|c|c}
		\Xhline{1pt}
		\multirow{2}{*}{$k$}&\multirow{2}{*}{Mesh}&\multirow{2}{*}{Iter}&\multicolumn{2}{c|}{$\|\bm\theta-\bm\theta_h\|_{\mathcal{T}_h}$}&\multicolumn{2}{c|}{$t\|\bm\gamma-\bm\gamma_h\|_{\mathcal{T}_h}$}&\multicolumn{2}{c|}{$\|\bm\sigma-\bm\sigma_h\|_{\mathcal{T}_h}$}&\multicolumn{2}{c}{$\|\omega-\omega_h\|_{\mathcal{T}_h}$}\\
		\cline{4-11}
		& & &Error&Rate&Error&Rate&Error&Rate&Error&Rate\\
		\hline
		\multirow{7}{*}{1}
		&$2\times2$	&	28	&	1.7568E-02	&	-	&	1.1626E-01	&	-	&	1.1571E-02	&	-	&	4.3815E-02	&	-	\\
		&$4\times4$	&	44	&	9.9463E-03	&	0.82 	&	1.1512E-01	&	0.01 	&	8.4114E-03	&	0.46 	&	1.3570E-02	&	1.69 	\\
		&$8\times8$	&	54	&	3.5871E-03	&	1.47 	&	5.7097E-02	&	1.01 	&	4.5478E-03	&	0.89 	&	3.3871E-03	&	2.00 	\\
		&$16\times16$	&	56	&	1.0446E-03	&	1.78 	&	2.7791E-02	&	1.04 	&	2.1731E-03	&	1.07 	&	1.0606E-03	&	1.68 	\\
		&$32\times32$	&	57	&	2.6797E-04	&	1.96 	&	1.3466E-02	&	1.05 	&	1.0725E-03	&	1.02 	&	2.9123E-04	&	1.86 	\\
		&$64\times64$	&	55	&	6.8216E-05	&	1.97 	&	6.7159E-03	&	1.00 	&	5.3538E-04	&	1.00 	&	7.5635E-05	&	1.95 	\\
		&$128\times128$	&	56	&	1.7062E-05	&	2.00 	&	3.3434E-03	&	1.01 	&	2.6710E-04	&	1.00 	&	1.9018E-05	&	1.99 	\\
		
		\Xhline{1pt}
		
		\multirow{7}{*}{2}
		&$2\times2$	&	57	&	1.1418E-02	&	-	&	1.1908E-01	&	-	&	1.0261E-02	&	-	&	2.1634E-02	&	-	\\
		&$4\times4$	&	118	&	1.6670E-03	&	2.78 	&	4.0647E-02	&	1.55 	&	3.4348E-03	&	1.58 	&	3.8990E-03	&	2.47 	\\
		&$8\times8$	&	140	&	1.8198E-04	&	3.20 	&	1.1997E-02	&	1.76 	&	9.4412E-04	&	1.86 	&	6.6742E-04	&	2.55 	\\
		&$16\times16$	&	143	&	2.1513E-05	&	3.08 	&	3.0604E-03	&	1.97 	&	2.5405E-04	&	1.89 	&	8.0739E-05	&	3.05 	\\
		&$32\times32$	&	147	&	2.4485E-06	&	3.14 	&	7.4485E-04	&	2.04 	&	6.4357E-05	&	1.98 	&	9.6052E-06	&	3.07 	\\
		&$64\times64$	&	141	&	3.0252E-07	&	3.02 	&	1.8351E-04	&	2.02 	&	1.6275E-05	&	1.98 	&	1.1686E-06	&	3.04 	\\
		&$128\times128$	&	144	&	3.7415E-08	&	3.02 	&	4.5194E-05	&	2.02 	&	4.0643E-06	&	2.00 	&	1.4270E-07	&	3.03 	\\
		
		\Xhline{1pt}
		\multirow{7}{*}{3}
		&$2\times2$	&	118	&	3.7853E-03	&	-	&	6.1640E-02	&	-	&	4.8347E-03	&	-	&	1.0738E-02	&	-	\\
		&$4\times4$	&	182	&	3.6357E-04	&	3.38 	&	1.2094E-02	&	2.35 	&	1.0255E-03	&	2.24 	&	9.8364E-04	&	3.45 	\\
		&$8\times8$	&	203	&	2.2050E-05	&	4.04 	&	1.6706E-03	&	2.86 	&	1.5662E-04	&	2.71 	&	6.7423E-05	&	3.87 	\\
		&$16\times16$	&	210	&	1.2673E-06	&	4.12 	&	2.2697E-04	&	2.88 	&	1.9755E-05	&	2.99 	&	4.6126E-06	&	3.87 	\\
		&$32\times32$	&	209	&	7.7548E-08	&	4.03 	&	2.7537E-05	&	3.04 	&	2.5016E-06	&	2.98 	&	2.8037E-07	&	4.04 	\\
		&$64\times64$	&	213	&	4.9043E-09	&	3.98 	&	3.4114E-06	&	3.01 	&	3.1685E-07	&	2.98 	&	1.7389E-08	&	4.01 	\\
		&$128\times128$	&	232	&	3.0577E-10	&	4.00 	&	4.2120E-07	&	3.02 	&	3.9582E-08	&	3.00 	&	1.0768E-09	&	4.01 	\\
		
		\Xhline{1pt}
\end{tabular}}
\caption{Results for $t=1$ on Quadrangle meshes}
\label{t=1:table2}
\end{table}

\begin{table}[H]\tiny
\centering
\resizebox{\textwidth}{!}{
	\begin{tabular}{c|c|c|c|c|c|c|c|c|c|c}
		\Xhline{1pt}
		\multirow{2}{*}{$k$}&\multirow{2}{*}{Mesh}&\multirow{2}{*}{Iter}&\multicolumn{2}{c|}{$\|\bm\theta-\bm\theta_h\|_{\mathcal{T}_h}$}&\multicolumn{2}{c|}{$t\|\bm\gamma-\bm\gamma_h\|_{\mathcal{T}_h}$}&\multicolumn{2}{c|}{$\|\bm\sigma-\bm\sigma_h\|_{\mathcal{T}_h}$}&\multicolumn{2}{c}{$\|\omega-\omega_h\|_{\mathcal{T}_h}$}\\
		\cline{4-11}
		& & &Error&Rate&Error&Rate&Error&Rate&Error&Rate\\
		\hline
		\multirow{7}{*}{1}
		&$2\times2$	&	31	&	1.7625E-02	&	-	&	1.1548E-02	&	-	&	1.1585E-02	&	-	&	3.5954E-03	&	-	\\
		&$4\times4$	&	38	&	9.0459E-03	&	0.96 	&	1.0046E-02	&	0.20 	&	8.3604E-03	&	0.47 	&	1.3964E-03	&	1.36 	\\
		&$8\times8$	&	43	&	3.0751E-03	&	1.56 	&	5.4947E-03	&	0.87 	&	4.4547E-03	&	0.91 	&	5.5968E-04	&	1.32 	\\
		&$16\times16$	&	51	&	8.9104E-04	&	1.79 	&	2.7645E-03	&	0.99 	&	2.1431E-03	&	1.06 	&	1.7253E-04	&	1.70 	\\
		&$32\times32$	&	56	&	2.4444E-04	&	1.87 	&	1.3459E-03	&	1.04 	&	1.0686E-03	&	1.00 	&	4.6105E-05	&	1.90 	\\
		&$64\times64$	&	58	&	6.5893E-05	&	1.89 	&	6.7150E-04	&	1.00 	&	5.3495E-04	&	1.00 	&	1.1980E-05	&	1.94 	\\
		&$128\times128$	&	59	&	1.6893E-05	&	1.96 	&	3.3431E-04	&	1.01 	&	2.6705E-04	&	1.00 	&	3.0295E-06	&	1.98 	\\
		
		\Xhline{1pt}
		
		\multirow{7}{*}{2}
		&$2\times2$	&	38	&	8.4971E-03	&	-	&	1.1876E-02	&	-	&	9.1063E-03	&	-	&	1.6458E-03	&	-	\\
		&$4\times4$	&	57	&	1.2955E-03	&	2.71 	&	3.7942E-03	&	1.65 	&	3.3597E-03	&	1.44 	&	5.4627E-04	&	1.59 	\\
		&$8\times8$	&	85	&	1.5445E-04	&	3.07 	&	1.1749E-03	&	1.69 	&	9.3936E-04	&	1.84 	&	1.1299E-04	&	2.27 	\\
		&$16\times16$	&	115	&	1.9857E-05	&	2.96 	&	3.0348E-04	&	1.95 	&	2.5358E-04	&	1.89 	&	1.4638E-05	&	2.95 	\\
		&$32\times32$	&	137	&	2.4087E-06	&	3.04 	&	7.4279E-05	&	2.03 	&	6.4328E-05	&	1.98 	&	1.9069E-06	&	2.94 	\\
		&$64\times64$	&	145	&	3.0112E-07	&	3.00 	&	1.8331E-05	&	2.02 	&	1.6273E-05	&	1.98 	&	2.4228E-07	&	2.98 	\\
		&$128\times128$	&	149	&	3.7369E-08	&	3.01 	&	4.5180E-06	&	2.02 	&	4.0641E-06	&	2.00 	&	3.0220E-08	&	3.00 	\\
		
		\Xhline{1pt}
		\multirow{7}{*}{3}
		&$2\times2$	&	57	&	3.4741E-03	&	-	&	6.0634E-03	&	-	&	4.6820E-03	&	-	&	1.7394E-03	&	-	\\
		&$4\times4$	&	84	&	2.7585E-04	&	3.65 	&	1.1686E-03	&	2.38 	&	1.0059E-03	&	2.22 	&	2.3709E-04	&	2.88 	\\
		&$8\times8$	&	126	&	2.0865E-05	&	3.72 	&	1.6266E-04	&	2.84 	&	1.5640E-04	&	2.69 	&	1.7804E-05	&	3.74 	\\
		&$16\times16$	&	169	&	1.2457E-06	&	4.07 	&	2.2568E-05	&	2.85 	&	1.9756E-05	&	2.98 	&	1.2253E-06	&	3.86 	\\
		&$32\times32$	&	198	&	7.7237E-08	&	4.01 	&	2.7541E-06	&	3.03 	&	2.5016E-06	&	2.98 	&	7.7321E-08	&	3.99 	\\
		&$64\times64$	&	216	&	4.8977E-09	&	3.98 	&	3.4121E-07	&	3.01 	&	3.1684E-07	&	2.98 	&	4.9332E-09	&	3.97 	\\
		&$128\times128$	&	232	&	3.0566E-10	&	4.00 	&	4.2122E-08	&	3.02 	&	3.9581E-08	&	3.00 	&	3.0675E-10	&	4.01 	\\
		
		\Xhline{1pt}
\end{tabular}}
\caption{Results for $t=0.1$ on Quadrangle meshes}
\label{t=0.1:table2}
\end{table}

\begin{table}[H]\tiny
\centering
\resizebox{\textwidth}{!}{
	\begin{tabular}{c|c|c|c|c|c|c|c|c|c|c}
		\Xhline{1pt}
		\multirow{2}{*}{$k$}&\multirow{2}{*}{Mesh}&\multirow{2}{*}{Iter}&\multicolumn{2}{c|}{$\|\bm\theta-\bm\theta_h\|_{\mathcal{T}_h}$}&\multicolumn{2}{c|}{$t\|\bm\gamma-\bm\gamma_h\|_{\mathcal{T}_h}$}&\multicolumn{2}{c|}{$\|\bm\sigma-\bm\sigma_h\|_{\mathcal{T}_h}$}&\multicolumn{2}{c}{$\|\omega-\omega_h\|_{\mathcal{T}_h}$}\\
		\cline{4-11}
		& & &Error&Rate&Error&Rate&Error&Rate&Error&Rate\\
		\hline
		\multirow{7}{*}{1}
		&$2\times2$	&	30	&	1.7645E-02	&	-	&	1.1541E-03	&	-	&	1.1583E-02	&	-	&	3.2521E-03	&	-	\\
		&$4\times4$	&	46	&	9.0791E-03	&	0.96 	&	9.7517E-04	&	0.24 	&	8.4549E-03	&	0.45 	&	1.3157E-03	&	1.31 	\\
		&$8\times8$	&	52	&	2.9914E-03	&	1.60 	&	5.5462E-04	&	0.81 	&	4.4634E-03	&	0.92 	&	5.4501E-04	&	1.27 	\\
		&$16\times16$	&	50	&	8.2085E-04	&	1.87 	&	2.9340E-04	&	0.92 	&	2.1414E-03	&	1.06 	&	1.6688E-04	&	1.71 	\\
		&$32\times32$	&	44	&	2.0982E-04	&	1.97 	&	1.5028E-04	&	0.97 	&	1.0665E-03	&	1.01 	&	4.3425E-05	&	1.94 	\\
		&$64\times64$	&	38	&	5.3644E-05	&	1.97 	&	7.5367E-05	&	1.00 	&	5.3397E-04	&	1.00 	&	1.1075E-05	&	1.97 	\\
		&$128\times128$	&	46	&	1.3817E-05	&	1.96 	&	3.6799E-05	&	1.03 	&	2.6662E-04	&	1.00 	&	2.8272E-06	&	1.97 	\\
		
		\Xhline{1pt}
		
		\multirow{8}{*}{2}
		&$2\times2$	&	42	&	8.4549E-03	&	-	&	1.1931E-03	&	-	&	9.0880E-03	&	-	&	1.5197E-03	&	-	\\
		&$4\times4$	&	61	&	1.3048E-03	&	2.70 	&	3.8651E-04	&	1.63 	&	3.3541E-03	&	1.44 	&	5.6406E-04	&	1.43 	\\
		&$8\times8$	&	65	&	1.5187E-04	&	3.10 	&	1.2475E-04	&	1.63 	&	9.3966E-04	&	1.84 	&	1.1453E-04	&	2.30 	\\
		&$16\times16$	&	61	&	1.9286E-05	&	2.98 	&	4.2464E-05	&	1.55 	&	2.5366E-04	&	1.89 	&	1.4774E-05	&	2.95 	\\
		&$32\times32$	&	60	&	2.3808E-06	&	3.02 	&	1.2671E-05	&	1.74 	&	6.4339E-05	&	1.98 	&	1.9185E-06	&	2.95 	\\
		&$64\times64$	&	74	&	2.9748E-07	&	3.00 	&	2.7996E-06	&	2.18 	&	1.6271E-05	&	1.98 	&	2.4330E-07	&	2.98 	\\
		&$128\times128$	&	105	&	3.7449E-08	&	2.99 	&	5.5725E-07	&	2.33 	&	4.0631E-06	&	2.00 	&	3.0343E-08	&	3.00 	\\

		\Xhline{1pt}
		\multirow{7}{*}{3}
		&$2\times2$	&	58	&	3.5164E-03	&	-	&	6.2234E-04	&	-	&	4.6620E-03	&	-	&	1.8221E-03	&	-	\\
		&$4\times4$	&	84	&	2.3951E-04	&	3.88 	&	1.2622E-04	&	2.30 	&	1.0011E-03	&	2.22 	&	2.4146E-04	&	2.92 	\\
		&$8\times8$	&	86	&	1.9040E-05	&	3.65 	&	2.6039E-05	&	2.28 	&	1.5597E-04	&	2.68 	&	1.8017E-05	&	3.74 	\\
		&$16\times16$	&	76	&	1.2191E-06	&	3.97 	&	5.5609E-06	&	2.23 	&	1.9790E-05	&	2.98 	&	1.2398E-06	&	3.86 	\\
		&$32\times32$	&	85	&	7.7220E-08	&	3.98 	&	8.7157E-07	&	2.67 	&	2.5078E-06	&	2.98 	&	7.8083E-08	&	3.99 	\\
		&$64\times64$	&	124	&	4.8841E-09	&	3.98 	&	9.0071E-08	&	3.27 	&	3.1722E-07	&	2.98 	&	4.9789E-09	&	3.97 	\\
		&$128\times128$	&	193	&	3.0440E-10	&	4.00 	&	7.5936E-09	&	3.57 	&	3.9596E-08	&	3.00 	&	3.0955E-10	&	4.01 	\\
		
		\Xhline{1pt}
\end{tabular}}
\caption{Results for $t=0.01$ on Quadrangle meshes}
\label{t=0.01:table2}
\end{table}

\subsection{Contrast of Iterations}
We next consider the convergence rate of solving the discrete system of our method and comparing it with the method presented in \cite{Chen1}, called ``Old Method'' in \Cref{contrast1,contrast2}. The system would typically be solved iteratively in practice so it is of interest to examine the condition number and it is well known that the convergence can be slow if the condition number is large. \par
Let us briefly analyze the result displayed in \Cref{contrast1,contrast2}. We let the stopping criterion for iteration be same for two methods taking $1.0E-10$ and solve the problem introduced in \Cref{analytical solution}. As the mesh refinement doubling the number of line elements, the iteration of the old methods approximately increases twice correspondingly. It's simply shows that the iterations increase as the $h$ decreases. However the iterations of our method increase slower than the old method. This shows that our new mehtod may have a good condition number compared to the old method.
\begin{table}[H] \tiny
\centering
\resizebox{0.8\textwidth}{!}{
	\begin{tabular}{c|c|c|c|c|c|c|c}
		\Xhline{1pt}
		$k$	&	\diagbox{Method}{Mesh}			&	$2\times2$	&	$4\times4$	&	$8\times8$	&	$16\times16$	&	$32\times32$	&	$64\times64$	\\
		\hline
		\multirow{2}{*}{1}	&	Old method			&	8	&	34	&	155	&	706	&	2758	&	10800	\\
		&	New method			&	4	&	11	&	19	&	19	&	21	&	23	\\
		\hline
		\multirow{2}{*}{2}	&	Old method			&	17	&	77	&	332	&	1245	&	4822	&	18740	\\
		&	New method			&	6	&	21	&	34	&	37	&	50	&	56	\\
		\hline
		\multirow{2}{*}{3}	&	Old method			&	23	&	106	&	436	&	1686	&	7723	&	22172	\\
		&	New method			&	9	&	31	&	45	&	60	&	83	&	99	\\
		\Xhline{1pt}
\end{tabular}}
\caption{Contrast of iteration for $t=1.0E-10$}
\label{contrast1}
\end{table}
\begin{table}[H]\tiny
\centering
\resizebox{0.8\textwidth}{!}{
	\begin{tabular}{c|c|c|c|c|c|c|c}
		\Xhline{1pt}
		$k$	&	\diagbox{Method}{Mesh}			&	$2\times2$	&	$4\times4$	&	$8\times8$	&	$16\times16$	&	$32\times32$	&	$64\times64$	\\
		\hline																	
		\multirow{2}{*}{1}	&	Old method			&	8	&	27	&	77	&	138	&	251	&	480	\\
		&	New method			&	4	&	17	&	28	&	40	&	46	&	53	\\
		\hline																	
		\multirow{2}{*}{2}	&	Old method			&	12	&	207	&	93	&	175	&	326	&	618	\\
		&	New method			&	15	&	35	&	52	&	55	&	62	&	68	\\
		\hline																	
		\multirow{2}{*}{3}	&	Old method			&	125	&	172	&	100	&	184	&	336	&	620	\\
		&	New method			&	26	&	52	&	68	&	96	&	119	&	136	\\
		\Xhline{1pt}
\end{tabular}}
\caption{Contrast of iteration for $t=1$}
\label{contrast2}
\end{table}

%\appendix
%\section{An example appendix} 
%
%
%
%
\section*{Acknowledgments}
We would like to acknowledge the assistance of volunteers in putting
together this example manuscript and supplement.

\end{document}